\documentclass[11pt]{article}

\usepackage[margin=1.2in]{geometry}

\usepackage[T1]{fontenc}
\usepackage[utf8]{inputenc}
\usepackage[USenglish, UKenglish, english]{babel}
\usepackage{amssymb}
\usepackage{amsthm}
\usepackage{amsmath}
\usepackage[mathscr]{euscript}
\usepackage{enumitem}
\usepackage{graphicx}
\usepackage{MnSymbol}
 \let\mathscr\relax
\usepackage[scr]{rsfso}
\usepackage{tikz-cd}
\usepackage{tikz}
\usetikzlibrary{matrix,arrows,decorations.pathmorphing}
\usepackage{hyperref}
\usepackage{marginnote}
\usetikzlibrary{arrows.meta}

\usepackage{libertine}

\theoremstyle{definition}
\newtheorem{defin}{Definition}[section]
\theoremstyle{definition}

\theoremstyle{plain}
\newtheorem{theo}[defin]{Theorem}
\theoremstyle{plain}
\newtheorem{prop}[defin]{Proposition}
\theoremstyle{plain}
\newtheorem{lem}[defin]{Lemma}
\theoremstyle{plain}
\newtheorem{cor}[defin]{Corollary}
\theoremstyle{definition}
\newtheorem{rmk}[defin]{Remark}
\theoremstyle{definition}

\theoremstyle{definition}

\theoremstyle{plain}
\newtheorem{conj}[defin]{Conjecture}
\theoremstyle{definition}
\newtheorem{notation}[defin]{Notation}
\theoremstyle{definition}
\newtheorem{hyp}[defin]{Hypothesis}

\theoremstyle{definition}
\newtheorem{cond}[defin]{Condition}

\theoremstyle{plain}
\newtheorem{para}[defin]{Parametrisation}

\theoremstyle{definition}

\theoremstyle{definition}
\newtheorem*{defin*}{Definition}
\theoremstyle{definition}
\newtheorem*{ex*}{Example}
\theoremstyle{plain}
\newtheorem*{theo*}{Theorem}
\theoremstyle{plain}
\newtheorem*{prop*}{Proposition}
\theoremstyle{plain}
\newtheorem*{lem*}{Lemma}
\theoremstyle{plain}
\newtheorem*{cor*}{Corollary}
\theoremstyle{definition}
\newtheorem*{rmk*}{Remark}
\theoremstyle{definition}
\newtheorem*{exe*}{Exercise}

\theoremstyle{plain}
\newtheorem{theoA}{Theorem}

\theoremstyle{plain}
\newtheorem{conjA}[theoA]{Conjecture}

\theoremstyle{plain}

\theoremstyle{plain}
\newtheorem{paraA}[theoA]{Parametrisation}

\theoremstyle{plain}

\numberwithin{equation}{section}
\usepackage{parskip}

\makeatletter
\def\thm@space@setup{%
  \thm@preskip=\parskip \thm@postskip=0pt
}
\makeatother

\setlist[enumerate]{label=(\roman*)}
\def\R{{\mathbf{R}}} %Twisted/parabolic induction

\def\Bl{{\rm Bl}}

\def\bl{{\rm bl}}

\def\irr{{\rm Irr}}

\def\abirr{{\rm AbIrr}}

\def\ab{{\rm Ab}}

\def\ker{{\rm Ker}}

\def\syl{{\rm Syl}}

\def\aut{{\rm Aut}}

\def\ind{{\rm Ind}}

\def\res{{\rm Res}}

\def\n{{\mathbf{N}}} %normalizer

\def\c{{\mathbf{C}}} %centralizer
 
\def\z{{\mathbf{Z}}} %center

\def\O{{\mathbf{O}}} %radical

\def\E{{\mathcal{E}}} %Lusztig series

\def\d{{\mathbb{D}}} %chain

\def\e{{\mathbb{E}}} %chain

\def\C{{\mathcal{C}}} %couples chain-character

\def\k{{\mathbf{k}}}

\def\G{{\mathbf{G}}} %algebraic group

\def\H{{\mathbf{H}}} %algebraic subgroup

\def\K{{\mathbf{K}}} %algebraic subgroup

\def\L{{\mathbf{L}}} %Levi

\def\M{{\mathbf{M}}} %Levi

\def\T{{\mathbf{T}}} %Torus

\def\S{{\mathbf{S}}} %Torus

\def\B{{\mathbf{B}}} %Borel

\def\P{{\mathbf{P}}} %Parabolic

\def\Q{{\mathbf{Q}}} %Parabolic

\def\CL{{\mathcal{L}}} %couples (chain of Levis)-character

\DeclareMathOperator{\BL}{\mathcal{B}\BLkern \mathcal{L}}
\newcommand{\BLkern}{%
  \mkern-4.0mu
  \mathchoice{}{}{\mkern0.2mu}{\mkern0.5mu}%
}
\DeclareMathOperator{\CP}{\mathcal{C}\CPkern \mathcal{P}}
\newcommand{\CPkern}{%
  \mkern-1.0mu
  \mathchoice{}{}{\mkern0.2mu}{\mkern0.5mu}%
}
\makeatletter
\newcommand{\uset}[3][0ex]{%
  \mathrel{\mathop{#3}\limits_{
    \vbox to#1{\kern-7\ex@
    \hbox{$\scriptstyle#2$}\vss}}}}
\makeatother

\newcommand{\wt}[1]{\widetilde{#1}} 

\newcommand{\wh}[1]{\widehat{#1}}

\newcommand{\ws}[1][1.5]{
  \mathrel{\scalebox{#1}[1]{$\sim$}}
}

\newcommand{\iso}[1]{\ws_{#1}}

\usepackage{sectsty}
\allsectionsfont{\centering}

\usepackage{xparse,etoolbox}

\newcommand{\blocktheorem}[1]{%
  \csletcs{old#1}{#1}% Store \begin
  \csletcs{endold#1}{end#1}% Store \end
  \RenewDocumentEnvironment{#1}{o}
    {\par\addvspace{1.5ex}
     \noindent\begin{minipage}{\textwidth}
     \IfNoValueTF{##1}
       {\csuse{old#1}}
       {\csuse{old#1}[##1]}}
    {\csuse{endold#1}
     \end{minipage}
     \par\addvspace{1.5ex}}
}

\blocktheorem{theo}
\blocktheorem{conj}
\blocktheorem{para}
\blocktheorem{theoA}
\blocktheorem{conjA}
\blocktheorem{paraA}

\makeatletter
\def\blfootnote{\gdef\@thefnmark{}\@footnotetext}
\makeatother

\title{
{\huge\bf Counting conjectures and $e$-local structures in finite reductive groups}\\
\author{\Large Damiano Rossi}
\date{}
\blfootnote{\emph{$2010$ Mathematical Subject Classification:} $20$C$20$ ($20$C$33$, $20$G$40$).
\\
\emph{Key words and phrases:} Dade's Projective Conjecture, finite reductive groups, generalised Harish-Chandra theory.
\\
The content of this paper is part of the author's doctoral thesis written in the framework of the research training group \textit{GRK2240: Algebro-geometric Methods in Algebra, Arithmetic and Topology} funded by the DFG. This work is also supported by the EPSRC grant EP/T$004592/1$. The author would like to thank Britta Sp\"ath for proposing a fascinating research topic and for providing countless suggestions, Marc Cabanes for many insightful comments on the block theory of finite reductive groups, Gunter Malle for a thorough reading of an earlier version of this paper and for instructive discussions on certain aspects of generalised Harish-Chandra theory and Radha Kessar for providing helpful comments and for suggesting Proposition \ref{prop:Transitivity in good series}. This project originates from ideas introduced by Michel Broué, Paul Fong and Bhama Srinivasan in an attempt to solve Dade's Projective Conjecture for unipotent blocks. The author is thankful to Michel Broué and Jean Michel for pointing out some inaccuracies in an earlier version of this paper.
}
}

\begin{document}

\renewcommand{\thetheoA}{\Alph{theoA}}

\renewcommand{\thecorA}{\Alph{corA}}

\selectlanguage{english}

\maketitle

\begin{abstract}
We prove new results in generalised Harish-Chandra theory providing a description of the so-called Brauer--Lusztig blocks in terms of information encoded in the $\ell$-adic cohomology of Deligne--Lusztig varieties. Then, we propose new conjectures for finite reductive groups by considering geometric analogues of the $\ell$-local structures that lie at the heart of the local-global counting conjectures. For large primes, our conjectures coincide with the counting conjectures thanks to a connection established by Broué, Fong and Srinivasan between $\ell$-structures and their geometric counterpart. Finally, using the description of Brauer--Lusztig blocks mentioned above, we reduce our conjectures to the verification of Clifford theoretic properties expected from certain parametrisations of generalised Harish-Chandra series.
\end{abstract}

\tableofcontents

\section{Introduction}

%State $e$-H-C bijection (Condition) and deduce iDade (i.e. CTC) by using the description of Brauer--Lusztig blocks.

%In a subsequent paper \cite{Ros-Criteria} we use Main theorem to reduce the verification to certain conditions similar to those used to prove the inductive (Alperin-)McKay condition as in \cite{Mal-Spa16}

Over the past few decades the research in representation theory of finite groups has been driven by the pursuit of an explanation of the so-called local-global principle. This states that for each prime number $\ell$ dividing the order
of a finite group $G$, the $\ell$-modular representation theory of $G$ is largely determined by the $\ell$-local structure
of the group $G$. The local-global principle is supported by numerous conjectural evidences including the
McKay Conjecture, the Alperin–McKay Conjecture, Alperin’s Weight Conjecture and Brauer’s Height Zero Conjecture among others.

In the 1990s, extending a connection made by Kn\"orr and Robinson between local-global conjectures and the Brown complex associated to chains of $\ell$-subgroups, Dade introduced a new conjecture known as Dade's Projective Conjecture. This provides a unifying statement which implies all of the local-global conjectures mentioned above \cite{Dad92}, \cite{Dad94}. More recently, Dade's Projective Conjecture has been reduced to the verification of the so-called inductive condition for Dade's Conjecture for finite quasi-simple groups \cite{Spa17}. This inductive condition can be stated in terms of Sp\"ath's Character Triple Conjecture (see \cite[Conjecture 6.3]{Spa17}).

When considering large primes in non-defining characteristic, work of Broué, Fong and Srinivasan shows that the $\ell$-local structure of a finite reductive group and the associated Brown complex can be seen as a shadow of geometric objects arising from the underlining linear algebraic group (see Section \ref{sec:Reduction first part}). Building on this idea, in this paper we propose new conjectures for finite reductive groups that can be seen as geometric realisations of the local-global counting conjectures (see Sections \ref{sec:Counting characters e-locally} and Section \ref{sec:CTC reductive}). These new conjectures imply the counting conjectures under the assumptions considered above (see Section \ref{sec:New and old conjecture for large primes}). Remarkably, our conjectures can be explained within the framework of generalised Harish-Chandra theory and, in fact, they reduce to the verification of Clifford theoretic properties expected from certain parametrisation of generalised Harish-Chandra series. In order to prove these results, we first prove new modular representation theoretic results for finite reductive groups by studying the decomposition of certain virtual representations constructed from the $\ell$-adic cohomology of Deligne--Lusztig varieties.

\subsection{$e$-Harish-Chandra theory}

Let $\G$ be a connected reductive group defined over an algebraic closure $\mathbb{F}$ of a finite field of characteristic $p$, $F:\G\to \G$ a Frobenius endomorphism endowing $\G$ with an $\mathbb{F}_q$-structure for some power $q$ of $p$ and $\G^F$ the finite reductive group consisting of the $\mathbb{F}_q$-rational points. Fix a prime number $\ell$ not dividing $q$ and denote by $e$ the multiplicative order of $q$ modulo $\ell$ (modulo $4$ if $\ell=2$). All modular representation theoretic notions are considered with respect to the prime $\ell$. Let $(\G^*,F^*)$ be in duality with $(\G,F)$. Blocks of finite reductive groups have been parametrised by work of Fong--Srinivasan \cite{Fon-Sri82}, \cite{Fon-Sri86}, Broué--Malle--Michel \cite{Bro-Mal-Mic93}, Cabanes--Enguehard \cite{Cab-Eng94}, \cite{Cab-Eng99}, Enguehard \cite{Eng00} and Kessar--Malle \cite{Kes-Mal13}, \cite{Kes-Mal15}. Given this parametrisation, we then need to understand the distribution of characters into such blocks. For this purpose, recall that the set $\irr(\G^F)$ of irreducible characters of $\G^F$ admits a partition
\[\irr\left(\G^F\right)=\coprod\limits_{B,s}\irr(B)\cap \E\left(\G^F,[s]\right)\]
where $B$ runs over the set of Brauer $\ell$-blocks of $\G^F$, $s$ runs over the set of semisimple elements in $\G^{*F^*}$ up to (rational) conjugation and $\E(\G^F,[s])$ is the rational Lusztig series associated to $s$. Using a terminology introduced by Broué, Fong and Srinivasan, we call each non-empty intersection
\[\E\left(\G^F,B,[s]\right):=\irr\left(B\right)\cap \E\left(\G^F,[s]\right)\]
a \emph{Brauer--Lusztig block}. In particular, each $\ell$-block $B$ is a union of Brauer--Lusztig blocks and therefore, in order to understand the distribution of characters into $\ell$-blocks, we need to describe Brauer--Lusztig blocks. Our first main result provides such a description in terms of $e$-Harish-Chandra series defined in terms of the $\ell$-adic cohomology of Deligne--Lusztig varieties.

\begin{theoA}
\label{thm:Main Brauer--Lusztig blocks}
Assume Hypothesis \ref{hyp:Brauer--Lusztig blocks} and let $\E(\G^F,B,[s])$ be a Brauer--Lusztig block. Then 
\[\E\left(\G^F,B,[s]\right)=\coprod\limits_{(\L,\lambda)}\E\left(\G^F,(\L,\lambda)\right),\]
where the union runs over the $\G^F$-conjugacy classes of $(e,s)$-cuspidal pairs $(\L,\lambda)$ (see Definition \ref{def:(e,s)-pair}) such that $\bl(\lambda)^{\G^F}=B$ via Brauer induction of $\ell$-blocks.
\end{theoA}

We point out that Hypothesis \ref{hyp:Brauer--Lusztig blocks} is satisfied in most of the cases of interest and, in particular, whenever $[\G,\G]$ is simply connected with no irreducible rational components of type ${^2 \mathbf{E}}_6(2)$, $\mathbf{E}_7(2)$, $\mathbf{E}_8(2)$ while considering $\ell\in\Gamma(\G,F)$ with $\ell\geq 5$ (see Remark \ref{rmk:Brauer-Lusztig for simple simply connected}). Moreover, in this case Brauer's induction of blocks is defined (see the discussion preceding Lemma \ref{lem:Jordan decomposition for l-elements, blocks}).

From the perspective of $e$-Harish--Chandra theory, Theorem \ref{thm:Main Brauer--Lusztig blocks} can be seen as an extension of results of Cabanes--Enguehard (see \cite[Theorem 4.1]{Cab-Eng99}) to $e$-cuspidal pairs associated to $\ell$-singular semisimple elements. In addition, Theorem \ref{thm:Main Brauer--Lusztig blocks} provides a generalisation of \cite[Theorem 3.2 (1)]{Bro-Mal-Mic93} to non-unipotent characters. This provides a uniform formulation for $e$-Harish-Chandra theory by considering arbitrary $e$-cuspidal pairs. In fact in Corollary \ref{cor:e-Harish-Chandra, disjointness}, we show that for every $e\geq 1$ the set of irreducible characters of $\G^F$ is partitioned into $e$-Harish-Chandra series provided that there exists a good prime $\ell$ for which $e$ is the order of $q$ modulo $\ell$. It is worth pointing out that this latter statement does not depend on the choice of the prime $\ell$ while, relying on block theoretic techniques, it's proof does. However, such a partition should existence for every $e\geq 1$ independently on the restrictions considered here.

On the way to prove Theorem \ref{thm:Main Brauer--Lusztig blocks}, we obtain two results that are of independent interest. First in Corollary \ref{cor:e-Harish-Chandra series and blocks}, we prove an extension of \cite[Theorem 2.5]{Cab-Eng99} that shows how Deligne--Lusztig induction preserves the decomposition into $\ell$-blocks. More precisely, under Hypothesis \ref{hyp:Brauer--Lusztig blocks} we show that for every $\ell$-block $b_\L$ of an $e$-split Levi subgroup $\L$ of $\G$ the irreducible constituents of the virtual representations obtained via Deligne--Lusztig induction from any $\lambda\in\irr(b_\L)$ belong to a unique $\ell$-block $b_\G$ of $\G^F$. Secondly in Corollary \ref{cor:Transitive closure}, we give a partial solution to a conjecture (see \cite[Notation 1.11]{Cab-Eng99} and Conjecture \ref{conj:Transitive closure}) introduced by Cabanes and Enguehard on the transitivity of a certain relation $\leq_e$ defined on the set of $e$-pairs (see also Proposition \ref{prop:Transitivity in good series}).

As an immediate consequence of Theorem \ref{thm:Main Brauer--Lusztig blocks}, we obtain the above-mentioned description of all the characters in any given $\ell$-block by considering the union over all conjugacy classes of semisimple elements of $\G^{*F^*}$. Namely, for every $\ell$-block $B$ of $\G^F$ there is a partition
\[\irr(B)=\coprod\limits_{(\L,\lambda)}\E\left(\G^F,(\L,\lambda)\right),\]
where the union runs over the $\G^F$-conjugacy classes of $e$-cuspidal pairs $(\L,\lambda)$ such that $\bl(\lambda)^{\G^F}=B$ via Brauer induction of $\ell$-blocks (see Theorem \ref{thm:Blocks are unions of e-HC series}). Given this partition, the next natural step to understand the distribution of characters into $\ell$-blocks of finite reductive groups is to find a parametrisation of the characters in each $e$-Harish-Chandra series $\E(\G^F,(\L,\lambda))$. Inspired by classical Harish-Chandra theory and by results of Broué, Malle and Michel for unipotent characters (see \cite[Theorem 3.2]{Bro-Mal-Mic93}), we propose a parametrisation for arbitrary $e$-Harish-Chandra series which is additionally compatible with Clifford theory and with the action of automorphisms. This Clifford theoretic compatibility is expressed via $\G^F$-block isomorphisms of character triples as defined in \cite[Definition 3.6]{Spa17}.

\begin{paraA}
\label{para:Main iEBC}
Let $\G$, $F$, $\ell$, $q$ and $e$ be as above and consider an $e$-cuspidal pair $(\L,\lambda)$ of $\G$. There exists a defect preserving $\aut_\mathbb{F}(\G^F)_{(\L,\lambda)}$-equivariant bijection
\[\Omega^\G_{(\L,\lambda)}:\E\left(\G^F,(\L,\lambda)\right)\to\irr\left(\n_\G(\L)^F\hspace{1pt}\middle|\hspace{1pt} \lambda\right)\]
such that
\[\left(X_\vartheta,\G^F,\vartheta\right)\iso{\G^F}\left(\n_{X_\vartheta}(\L),\n_{\G^F}(\L),\Omega^\G_{(\L,\lambda)}(\vartheta)\right)\]
in the sense of \cite[Definition 3.6]{Spa17} for every $\vartheta\in\E\left(\G^F,(\L,\lambda)\right)$ and where $X:=\G^F\rtimes \aut_\mathbb{F}(\G^F)$.
\end{paraA}

It is the author's belief that existence of the above parametrisation should provide an explaination for the validity of the inductive conditions for the local-global  conjectures for finite reductive groups in non-defining characteristic. As a matter of fact, similar bijections have been used in \cite{Mal-Spa16} to verify the McKay Conjecture for the prime $\ell=2$ and then in \cite{Ruh22AM} to prove the Alperin--McKay Conjecture and Brauer's Height Zero Conjecture for $\ell=2$. Regarding the validity of the above parametrisation, in \cite{Ros-Clifford_automorphisms_HC} we show that in order to obtain Parametrisation \ref{para:Main iEBC} it is enough to verify certain requirements on the extendibility of characters of $e$-split Levi subgroups. These also appear in the proofs of the inductive conditions for the McKay, the Alperin--McKay and the Alperin Weight conjectures and the checking of these requirements is part of an ongoing project in representation theory of finite reductive groups (see \cite{Mal-Spa16}, \cite{Cab-Spa17I}, \cite{Cab-Spa17II}, \cite{Cab-Spa19}, \cite{Bro-Spa20}, \cite{Spa21}, \cite{Bro22}).

\subsection{Counting conjectures via $e$-local structures}

Our next aim is to provide a geometric realisation of the local-global principle for finite reductive groups by replacing $\ell$-local structures with $e$-local structures. For this purpose we first need some notation. We keep $\G$, $F$, $q$, $\ell$ and $e$ as previously defined. Denote by $\CL_e(\G,F)_{>0}$ the set of non-trivial descending chains of $e$-split Levi subgroups $\sigma=\{\G=\L_0>\L_1\>\dots>\L_n\}$ and define the length of $\sigma$ as $|\sigma|:=n>0$. We denote by $\G^F_\sigma$ the stabiliser of a chain $\sigma$ under the action of $\G^F$. Then, as explained in Section \ref{sec:Counting characters e-locally}, we can associate to every $\ell$-block $b$ of $\G_\sigma^F$ an $\ell$-block $\R_{\G_\sigma}^\G(b)$ of $\G^F$. For every $\ell$-block $B$ of $\G^F$ and $d\geq 0$, we denote by $\k^d(B_\sigma)$ the number of irreducible characters $\vartheta$ of $\G^F_\sigma$ with $\ell$-defect $d$ and such that the $\ell$-block of $\vartheta$ in $\G_\sigma^F$ correspond to $B$ via $\R_{\G_\sigma}^\G$. Moreover, let $\k^d(B)$ and $\k_{\rm c}^d(B)$ be the number of irreducible and $e$-cuspidal characters respectively, with $\ell$-defect $d$ and belonging to the $\ell$-block $B$. Our first conjecture proposes a formula to count the number of characters of any given defect in any $\ell$-block in terms of $e$-local data.

\begin{conjA}
\label{conj:Main Dade reductive}
Let $B$ be an $\ell$-block of $\G^F$ and $d\geq 0$. Then
\[\k^d(B)=\k_{\rm c}^d(B)+\sum\limits_\sigma(-1)^{|\sigma|+1}\k^d(B_\sigma)\]
where $\sigma$ runs over a set of representatives for the action of $\G^F$ on $\CL_e(\G,F)_{>0}$.
\end{conjA}

Conjecture \ref{conj:Dade reductive} provides a geometric form of Dade's Conjecture for finite reductive groups. In fact, in Section \ref{sec:Towards Dade's Projective Conjecture} we show that the two statements coincide when the prime $\ell$ is large for $\G$ (see Proposition \ref{prop:Equivalences Dade}). We expect this connection to hold for good primes as well and we suspect that this could be connected to the existence of a homotopy equivalence between the Brown complex and the simplicial complex associated to $\CL_e(\G,F)_{>0}$. In the spirit of Conjecture \ref{conj:Main Dade reductive}, in Section \ref{sec:New conjectures} we introduce a statement (see Conjecture \ref{conj:AWC reductive}) which relates to Alperin's Weight Conjecture. As for the Alperin--McKay Conjecture and its inductive condition, at least for large primes, we identify a geometric counterpart in Parametrisation \ref{para:Main iEBC} (see Proposition \ref{prop:Equivalence AM} and Proposition \ref{prop:Equivalence iAM}). 

The numerical phenomena proposed in Conjecture \ref{conj:Main Dade reductive} as well as in the local-global counting conjectures are believed to be consequences of a deeper underlying theory. For blocks with abelian defect groups, Broué has suggested a structural explanation which predicts the existence of certain derived equivalences \cite{Bro90}. In a different direction, work of Isaacs, Malle and Navarro \cite{Isa-Mal-Nav07}, followed by Navarro and Sp\"ath \cite[Theorem 7.1]{Nav-Spa14I} (see also \cite{Ros-iMcK}) and Sp\"ath \cite[Conjecture 1.2]{Spa17}, suggests a description of Clifford theoretic properties hidden behind the counting conjectures by studying certain relations on sets of character triples. Exploiting this second approach, we now introduce a more conceptual background for Conjecture \ref{conj:Main Dade reductive} by introducing $\G^F$-block isomorphisms of character triples in this context.

In Section \ref{sec:CTC reductive}, for every $\ell$-block $B$, $d\geq 0$ and $\epsilon=\pm$, we introduce a set $\CL^d(B)_\epsilon/\G^F$ consisting of $\G^F$-orbits of quadruples $\omega$. Each $\omega\in\CL^d(B)_\epsilon/\G^F$ determines a $\G^F$-orbit $\omega^\bullet$ of pairs $(\sigma,\vartheta)$ consisting of a chain $\sigma\in\CL_e(\G,F)$ and a certain irreducible character $\vartheta$ of $\G_\sigma^F$. With this notation, we can now present our second conjecture.

\begin{conjA}
\label{conj:Main CTC reductive}
For every $\ell$-block $B$ of $\G^F$ and $d\geq 0$, there is an $\aut_\mathbb{F}(\G^F)_B$-equivariant bijection
\[\Lambda:\CL^d\left(B\right)_+/\G^F\to\CL^d\left(B\right)_-/\G^F\]
such that
\[\left(X_{\sigma,\vartheta},\G^F_\sigma,\vartheta\right)\iso{\G^F}\left(X_{\rho,\chi},\G^F_\rho,\chi\right)\]
in the sense of \cite[Definition 3.6]{Spa17} for every $\omega\in\CL^d(B)_+/\G^F$, any $(\sigma,\vartheta)\in\omega^\bullet$, any $(\rho,\chi)\in\Lambda(\omega)^\bullet$ and where $X:=\G^F\rtimes \aut_\mathbb{F}(\G^F)$.
\end{conjA}

As a first application of Theorem \ref{thm:Main Brauer--Lusztig blocks}, we establish a connection between Conjecture \ref{conj:Main Dade reductive} and Conjecture \ref{conj:Main CTC reductive}. More precisely, we show that for a fixed $\sigma\in\CL^d(\G,F)_{>0}$ the number $\k^d(B_\sigma)$ coincides with the number of quadruples $\omega$ for which there is a character $\vartheta\in\irr(\G_\sigma^F)$ such that $(\sigma,\vartheta)\in\omega^\bullet$. On the other hand $\k^d(B)-\k_{\rm c}^d(B)$ coincide with the number of quadruples $\omega$ for which there exists a character $\chi$ of $\G^F$ such that $(\sigma_0,\chi)\in\omega^\bullet$ and where $\sigma_0=\{\G\}$ is the trivial chain. In particular we see that Conjecture \ref{conj:Main CTC reductive} implies Conjecture \ref{conj:Main Dade reductive}.

\begin{theoA}
\label{thm:Main CTC reductive implies Dade reductive}
Assume Hypothesis \ref{hyp:Brauer--Lusztig blocks} and consider an $\ell$-block $B$ and $d\geq 0$. If Conjecture \ref{conj:Main CTC reductive} holds for $B$ and $d\geq 0$, then Conjecture \ref{conj:Main Dade reductive} holds for $B$ and $d\geq 0$.
\end{theoA}

In analogy with the inductive conditions for the local-global counting conjectures, Conjecture \ref{conj:Main CTC reductive} provides a more conceptual explanation for the numerical phenomenon proposed by Conjecture \ref{conj:Main Dade reductive} and, in particular, yields a description of the Clifford theoretic properties naturally arising in this context. Furthermore, in Section \ref{sec:Towards Dade's Projective Conjecture} we show that Conjecture \ref{conj:Main CTC reductive} implies Sp\"ath's Character Triple Conjecture and the inductive condition for Dade's Conjecture for large primes under suitable assumptions (see Proposition \ref{prop:Equivalence CTC} and Corollary \ref{cor:Equivalence iDade}).

Remarkably, Conjecture \ref{conj:Main CTC reductive} can be explained within the framework of $e$-Harish-Chandra theory. Our final theorem shows that Conjecture \ref{conj:Main CTC reductive} (and hence Conjecture \ref{conj:Main Dade reductive}) is a consequence of Parametrisation \ref{para:Main iEBC} via an application of Theorem \ref{thm:Main Brauer--Lusztig blocks}. In what follows, we say that Parametrisation \ref{para:Main iEBC} holds for $(\G,F)$ at the prime $\ell$ if it holds for every $e$-cuspidal pair $(\L,\lambda)$ of $\G$ where $q$ is the prime power associated to $F$ and $e$ is the order of $q$ modulo $\ell$. Moreover, we refer the reader to Section \ref{sec:Final} for the definition of irreducible rational components (see Definition \ref{def:Irreducible rational component}).

\begin{theoA}
\label{thm:Main reduction}
Assume Hypothesis \ref{hyp:Brauer--Lusztig blocks} and suppose that $\G$ is simply connected with Frobenius endomorphism $F$. If Parametrisation \ref{para:Main iEBC} holds at the prime $\ell$ for every irreducible rational component $(\H,F)$ of every $e$-split Levi subgroup of $\G$, then Conjecture \ref{conj:Main CTC reductive} holds for the prime $\ell$.
\end{theoA}

As a consequence of Theorem \ref{thm:Main CTC reductive implies Dade reductive}, Theorem \ref{thm:Main reduction} and the results obtained in \cite{Ros-Clifford_automorphisms_HC}, our conjectures are now reduced to the verification of technical requirements on character extendibility from $e$-split Levi subgroups in finite reductive groups of irreducible rational type.

\subsection{Reader's guide}

The paper is organised as follows. Section \ref{sec:Preliminaries} contains the main notation and preliminary results on finite reductive groups. In Section \ref{sec:transitivity}, and more precisely in Proposition \ref{prop:e-HC theory for connected reductive}, we study the transitivity of a certain relation defined on the set of $e$-pairs and provide a solution to a fundamental case of a conjecture proposed by Cabanes--Enguehard in \cite[Notation 1.11]{Cab-Eng99}. This is then extended to the $\ell$-singular case in Corollary \ref{cor:Transitive closure}. Furthermore, in Proposition \ref{prop:Transitivity in good series} we prove the conjecture inside $e$-Harish-Chandra series associated with certain good semisimple elements. Most importantly, in Section \ref{sec:Brauer-Lusztig blocks and e-HC series} we obtain a description of the distribution of characters into $\ell$-blocks of finite reductive groups (see Theorem \ref{thm:Blocks are unions of e-HC series}) and prove Theorem \ref{thm:Main Brauer--Lusztig blocks} (see Theorem \ref{thm:Brauer-Lusztig blocks are unions of e-HC series}). As a by-product, in Corollary \ref{cor:e-Harish-Chandra series and blocks} we prove an extension of \cite[Theorem 2.5]{Cab-Eng99} to arbitrary $e$-cuspidal pairs while Corollary \ref{cor:e-Harish-Chandra, disjointness} show the existence of an $e$-Harish-Chandra partition of characters under the existence of a prime $\ell$ satisfying certain properties. In section \ref{sec:New conjectures}, we introduce Conjecture \ref{conj:Main Dade reductive} (see Conjecture \ref{conj:Dade reductive}), a similar statement relating to Alperin's Weight Conjecture (see Conjecture \ref{conj:AWC reductive}), Conjecture \ref{conj:Main CTC reductive} (see Conjecture \ref{conj:CTC reductive}) and prove Theorem \ref{thm:Main CTC reductive implies Dade reductive} (see Theorem \ref{thm:CTC reductive imples Dade reductive}). Here, we also consider Parametrisation \ref{para:Main iEBC} (see Parametrisation \ref{para:iEBC}) in more detail and consider an interesting consequence in Remark \ref{rmk:Consequences of parametrisation}. Section \ref{sec:Final} is dedicated to the proof of Theorem \ref{thm:Main reduction}. This relies on Theorem \ref{thm:Main Brauer--Lusztig blocks} and requires a careful study of the interaction between Parametrisation \ref{para:Main iEBC} and Conjecture \ref{conj:Main CTC reductive}. Finally, in Section \ref{sec:Towards Dade's Projective Conjecture} we make use of ideas of Broué, Fong and Srinivasan in order to show that our conjectures agree with the local-global counting conjectures for finite reductive groups and large primes.

\section{Preliminaries}
\label{sec:Preliminaries}

Throughout this paper, $\G$ is a connected reductive linear algebraic group defined over an algebraic closure $\mathbb{F}$ of a finite field of characteristic $p$ and $F:\G\to \G$ is a Frobenius endomorphism endowing $\G$ with an $\mathbb{F}_q$-structure for a power $q$ of $p$. We denote by $(\G^*,F^*)$ a group in duality with $(\G,F)$ with respect to a choice of an $F$-stable maximal torus $\T$ of $\G$ and an $F^*$-stable maximal torus $\T^*$ of $\G^*$. In this case, there exists a bijection $\L\mapsto \L^*$ between the set of Levi subgroups of $\G$ containing $\T$ and the set of Levi subgroups of $\G^*$ containing $\T^*$ (see \cite[p.123]{Cab-Eng04}). This bijection induces a correspondence between the set of $F$-stable Levi subgroups of $\G$ and the set of $F^*$-stable Levi subgroups of $\G^*$. Moreover, it is compatible with the action of $\G^F$ and $\G^{*F^*}$.

\subsection{Automorphisms}
\label{sec:Automorphisms}

Let $\G$ and $F$ be as above. If $\sigma:\G\to\G$ is a bijective morphism of algebraic groups satisfying $\sigma\circ F=F\circ \sigma$, then the restriction of $\sigma$ to $\G^F$, which by abuse of notation we denote again by $\sigma$, is an automorphism of the finite group $\G^F$. We denote by $\aut_\mathbb{F}(\G^F)$ the set of those automorphisms of $\G^F$ obtained in this way. As mentioned in \cite[Section 2.4]{Cab-Spa13}, a morphism $\sigma\in\aut_\mathbb{F}(\G^F)$ is determined by its restriction to $\G^F$ up to a power of $F$. It follows that $\aut_\mathbb{F}(\G^F)$ acts on the set of $F$-stable closed connected subgroups $\H$ of $\G$. In particular, for any $F$-stable closed connected subgroup $\H$ of $\G$, there is a well defined set $\aut_\mathbb{F}(\G^F)_\H$ whose elements are the restrictions to $\G^F$ of those morphisms $\sigma$ as above that stabilize $\H$. When $\G$ is a simple algebraic group of simply connected type such that $\G^F/\z(\G^F)$ is a non-abelian simple group, then we have $\aut_\mathbb{F}(\G^F)=\aut(\G^F)$ (see \cite[Section 1.15]{GLS} and the comments in \cite[Section 2.4]{Cab-Spa13}).

Assume now that $\G$ is simple of simply connected type. Fix a maximally split torus $\T_0$ contained in an $F$-stable Borel subgroup $\B_0$ of $\G$. This choice corresponds to a set of simple roots $\Delta\subseteq \Phi:=\Phi(\G,\T_0)$. For every $\alpha\in\Phi$ consider a one-parameter subgroup $x_\alpha:\mathbb{G}_{\rm a}\to \G$. Then $\G$ is generated by the elements $x_\alpha(t)$, where $t\in\mathbb{G}_{\rm a}$ and $\alpha\in\pm\Delta$. Consider the \emph{field endomorphism} $F_0:\G\to \G$ given by $F_0(x_\alpha(t)):=x_{\alpha}(t^p)$ for every $t\in\mathbb{G}_{\rm a}$ and $\alpha\in\Phi$. Moreover, for every symmetry $\gamma$ of the Dynkin diagram of $\Delta$, we have a \emph{graph automorphism} $\gamma:\G\to \G$ given by $\gamma(x_\alpha(t)):=x_{\gamma(\alpha)}(t)$ for every $t\in\mathbb{G}_{\rm a}$ and $\alpha\in\pm\Delta$. Then, up to inner automorphisms of $\G$, any Frobenius endomorphism $F$ defining an $\mathbb{F}_q$-structure on $\G$ can be written as $F=F_0^m\gamma$, for some symmetry $\gamma$ and $m\in\mathbb{Z}$ with $q=p^m$ (see \cite[Theorem 22.5]{Mal-Tes}). One can construct a regular embedding $\G\leq\wt{\G}$ in such a way that the Frobenius endomorphism $F_0$ extends to an algebraic group endomorphism $F_0:\wt{\G}\to\wt{\G}$ defining an $\mathbb{F}_p$-structure on $\wt{\G}$. Moreover, every graph automorphism $\gamma$ can be extended to an algebraic group automorphism of $\wt{\G}$ commuting with $F_0$ (see \cite[Section 2B]{Mal-Spa16}). If we denote by $\mathcal{A}$ the group generated by $\gamma$ and $F_0$, then we can construct the semidirect product $\wt{\G}^F\rtimes \mathcal{A}$. Finally, we define the set of \emph{diagonal automorphisms} of $\G^F$ to be the set of those automorphisms induced by the action of $\wt{\G}^F$ on $\G^F$. If $\G^F/\z(\G^F)$ is a non-abelian simple group, then the group $\wt{\G}^F\rtimes \mathcal{A}$ acts on $\G^F$ and induces all the automorphisms of $\G^F$ (see, for instance, the proof of \cite[Proposition 3.4]{Spa12} and of \cite[Theorem 2.4]{Cab-Spa19}).

We conclude this section by recalling an important property that is needed in Section \ref{sec:Final}.

\begin{lem}
\label{lem:Centralizer, automorphisms and central subgroups}
Let $\G$, $\wt{\G}$, $F$ and $\mathcal{A}$ as in the above paragraph and suppose that $\G^F$ is the universal covering group of $\G^F/\z(\G^F)$. Let $Z\leq \z(\G^F)$ and denote by $(\wt{\G}^F\mathcal{A})_Z$ the normaliser of $Z$ in $\wt{\G}^F\mathcal{A}$. Then
\[\c_{(\wt{\G}^F\mathcal{A})_Z/Z}\left(\G^F/Z\right)=\z\left(\wt{\G}^F\right)/Z\]
and the canonical map $(\wt{\G}^F\mathcal{A})_Z\to\aut(\G^F/Z)$ induces an isomorphism
\[\left(\wt{\G}^F\mathcal{A}\right)_Z/\z\left(\wt{\G}^F\right)\simeq \aut\left(\G^F/Z\right).\]
\end{lem}

\begin{proof}
By the above paragraph, we know that $\wt{\G}^F\mathcal{A}/\c_{\wt{\G}^F\mathcal{A}}(\G^F)\simeq \aut(\G^F)$ and therefore, using the fact that $\c_{\wt{\G}^F\mathcal{A}}(\G^F)=\z(\wt{\G}^F)$ (see the argument used in \cite[Proposition 3.4 (a)]{Spa12}, \cite[Theorem 2.4]{Cab-Spa19} and ultimately \cite[Theorem 2.5.1]{GLS}), we obtain $(\wt{\G}^F\mathcal{A})_Z/\z(\wt{\G}^F)\simeq \aut(\G^F)_Z$. Then, by \cite[Corollary 5.1.4 (b)]{GLS}, it follows that
\[\left(\wt{\G}^F\mathcal{A}\right)_Z/\z\left(\wt{\G}^F\right)\simeq \aut\left(\G^F\right)_Z\simeq \aut\left(\G^F/Z\right).\]
On the other hand, since
\[\aut\left(\G^F/Z\right)\simeq \dfrac{\left(\wt{\G}^F\mathcal{A}\right)_Z/Z}{\c_{\left(\wt{\G}^F\mathcal{A}\right)_Z/Z}\left(\G^F/Z\right)},\]
the third isomorphism theorem yields the desired isomorphism.
\end{proof}

\subsection{Good primes and $e$-split Levi subgroups}
\label{sec:Good primes and Levi}

For the rest of this section we consider the following setting.

\begin{notation}
\label{notation}
Let $\G$ be a connected reductive linear algebraic group defined over an algebraic closure $\mathbb{F}$ of a finite field of characteristic $p$ and $F:\G\to \G$ a Frobenius endomorphism defining an $\mathbb{F}_q$-structure on $\G$, for a power $q$ of $p$. Consider a prime $\ell$ different from $p$ and denote by $e$ the multiplicative order of $q$ modulo $\ell$ (modulo $4$ if $\ell=2$). All blocks are considered with respect to the prime $\ell$.
\end{notation}

In what follows we make some restrictions on the prime $\ell$. First, recall that $\ell$ is a \emph{good prime} for $\G$ if it is good for each simple factor of $\G$, while the conditions for the simple factors are as follows:
\begin{align*}
{\rm \bf A}_n&: \text{every prime is good}
\\
{\rm \bf B}_n, {\rm \bf C}_n, {\rm \bf D}_n&: \ell\neq 2
\\
{\rm \bf G}_2, {\rm \bf F}_4, {\rm \bf E}_6, {\rm \bf E}_7&: \ell\neq 2,3
\\
{\rm \bf E}_8&: \ell\neq 2,3,5.
\end{align*}
We say that $\ell$ is a \emph{bad prime} for $\G$ if it is not a good prime. Next, we introduce the set of primes $\Gamma(\G,F)$ from \cite[Notation 1.1]{Cab-Eng94}.

\begin{defin}
\label{def:Very good primes}
We denote by $\gamma(\G,F)$ the set of primes $\ell$ such that: $\ell$ is odd, $\ell\neq p$, $\ell$ is good for $\G$ and $\ell$ doesn't divide $|\z(\G)^F:\z^\circ(\G)^F|$. Let $(\G^*,F^*)$ be in duality with $(\G,F)$ and set $\Gamma(\G,F):=(\gamma(\G,F)\cap \gamma(\G^*,F^*))\setminus\{3\}$ if $\G_{\rm ad}^F$ has a component of type $^3\mathbf{D}_4(q^m)$ and $\Gamma(\G,F):=\gamma(\G,F)\cap \gamma(\G^*,F^*)$ otherwise.
\end{defin}

\begin{rmk}
\label{rmk:Very good primes go to Levi subgroups}
 Notice that, if $\ell\in\Gamma(\G,F)$, then $\ell\in\Gamma(\G^*,F^*)$ and $\ell\in\Gamma(\H,F)$ for every $F$-stable connected reductive subgroup $\H$ of $\G$ containing an $F$-stable maximal torus of $\G$ (see \cite[Proposition 13.12]{Cab-Eng04}). In particular, if $\ell\in\Gamma(\G,F)$ and $\L$ is an $F$-stable Levi subgroup of $\G$, then $\ell\in\Gamma(\L,F)$.
\end{rmk}

Using \cite[Table 13.11]{Cab-Eng04}, we describe the primes $\ell\in\Gamma(\G,F)$ when $\G$ is simple of simply connected type with Frobenius endomorphism $F$ defining an $\mathbb{F}_q$-structure on $\G$:
\begin{align*}
{\rm \bf A}_n(q)&: \ell\nmid 2q(n+1,q-1), 
\\
^2{\rm \bf A}_n(q)&: \ell\nmid 2q(n+1,q+1), 
\\
{\rm \bf B}_n(q), {\rm \bf C}_n(q), {\rm \bf D}_n(q), ^2{\rm \bf D}_n(q)&: \ell\neq 2,p
\\
^3{\rm \bf D}_4(q), {\rm \bf G}_2(q), {\rm \bf F}_4(q), {\rm \bf E}_6(q), ^2{\rm \bf E}_6(q), {\rm \bf E}_7(q)&: \ell\neq 2,3,p
\\
{\rm \bf E}_8(q)&: \ell\neq 2,3,5,p.
\end{align*}
As a consequence, if a connected reductive group $\G$ has no simple components of type ${\rm \bf A}$, then $\ell\in\Gamma(\G,F)$ if and only if $\ell$ is good for $\G$ and $\ell\neq p$.

In this paper we make use of the terminology of Sylow $\Phi_e$-theory introduced in \cite{Bro-Mal92} (see also \cite{Bro-Mal-Mic93}). For a set of positive integers $E$, we say that an $F$-stable Levi subgroup $\T$ of $\G$ is a $\Phi_E$-torus if its order polynomial is of the form $P_{(\T,F)}=\prod_{n\in E}\Phi_n^{a_n}$ for some integers $a_n$ and where $\Phi_n$ denotes the $n$-th cyclotomic polynomial (see \cite[Definition 13.3]{Cab-Eng04}). The centralisers of $\Phi_E$-tori are called $E$-split Levi subgroups. If $E=\{e\}$, we call $\Phi_{\{e\}}$-tori and $\{e\}$-split Levi subgroups simply $\Phi_e$-tori and $e$-split Levi subgroups respectively. When $\ell\in\Gamma(\G,F)$, some significant consequences on the structure of $e$-split Levi subgroups can be drawn.

\begin{lem}
\label{lem:e-split Levi}
Let $\L$ be an $F$-stable Levi subgroup of $\G$.
\begin{enumerate}
\item Let $E$ be a set of positive integers. Then $\L$ is $E$-split if and only if $\L=\c_\G(\z^\circ(\L)_{\Phi_E})$.
\item Set $E_{q,\ell}:=\{e\cdot\ell^m\mid m\in\mathbb{N}\}$. If $\L=\c^\circ_\G\left(\z^\circ(\L)^F_\ell\right)$, then $\L$ is $E_{q,\ell}$-split. The converse holds provided that $\ell\in\Gamma(\G,F)$.
\end{enumerate}
\end{lem}

\begin{proof}
The first statement follows directly from the definition. In fact, since $\z^\circ(\L)$ is a torus, we deduce that $\z^\circ(\L)_{\Phi_E}$ is a $\Phi_E$-torus and therefore $\c_\G(\z^\circ(\L)_{\Phi_E})$ is $E$-split. Conversely, assume that $\L$ is $E$-split. Then there exists a $\Phi_E$-torus $\T$ such that $\L=\c_\G(\T)$. Since $\T$ is abelian, we deduce that $\T\leq \z(\L)$. Then, as $\T$ is connected, we have $\T\leq \z^\circ(\L)$ and therefore $\T\leq \z^\circ(\L)_{\Phi_E}$ because $\T=\T_{\Phi_E}$. By \cite[Proposition 1.21]{Dig-Mic91}, we conclude that $\L=\c_\G(\z^\circ(\L))\leq \c_\G(\z^\circ(\L)_{\Phi_E})\leq \c_\G(\T)=\L$. For the second statement see \cite[Proposition 13.19]{Cab-Eng04}.
\end{proof}

\begin{lem}
\label{prop:Good primes}
Let $\ell\in\Gamma(\G,F)$ and consider  an $\ell$-subgroup $Y$ of $\G^F$. Then:
\begin{enumerate}
\item $\c_\G(Y)^F=\c_\G^\circ(Y)^F$;
\item if $Y$ is abelian, then $Y\leq \z^\circ(\c_\G^\circ(Y))$; 
\item if $Y$ is abelian, then $\c_\G^\circ(Y)$ is an $E_{q,\ell}$-split Levi subgroup of $\G$.
\end{enumerate}
\end{lem}

\begin{proof}
First notice that $\c_\G^\circ(Y)$ is a Levi subgroup of $\G$ by \cite[Proposition 13.16 (ii)]{Cab-Eng04}. The first statement is \cite[Proposition 13.16 (i)]{Cab-Eng04}. Assume now that $Y$ is abelian and notice that $Y\leq \c_\G(Y)^F=\c_\G^\circ(Y)^F$. Then $Y\leq \z(\c_\G^\circ(Y))$. By (i) we know that $\c^\circ_\G(Y)$ is a Levi subgroup of $\G$ and hence $\ell\in\Gamma(\c_\G^\circ(Y),F)$ by Remark \ref{rmk:Very good primes go to Levi subgroups}. In particular $\ell$ does not divide $|\z(\c^\circ_\G(Y))^F:\z^\circ(\c_\G^\circ(Y))^F|$ and so $Y\leq \z^\circ(\c_\G^\circ(Y))$. Now, $\L:=\c_\G^\circ(Y)$ is an $F$-stable Levi subgroup with $Y\leq \z^\circ(\L)$. Then $Y\leq \z^\circ(\L)^F_\ell\leq \z^\circ(\L)$ and \cite[Proposition 1.21]{Dig-Mic91} implies that $\L=\c_\G^\circ(\z^\circ(\L))\leq\c_\G^\circ(\z^\circ(\L)^F_\ell)\leq \c_\G^\circ(Y)=\L$. By Lemma \ref{lem:e-split Levi} it follows that $\L$ is $E_{q,\ell}$-split in $\G$.
\end{proof}

We say that a prime $\ell$ is large for $(\G,F)$ if there exists a unique integer $e_0$ such that $\Phi_{e_0}$ divides the order polinomial $P_{(\G,F)}$ and $\ell$ divides $\Phi_{e_0}(q)$ (see \cite[Definition 5.1]{Bro-Mal-Mic93} and \cite[Section 2.1]{Mal14}). In this case we also say that $\ell$ is $(\G,F,e_0)$-adapted (see \cite[Definition 5.3]{Bro-Mal-Mic93}). When $\G$ is semisimple, then $e_0$ coincides with $e$. Observe that if $\ell$ is large, then the Sylow $\ell$-subgroups of $\G^F$ are abelian \cite[Theorem 25.14]{Mal-Tes}. Moreover, if $\ell$ is large for $\G$, then $\ell\in\Gamma(\G,F)$ by \cite[Lemma 2.1]{Mal14}.

For any finite $\ell$-group $H$ and positive integer $n$ we define the subgroup $\Omega_n(H):=\langle h\in H\mid h^{\ell^n}=1\rangle$. In particular, when $H$ is abelian, $\Omega_1(H)$ is the largest $\ell$-elementary abelian subgroup of $H$.

\begin{prop}
\label{prop:e-split Levi large primes}
Suppose that $\ell$ is large for $\G$ and consider an $e$-split Levi subgroup $\L$ of $\G$ and an $\ell$-subgroup $Y$ of $\G^F$. Then:
\begin{enumerate}
\item if $\ell$ is $(\G,F,e_0)$-adapted for some integer $e_0$, then $\c_\G^\circ(Y)$ is an $e_0$-split Levi subgroup of $\G$;
\item if $\S:=\z^\circ(\L)_{\Phi_e}$, then $\S\leq \z(\G)$ if and only if $\S_\ell^F\leq \z(\G)$ if and only if $\Omega_1(\S_\ell^F)\leq \z(\G)$;
\item $\L=\c_\G^\circ((\z^\circ(\L)_{\Phi_e})^F_\ell)=\c_\G^\circ(\Omega_1((\z^\circ(\L)_{\Phi_e})^F_\ell))$.
\end{enumerate}
\end{prop}

\begin{proof}
By \cite[Proposition 13.16 (ii)]{Cab-Eng04} and Lemma \ref{prop:Good primes} (ii) we know that $\c^\circ_\G(Y)$ is an $F$-stable Levi subgroup and $Y\leq \z^\circ(\c_\G^\circ(Y))$. Then \cite[Proposition 2.4 (1)]{Bro-Mal-Mic93} implies that $\c^\circ_\G(Y)$ is an $e_0$-split Levi subgroup. This proves (i).

Next, set $\S:=\z^\circ(\L)_{\Phi_e}$. It is enough to show that $\S\nleq \z(\G)$ implies $\Omega_1(\S_\ell^F)\nleq\z(\G)$. Since $\S$ is a $\Phi_e$-torus, we have $\S\nleq\z(\G)$ if and only if $\S\nleq \z^\circ(\G)_{\Phi_e}$ while, using the fact that $\ell$ is large, we deduce that $\z(\G)_\ell^F=(\z^\circ(\G)_{\Phi_e})_\ell^F$ and therefore $\Omega_1(\S_\ell^F)\nleq\z(\G)$ if and only if $\Omega_1(\S_\ell^F)\nleq\z^\circ(\G)_{\Phi_e}$. Hence, in order to obtain (ii) we need to show that $\S\nleq \z^\circ(\G)_{\Phi_e}$ implies $\Omega_1(\S_\ell^F)\nleq\z^\circ(\G)_{\Phi_e}$. Assume that $\S\nleq \z^\circ(\G)_{\Phi_e}=:\T_e$ and consider the canonical morphism $\pi_e:\G\to \G/\T_e$. Observe, that $\T_e\leq \S$ and that $\S$ and $\pi_e(\S)\neq 1$ are $\Phi_{e}$-tori. If $\ell^{a}$ is the largest power of $\ell$ dividing $\Phi_{e}(q)$, then $\T^F_\ell$ is the direct product of copies of $C_{\ell^{a}}$ for every $\Phi_{e}$-torus $\T$ (see \cite[Proposition 3.3]{Bro-Mal92}). Let $y\in\pi_e(\S)^F_\ell$ be an element of order $\ell^a$. Proceeding as in the proof of \cite[Lemma 13.17 (i)]{Cab-Eng04} and noticing that $(\S/\T_e)^F=\S^F/\T_e^F$, we deduce that $\pi_e(\S)_\ell^F=\pi_e(\S^F_\ell)$ and hence there exists $x\in\S_\ell^F$ such that $\pi_e(x)=y$. Now, the order of $y$ divides the order of $x$ while the order of $x$ must be less or equal than $\ell^a$ by the description of $\S^F_\ell$ given above. We conclude that $x$ has order $\ell^a$. Then $s:=x^{\ell^{a-1}}\in\Omega_1(\S_\ell^F)$ and $\pi_e(s)=y^{\ell^{a-1}}\neq 1$. This shows that $\Omega_1(\S_\ell^F)\nleq \T_e=\z^\circ(\G)_{\Phi_e}$.

To prove (iii), we proceed by induction on the dimension of $\G$. Set $\S:=\z^\circ(\L)_{\Phi_e}$ and notice that $\L=\c_\G(\S)\leq\c_\G^\circ(\S_\ell^F)\leq \c_\G^\circ(\Omega_1(\S_\ell^F))$. We need to show that $\K:=\c_\G^\circ(\Omega_1(\S_\ell^F))\leq \L$. Observe that $\K$ is a Levi subgroup of $\G$ by \cite[Proposition 13.16 (ii)]{Cab-Eng04}. If $\S\leq \z(\G)$, then $\K=\G=\L$. Therefore, we can assume $\S\nleq\z(\G)$. By the above paragraph, we obtain $\Omega_1(\S_\ell^F)\nleq \z(\G)$ and therefore $\dim(\K)<\dim(\G)$. Noticing that $\ell$ is large for $\K$ and that $\L$ is an $e$-split Levi subgroup of $\K$, the inductive hypothesis yields $\L=\c_\K^\circ(\Omega_1(\S_\ell^F))$. The result follows by noticing that $\c^\circ_\K(\Omega_1(\S_\ell^F))=\K$.
\end{proof}

\subsection{Deligne--Lusztig induction and blocks}
\label{sec:Deligne-Lusztig induction and blocks}

Let $\G$, $F$, $q$, $\ell$ and $e$ be as in Notation \ref{notation} and consider an $F$-stable Levi subgroup of a (not necessarily $F$-stable) parabolic subgroup $\P$ of $\G$. By tensoring with a $(\G^F,\L^F)$-bimodules arising from the $\ell$-adic cohomology of Deligne--Lusztig varieties, Deligne--Lusztig \cite{Del-Lus76} (in the case where $\L$ is a maximal torus) and Lusztig \cite{Lus76} (in the general case) have defined a map
\[\R_{\L\leq \P}^\G:\mathbb{Z}\irr\left(\L^F\right)\to\mathbb{Z}\irr\left(\G^F\right)\]
with adjoint
\[^*\R_{\L\leq \P}^\G:\mathbb{Z}\irr\left(\G^F\right)\to\mathbb{Z}\irr\left(\L^F\right)\]
that we call Deligne--Lusztig induction and restriction respectively. Notice that these maps are often referred to simply as Lusztig induction and restriction, and the terms Deligne--Lusztig induction and restriction are only used when considering the case of a maximal torus. Nonetheless we believe that the contribution of Deligne should be acknowledged. It is expected that the map $\R_{\L\leq \P}^\G$ does not depend on the choice of the parabolic subgroup $\P$ and this would, for instance, follow from the Mackey formula which has been proved whenever $\G^F$ does not have components of type $^2\mathbf{E}_6(2)$, $\mathbf{E}_7(2)$ or $\mathbf{E}_8(2)$ \cite{Bon-Mic10}. For simplicity, we just write $\R_\L^\G$ when the results are known not to depend on the choice of $\P$. Similar remarks apply for Deligne--Lusztig restriction.

Recall that $(\L,\lambda)$ is an $e$-cuspidal pair of $(\G,F)$ (or simply of $\G$ when no confusion arises) if $\L$ is an $e$-split Levi subgroup of $\G$ and $\lambda\in\irr(\L^F)$ satisfies $^*\R_{\M\leq \Q}^\L(\lambda)=0$ for every $e$-split Levi subgroup $\M<\L$ and every parabolic subgroup $\Q$ of $\L$ containing $\M$ as Levi complement.

To fix our notation, we now review the parametrisation of blocks given in \cite{Cab-Eng99}. Let's assume $\ell\geq 5$ with $\ell\geq 7$ if $\G$ has a component of type $\mathbf{E}_8$. Then for every $B\in\Bl(\G^F)$ there exists a unique $e$-cuspidal pair $(\L,\lambda)$ up to $\G^F$-conjugation such that $\lambda$ lies in a rational Lusztig series associated with an $\ell$-regular semisimple element and all the irreducible constituents of $\R_{\L\leq \P}^\G(\lambda)$ belongs to the block $B$ for every parabolic subgroup $\P$ of $\G$ having $\L$ as Levi complement. In this case we write $B=b_{\G^F}(\L,\lambda)$. Moreover, \cite[Theorem 2.5]{Cab-Eng99} implies that $\bl(\lambda)^{\G^F}=B$ whenever $\ell\in\Gamma(\G,F)$ (see Lemma \ref{lem:e-split Levi}). See \cite{Kes-Mal15} for a generalisation of these results to all primes.

\section{$(e,\ell')$-pairs and transitivity}
\label{sec:transitivity}

In this section we provide new evidence for a conjecture proposed by Cabanes and Enguehard in \cite[Notation 1.11]{Cab-Eng99}. Consider $\G$, $F$, $q$, $\ell$ and $e$ as in Notation \ref{notation}. We start by defining the notion of $e$-pair and $(e,s)$-pair.

\begin{defin}
\label{def:(e,s)-pair}
An $e$-pair of $(\G,F)$ (or simply of $\G$ when no confusion arises) is a pair $(\L,\lambda)$ where $\L$ is an $e$-split Levi subgroup of $\G$ and $\lambda\in\irr(\L^F)$. For any semisimple element $s\in\G^{*F^*}$, we say that an $e$-pair $(\L,\lambda)$ is an \emph{$(e,s)$-pair} of $(\G,F)$ if $\lambda\in\E(\L^F,[s'])$ for some $s'\in\L^{*F^*}$ that is $\G^{*F^*}$-conjugate to $s$. Finally, we say that $(\L,\lambda)$ is an \emph{$(e,\ell')$-pair} if it is an $(e,s)$-pair for some $\ell$-regular semisimple element $s\in\G^{*F^*}$.
\end{defin}

If $\mathcal{P}_e(\G,F)$ is the set of $e$-pairs of $(\G,F)$, then there exists a binary relation on $\mathcal{P}_e(\G,F)$ denoted by $\leq_e$ (see \cite[Notation 1.11]{Cab-Eng99}). Namely, write $(\L,\lambda)\leq_e(\K,\kappa)$ provided that $\L\leq\K$ are $e$-split Levi subgroups of $\G$ and there exists a parabolic subgroup $\P$ of $\K$ containing $\L$ as a Levi complement such that $\kappa$ is an irreducible constituent of the generalised character $\R_{\L\leq \P}^\K(\lambda)$. Noticing that Deligne--Lusztig induction sends characters to generalised characters, we observe that the relation $\leq_e$ might not be transitive at first glance. We denote by $\ll_e$ the transitive closure of $\leq_e$. Since $\mathcal{P}_e(\G,F)$ is finite, %the set of $F$-stable maximal tori is finite by \cite[Theorem 25.5]{Mal-Tes}. For every $F$-stable maximal torus $\T$, the set of Levi subgroups containing $\T$ is finite (in fact is determined by subsystems of the root system).
we deduce that two $e$-pairs $(\L,\lambda)$ and $(\K,\kappa)$ satisfy $(\L,\lambda)\ll_e(\K,\kappa)$ if and only if there exist a finite number of $e$-pairs $(\L_i,\lambda_i)$, with $i=1,\dots,n$, such that
\[(\L,\lambda)\leq_e(\L_1,\lambda_1)\leq_e\dots\leq_e(\L_n,\lambda_n)\leq_e(\K,\kappa).\]
With this notation, a pair $(\L,\lambda)$ is $e$-cuspidal if and only if it is a minimal element in the poset $(\mathcal{P}_e(\G,F),\ll_e)$. We denote by $\CP_e(\G,F)$ the subset of $\mathcal{P}_e(\G,F)$ consisting of $e$-cuspidal pairs. Observe that, by \cite[Proposition 15.7]{Cab-Eng04} the relations $\leq_e$ and $\ll_e$ restrict to the set of $(e,s)$-pairs for every $s\in\G^{*F^*}_{\rm ss}$. A minimal element in the induced poset of $(e,s)$-pairs is called \emph{$(e,s)$-cuspidal}.

The following conjecture has been proposed in \cite[Notation 1.11]{Cab-Eng99} and is inspired by \cite[Theorem 3.11]{Bro-Mal-Mic93}.

\begin{conj}[Cabanes--Enguehard]
\label{conj:Transitive closure}
The relation $\leq_e$ is transitive and therefore coincides with $\ll_e$.
\end{conj}

In this section, we show that this conjecture holds when considering $(e,\ell')$-cuspidal pairs in groups of simply connected type under certain assumptions on $\ell$. Before proceeding with the proof of this result, we point out an important consequence of Conjecture \ref{conj:Transitive closure}. Let $(\L,\lambda)$ be an $e$-pair of $\G$. If Conjecture \ref{conj:Transitive closure} holds, then
\[\left\lbrace\chi\in\irr\left(\G^F\right)\hspace{1pt}\middle|\hspace{1pt}(\L,\lambda)\ll_e(\G,\chi)\right\rbrace=\E\left(\G^F,(\L,\lambda)\right),\]
where $\E(\G^F,(\L,\lambda))$ is the \emph{$e$-Harish-Chandra series} determined by $(\L,\lambda)$, that is the set of irreducible constituents of $\R_{\L\leq \P}^\G(\lambda)$ for every parabolic subgroup $\P$ of $\G$ having $\L$ as a Levi complement. In addition, if Deligne--Lusztig induction does not depend on the choice of a parabolic subgroup, then
\[\left\lbrace\chi\in\irr\left(\G^F\right)\hspace{1pt}\middle|\hspace{1pt}(\L,\lambda)\ll_e(\G,\chi)\right\rbrace=\irr\left(\R_\L^\G(\lambda)\right),\]
where we recall that, for any finite group $X$ and $\chi\in\mathbb{Z}\irr(X)$, we denote by $\irr(\chi)$ the set of irreducible constituent of $\chi$. Because this remark is used multiple times in Section \ref{sec:Brauer-Lusztig blocks and e-HC series}, we introduce the following condition.

\begin{cond}
\label{cond:Transitivity}
Consider $\G$, $F$, $q$, $\ell$ and $e$ as in Notation \ref{notation} and assume that Deligne--Lusztig induction does not depend on the choice of parabolic subgroups and
\[\left\lbrace\kappa\in\irr\left(\K^F\right)\hspace{1pt}\middle|\hspace{1pt}(\L,\lambda)\ll_e(\K,\kappa)\right\rbrace=\irr\left(\R_\L^\K(\lambda)\right)\]
for every $F$-stable Levi subgroup $\K$ of $\G$ and every $(e,\ell')$-cuspidal pair $(\L,\lambda)$ of $\K$.
\end{cond}

Observe that Conjecture \ref{conj:Transitive closure} is known for $(e,1)$-pairs by \cite[3.11]{Bro-Mal-Mic93} while Condition \ref{cond:Transitivity} has been proved for $\G$ simple of exceptional simply connected type and good primes in \cite[Theorem 1.1]{Hol22}. Exceptional simple groups and bad primes have been considered in \cite[Theorem 1.4]{Kes-Mal13}. Moreover Condition \ref{cond:Transitivity} is known to hold for groups with connected centre and good primes $\ell\geq 5$ by \cite[Proposition 2.2.4]{Eng13}. Proposition \ref{prop:e-HC theory for connected reductive} below extends these results and shows that Condition \ref{cond:Transitivity} holds for every connected reductive group $\G$ with $[\G,\G]$ simply connected and good primes $\ell\geq 5$. In the next section we extend this result to $e$-pairs associated with $\ell$-singular semisimple elements (see Corollary \ref{cor:Transitive closure}). Moreover, in Proposition \ref{prop:Transitivity in good series} we prove Conjecture \ref{conj:Transitive closure} inside $e$-Harish-Chandra series associated to certain semisimple elements. Notice that our proof does not depend on \cite{Eng13} in any way.

\begin{lem}
\label{lem:e-HC theory for connected reductive, reduction to semisimple}
Let $\L$ be an $e$-split Levi subgroup of a connected reductive group $\G$ and consider $\G_0:=[\G,\G]$ and $\L_0:=\L\cap \G_0$. \begin{enumerate}
\item Let $\lambda_0\in\irr(\L_0^F)$ and $\chi_0\in\irr(\G_0^F)$. If $(\L_0,\lambda_0)\leq_e(\G_0,\chi_0)$ and $\chi\in\irr(\G^F\mid \chi_0)$, then there exists $\lambda\in\irr(\L^F\mid \lambda_0)$ such that $(\L,\lambda)\leq_e(\G,\chi)$.

\item Let $\lambda\in\irr(\L^F)$ and $\chi\in\irr(\G^F)$. If $(\L,\lambda)\leq_e(\G,\chi)$ and $\lambda_0\in\irr(\lambda_{\L_0^F})$, then there exists $\chi_0\in\irr(\chi_{\G_0^F})$ such that $(\L_0,\lambda_0)\leq_e(\G_0,\chi_0)$.
\end{enumerate}
\end{lem}

\begin{proof}
First observe that $\L_0$ is an $e$-split Levi subgroup of $\G_0$. By \cite[Proposition 3.3.24]{Gec-Mal20} (see also the proof of \cite[Corollary 3.3.25]{Gec-Mal20}) and since $\G=\z^\circ(\G)\G_0$, it follows that
\begin{equation}
\label{eq:e-HC theory for connected reductive, reduction to semisimple, 1}
\R_\L^\G\circ {\rm Ind}_{\L_0^F}^{\L^F}={\rm Ind}_{\G_0^F}^{\G^F}\circ \R_{\L_0}^{\G_0}
\end{equation}
and
\begin{equation}
\label{eq:e-HC theory for connected reductive, reduction to semisimple, 2}
{^*\R}_{\L_0}^{\G_0}\circ {\rm Res}_{\G_0^F}^{\G^F}={\rm Res}^{\L^F}_{\L_0^F}\circ {^*\R}_\L^\G.
\end{equation}
Suppose first that $(\L_0,\lambda_0)\leq_e(\G_0,\chi_0)$ and consider $\chi\in\irr(\G_0^F\mid \chi_0)$. Then $\chi$ is an irreducible constituent of ${\rm Ind}_{\G_0^F}^{\G^F}(\R_{\L_0}^{\G_0}(\lambda_0))$ and by \eqref{eq:e-HC theory for connected reductive, reduction to semisimple, 1} we can find $\lambda\in\irr(\L^F\mid \lambda_0)$ such that $(\L,\lambda)\leq_e(\G,\chi)$.

Suppose now that $(\L,\lambda)\leq_e(\G,\chi)$ and let $\lambda_0$ be an irreducible constituent of $\lambda_{\L_0^F}$. Since Deligne--Lusztig induction and restriction are adjoint with respect to the usual scalar product, we deduce that $\lambda_0$ is an irreducible constituent of ${\rm Res}_{\L_0^F}^{\L^F}({^*\R}_\L^\G(\chi))$. By \eqref{eq:e-HC theory for connected reductive, reduction to semisimple, 2} there exists $\chi_0\in\irr(\chi_{\G_0^F})$ such that $\lambda_0$ is a constituent of ${^*\R}_{\L_0}^{\G_0}(\chi_0)$ and therefore $(\L_0,\lambda_0)\leq_e(\G_0,\chi_0)$.
\end{proof}

The following result shows that Condition \ref{cond:Transitivity} holds when $\G$ has only components of classical types and $\ell\geq 5$ or when $\G$ is simple, $\K=\G$ and $\lambda$ lies in a rational Lusztig series associated with a quasi-isolated element. Recall that a semisimple element $s$ of a reductive group $\G$ is called \emph{quasi-isolated} if $\c_\G(s)$ is not contained in any proper Levi subgroup of $\G$.

\begin{lem}
\label{lem:e-HC theory for connected reductive, quasi-isolated}
Let $\G$ be connected reductive, $\chi\in\irr(\G^F)$ and consider an $e$-cuspidal pair $(\L,\lambda)\ll_e(\G,\chi)$, where $\lambda\in\E(\L^F,[s])$ for some $\ell$-regular semisimple element $s\in\L^{*F^*}$. Suppose that $\ell\geq 5$ is good for $\G$ and that the Mackey formula holds for $(\G,F)$. If either $\G$ has only components of classical types and $F$ does not induce the triality automorphism on components of type $\mathbf{D}_4$ or $\G$ is simple and $s$ is quasi-isolated in $\G^*$, then $(\L,\lambda)\leq_e(\G,\chi)$.
\end{lem}

\begin{proof}
Consider a regular embedding $i:\G\to \wt{\G}$. By applying \cite[3.11]{Bro-Mal-Mic93} together with \cite[Theorem 4.7.2 and Corollary 4.7.8]{Gec-Mal20} %(notice that there is a typo in \cite[Corollary 4.7.8]{Gec-Mal20} and the hypothesis should be as in \cite[Corollary 4.7.7]{Gec-Mal20})
to $\wt{\G}$, it follows that Conjecture \ref{conj:Transitive closure} holds in $\wt{\G}$ unless $s$ is quasi-isolated in $\G$ and $\G$ is simple of simply connected type $\mathbf{E}_6$ or $\mathbf{E}_7$ or $\G^F={^3\mathbf{D}}_4(q)$. However, in these excluded cases the result holds by \cite[Theorem 1.1]{Hol22} and we can therefore assume that Conjecture \ref{conj:Transitive closure} holds in $\wt{\G}$. Now, the result follows by applying \cite[Proposition 5.2]{Cab-Eng99}.
\end{proof}

We can now prove the main result of this section. Recall that for a connected reductive group $\G$, we say that $\G$ is simply connected if the semisimple group $[\G,\G]$ is simply connected.

\begin{prop}
\label{prop:e-HC theory for connected reductive}
Let $\G$ be a simply connected reductive group, $\chi\in\irr(\G^F)$ and consider an $(e,\ell')$-cuspidal pair $(\L,\lambda)\ll_e(\G,\chi)$. If $\ell\geq 5$ is good for $\G$ and the Mackey formula holds for $(\G,F)$, then $(\L,\lambda)\leq_e(\G,\chi)$.
\end{prop}

\begin{proof}
Let $(\G^*,F^*)$ be dual to $(\G,F)$ and let $\L^*$ be the $e$-split Levi subgroup of $\G^*$ corresponding to $\L$. Consider an $\ell$-regular semisimple element $s\in\L^{*F^*}$ such that $\lambda\in\E(\L^F,[s])$ and notice that $\chi\in\E(\G^F,[s])$ because $(\L,\lambda)\ll_e(\G,\chi)$ (see \cite[Proposition 15.7]{Cab-Eng04}). By induction on $\dim(\G)$, we show that $s$ is quasi-isolated in $\G^*$. Suppose that $\G_1$ is a proper $F$-stable Levi subgroup of $\G$ such that $\c_{\G^*}(s)\leq \G_1^*$. Observe that $\G_1$ is simply connected by \cite[Proposition 12.14]{Mal-Tes}. Set $\L_1^*:=\c_{\G_1^*}(\z^\circ(\L^*)_{\Phi_e})=\c_{\G^*}(\z^\circ(\L^*)_{\Phi_e})\cap \G_1^*=\L^*\cap \G_1^*$ and notice that its dual $\L_1\leq \L$ is an $e$-split Levi subgroup of $\G_1$ and that $\c_{\L^*}(s)\leq \L^*\cap \G_1^*=\L_1^*$. By \cite[Theorem 8.27]{Cab-Eng04} there exist unique $\lambda_1\in\E(\L_1^F,[s])$ and $\chi_1\in\E(\G_1^F,[s])$ such that $\lambda=\pm\R_{\L_1}^\L(\lambda_1)$ and $\chi=\pm\R_{\G_1}^\G(\chi_1)$. Since $(\L,\lambda)\ll_e(\G,\chi)$, it follows by the transitivity of Deligne--Lusztig induction that $(\L_1,\lambda_1)\ll_e(\G_1,\chi_1)$. A similar argument also shows that $\lambda_1$ is $e$-cuspidal. Since $\dim(\G_1)<\dim(\G)$, we obtain $(\L_1,\lambda_1)\leq_e(\G_1,\chi_1)$. This shows that $\chi_1$ is an irreducible constituent of $\R_{\L_1}^{\G_1}(\lambda_1)$ and, because all constituents of $\R_{\L_1}^{\G_1}(\lambda_1)$ are contained in $\E(\G_1^F,[s])$ and $\R_{\G_1}^{\G}$ induces a bijection between $\E(\G_1^F,[s])$ and $\E(\G^F,[s])$ (see \cite[Theorem 8.27]{Cab-Eng04}), we conclude that $\chi$ is an irreducible constituent of $\pm\R_{\G_1}^\G(\R_{\L_1}^{\G_1}(\lambda_1))=\pm\R_\L^\G(\lambda)$. Hence $(\L,\lambda)\leq_e(\G,\chi)$ and we may assume that $s$ is quasi-isolated in $\G^*$.

Let $\G_0:=[\G,\G]$ and $\L_0:=\L\cap \G_0$. By assumption, there exist $e$-split Levi subgroups $\L_i$ of $\G$ containing $\L$ and characters $\lambda_i\in\irr(\L_i^F)$ such that $(\L,\lambda)\leq_e(\L_1,\lambda_1)\leq_e\cdots\leq_e(\G,\chi)$. If we define $\L_{i,0}:=\L_i\cap\G_0$, then a repeated application of Lemma \ref{lem:e-HC theory for connected reductive, reduction to semisimple} yields characters $\lambda_0\in\irr(\lambda_{\L_0^F})$, $\lambda_{i,0}\in\irr(\lambda_{i,\L_{i,0}^F})$ and $\chi_0\in\irr(\chi_{\G_0^F})$ such that $(\L_0,\lambda_0)\leq_e(\L_{1,0},\lambda_{1,0})\leq_e\cdots \leq_e(\G_0,\chi_0)$. Then $(\L_0,\lambda_0)\ll_e(\G_0,\chi_0)$ with $(\L_0,\lambda_0)$ an $(e,\ell')$-cuspidal pair. Moreover, if the result is true for $\G_0$, then $(\L_0,\lambda_0)\leq_e(\G_0,\chi_0)$ and using Lemma \ref{lem:e-HC theory for connected reductive, reduction to semisimple} we find $\lambda'\in\irr(\L^F\mid \lambda_0)$ such that $(\L,\lambda')\leq_e(\G,\chi)$. Then \cite[Theorem 4.1]{Cab-Eng99} shows that $\lambda'^g=\lambda$, for some $g\in \n_\G(\L)^F$, and hence $(\L,\lambda)=(\L,\lambda')^g\leq_e(\G,\chi)^g=(\G,\chi)$. Notice that the inclusion $\G_0\to\G$ induces a dual morphism $\G^*\to\G_0^*$ and that, if $s\in\G^{*F^*}_{\rm ss}$ is quasi-isolated in $\G^*$, then the corresponding element $s_0\in\G^{*F^*}_{0,\rm ss}$ is quasi-isolated in $\G_0^*$ by \cite[Proposition 2.3]{Bon05}. Without loss of generality we can hence assume $\G=[\G,\G]$.

Now, $\G$ is a direct product of simple algebraic groups $\H_1,\dots, \H_n$ (see \cite[Proposition 1.4.10]{Mar91}). The action of $F$ induces a permutation on the set of simple components $\H_i$. For every orbit of $F$ we denote by $\G_j$, $j=1,\dots, t$, the direct product of the simple components in such an orbit. Then $\G_j$ is $F$-stable and
\[\G^F=\G_1^F\times \cdots\times \G_t^F.\]
Define $\L_j:=\L\cap \G_j$ and observe that $\L_j$ is an $e$-split Levi subgroup of $\G_j$ and that
\[\L^F=\L_1^F\times \cdots \times\L_j^F.\]
Then we can write $\chi=\chi_1\times \cdots \times\chi_t$ and $\lambda=\lambda_1\times\cdots\times \lambda_t$, with $\chi_j\in\irr(\G_j^F)$ and $\lambda_j\in\irr(\L_j^F)$. Since $\R_\L^\G=\R_{\L_1}^{\G_1}\times\cdots\times \R_{\L_t}^{\G_t}$ (see \cite[Proposition 10.9 (ii)]{Dig-Mic91}), eventually considering intermediate $e$-split Levi subgroups, the fact that $(\L,\lambda)\ll_e(\G,\chi)$ implies that $(\L_j,\lambda_j)\ll_e(\G_j,\chi_j)$ for every $j$. Noticing that $\G^{*F^*}=\G_1^{*F^*}\times \cdots\times \G_t^{*F^*}$, we can write $s=s_1\times \cdots\times s_t$ for some $\ell$-regular semisimple elements $s_j\in\G_j^{*F^*}$. Moreover, since $s$ is quasi-isolated in $\G^*$, it follows that $s_j$ is quasi-isolated in $\G_j^*$. Without loss of generality, we may thus assume that $F$ is transitive on the set of simple components $\H_i$ or equivalently that $t=1$. 

Now, consider a simple component $\H$ of $\G$ and observe that there are isomorphisms
\begin{equation}
\label{eq:e-HC theory for connected reductive}
\G^F\simeq \H^{F^n}
\end{equation}
and
\begin{equation}
\label{eq:e-HC theory for connected reductive 2}
\G^{*F^*}\simeq \H^{*{F^*}^n}
\end{equation}
where $n$ is the number of simple components $\H_i$ of $\G$. Let $\M:=\L\cap \H$ and notice that $\M$ is an $e$-split Levi subgroup of $(\H,F^n)$ and that the isomorphism from \eqref{eq:e-HC theory for connected reductive} restricts to an isomorphism
\begin{equation}
\label{eq:e-HC theory for connected reductive 3}
\L^F\simeq \M^{F^n}.
\end{equation}
Let $\psi\in\irr(\H^{F^n})$ correspond to $\chi\in\irr(\G^F)$ via the isomorphism \eqref{eq:e-HC theory for connected reductive} and similarly $\mu\in\irr(\M^{F^n})$ correspond to $\lambda\in\irr(\L^F)$ via the isomorphism \eqref{eq:e-HC theory for connected reductive 3}. Since $(\L,\lambda)<<_e(\G,\chi)$, we deduce that $(\M,\mu)<<_e(\H,\psi)$. Moreover, as $s$ is quasi-isolated in $\G^*$, it follows that the semisimple element $t\in\H^{*{F^*}^n}$ obtained via the isomorphism \eqref{eq:e-HC theory for connected reductive 2} is quasi-isolated in $\H^*$. Finally, Lemma \ref{lem:e-HC theory for connected reductive, quasi-isolated} implies that $(\M,\mu)\leq_e(\H,\psi)$ and we hence conclude that $(\L,\lambda)\leq_e(\G,\chi)$.
\end{proof}

Since the hypotheses of the above proposition are inherited by Levi subgroups, it follows that Condition \ref{cond:Transitivity} holds whenever $\G$ is a simply connected reductive group.

\begin{cor}
\label{cor:e-HC theory for connected reductive}
Let $\G$ be a simply connected reductive group. Then Condition \ref{cond:Transitivity} holds for every $F$-stable Levi subgroup $\K$ of $\G$ and every $(e,\ell')$-cuspidal pair $(\L,\lambda)$ of $\K$.
\end{cor}

\section{Brauer--Lusztig blocks and $e$-Harish-Chandra series}
\label{sec:Brauer-Lusztig blocks and e-HC series}

In this section we prove Theorem \ref{thm:Main Brauer--Lusztig blocks} and hence obtain a description of the distribution of characters into blocks for finite reductive groups in non-defining characteristic under suitable assumptions on $\ell$.

As already mentioned in Section \ref{sec:Deligne-Lusztig induction and blocks}, under certain assumptions on $\ell$, the results of \cite[Theorem 4.1]{Cab-Eng99} show that for every block $B\in\Bl(\G^F)$ there exists a unique $\G^F$-conjugacy class of $e$-cuspidal pairs $(\L,\lambda)$ such that $\lambda\in\E(\L^F,[s])$ for some $\ell$-regular semisimple element $s\in\L^{*F^*}$ and every irreducible constituent of $\R_{\L\leq \P}^\G(\lambda)$ is contained in $\irr(B)$ for every parabolic subgroup $\P$ of $\G$ containing $\L$ as Levi complement. \cite[Theorem 4.1]{Cab-Eng99} also provides a characterisation of the set of so-called $\ell'$-characters in the block $B$ as
\begin{equation}
\label{eq:Cabanes-Enguehard and Brauer-Lusztig blocks}
\E\left(\G^F,\ell'\right)\cap \irr(B)=\left\lbrace\chi\in\irr\left(\G^F\right)\hspace{1pt}\middle|\hspace{1pt} (\L,\lambda)\ll_e (\G,\chi)\right\rbrace.
\end{equation}
On the other hand, \cite[Theorem 2.2]{Bro-Mic89} shows that
\[\irr(B)=\coprod\limits_{t}\E\left(\G^F,B,[st]\right),\]
where $t$ runs over the elements of $\c_{\G^*}(s)_\ell^{F^*}$ up to conjugation and $\E(\G^F,B,[st])$ is a Brauer--Lusztig block (see Definition \ref{def:Brauer-Lusztig blocks}). In particular, in order to obtain all the characters in $\irr(B)$, we have to describe the Brauer--Lusztig blocks $\E(\G^F,B,[st])$. Now, by using Corollary \ref{cor:e-HC theory for connected reductive}, the equality \eqref{eq:Cabanes-Enguehard and Brauer-Lusztig blocks} can be restated as
\[\E(\G^F,B,[s])=\E(\G^F,(\L,\lambda))\]
showing that those Brauer--Lusztig blocks associated with $\ell$-regular semisimple elements coincide with $e$-Harish-Chandra series. Our aim is to provide a similar description for arbitrary Brauer--Lusztig blocks and remove the restriction on $s$ being $\ell$-regular. 

\subsection{$e$-Harish-Chandra series and $\ell$-blocks}
\label{sec:e-Harish-Chandra and blocks}

Throughout this section we assume the following conditions.

\begin{hyp}
\label{hyp:Brauer--Lusztig blocks}
Let $\G$, $F:\G\to \G$, $q$, $\ell$ and $e$ be as in Notation \ref{notation}. Assume that:
\begin{enumerate}
\item $\ell\in\Gamma(\G,F)$ with $\ell\geq 5$ and the Mackey formula hold for $(\G,F)$;
\item Condition \ref{cond:Transitivity} holds for $(\G,F)$.
\end{enumerate}
\end{hyp}

Observe that the Mackey formula and Condition \ref{cond:Transitivity} are expected to hold for any connected reductive group. Moreover, it follows from Remark \ref{rmk:Very good primes go to Levi subgroups} that Hypothesis \ref{hyp:Brauer--Lusztig blocks} is inherited by $F$-stable Levi subgroups.

%In the following remark we show that Hypothesis \ref{hyp:Brauer--Lusztig blocks} is satisfied in most of the cases we are interested in.

\begin{rmk}
\label{rmk:Brauer-Lusztig for simple simply connected}
Suppose that $[\G,\G]$ is of simply connected type and has no irreducible rational component of type ${^2 \mathbf{E}}_6(2)$, $\mathbf{E}_7(2)$, $\mathbf{E}_8(2)$ and consider $\ell\in\Gamma(\G,F)$ with $\ell\geq 5$. Then Hypothesis \ref{hyp:Brauer--Lusztig blocks} is satisfied. In fact, under our assumption, the Mackey formula holds by \cite{Bon-Mic10} while Condition \ref{cond:Transitivity} holds by Corollary \ref{cor:e-HC theory for connected reductive}.
\end{rmk}

%Before proceeding further, we prove the following result on the intersection of $E$-split Levi subgroups.

%\begin{lem}
%\label{lem:Intersection of Levi subgroups}
%Consider a set of positive integers $E$. Let $\L_1$ and $\L_2$ be two $E$-split Levi subgroups of $\G$ containing a common $F$-stable maximal torus $\T$. Then $\L_1\cap \L_2$ is an $E$-split Levi subgroup of $\G$.}}
%\end{lem}

%\begin{proof}
%For $i=1,2$, let $\S_i$ be a $\Phi_E$-torus of $\G$ such that $\L_i=\c_\G(\S_i)$. Notice that $\S_i\leq \z^\circ(\L_i)\leq \T$. Then $\S_i\leq \T_{\Phi_E}$. Moreover, as $\T$ is abelian, we deduce that $\S:=\S_1\S_2$ is a subgroup of $\T$. Since $\S$ is connected it follows that $\S$ is a torus contained in $\T$. By \cite[Proposition 13.2]{Cab-Eng04} it follows that $\S$ is a $\Phi_E$-torus and therefore $\L:=\c_\G(\S)$ is an $E$-split Levi subgroup of $\G$. To conclude, observe that $\L=\L_1\cap \L_2$.
%\end{proof}

We now start working towards a proof of Theorem \ref{thm:Main Brauer--Lusztig blocks}. First, we show how to associate to every $(e,s)$-pair an $(e,s_{\ell'})$-pair via Jordan decomposition. This can be used to extend some of the results of \cite{Cab-Eng99} from $(e,\ell')$-pairs to arbitrary $e$-pairs.

\begin{lem}
\label{lem:Jordan decomposition for l-elements}
Assume Hypothesis \ref{hyp:Brauer--Lusztig blocks} (i). If $(\L,\lambda)$ is an $(e,s)$-pair of $\G$ with $s\in\L^*$, then there exists an $E_{q,\ell}$-split Levi subgroup $\G(s_\ell)$ of $\G$, an $(e,s_{\ell'})$-pair $(\L(s_\ell),\lambda(s_\ell))$ of $\G(s_\ell)$ and a linear character $\wh{s_\ell}$ of $\L(s_\ell)^F$ such that $\lambda=\epsilon_\L\epsilon_{\L(s_\ell)}\R_{\L(s_\ell)}^\L(\lambda(s_\ell)\cdot\wh{s_\ell})$. In particular $\L(s_\ell)$ is an $E_{q,\ell}$-split Levi subgroup of $\G$.
\end{lem}

\begin{proof}
Under our assumptions, Lemma \ref{prop:Good primes} (iii) implies that $\c_{\G^*}^\circ(s_\ell)$ is an $E_{q,\ell}$-stable Levi subgroup of $\G^*$. By \cite[Proposition 2.12]{Hol22} we know that $\c_{\L^*}^\circ(s_\ell)$ is an $e$-split Levi subgroup of $\c_{\G^*}^\circ(s_\ell)$. As $\ell\in\Gamma(\G,F)$, Remark \ref{rmk:Very good primes go to Levi subgroups} implies that $\ell\in\Gamma(\L^*,F^*)$ and therefore $\c_{\L^*}^\circ(s_\ell)^F=\c_{\L^*}(s_\ell)^F$ by Lemma \ref{prop:Good primes} (i). Recalling that $s_\ell$ is a power of $s$, it follows that $\c_{\L^*}^\circ(s)\c_{\L^*}(s)^F\subseteq\c_{\L^*}^\circ(s_\ell)\c_{\L^*}(s_\ell)^F=\c_{\L^*}^\circ(s_\ell)$. Let $\G(s_\ell)$ be an $E_{q,\ell}$-split Levi subgroup of $\G$ in duality with $\c_{\G^*}^\circ(s_\ell)$ and $\L(s_\ell)$ an $e$-split Levi subgroup of $\G(s_\ell)$ in duality with $\c_{\L^*}^\circ(s_\ell)$. By \cite[Proposition 8.26]{Cab-Eng04} and \cite[Theorem 8.27]{Cab-Eng04} there exists a unique character $\lambda(s_\ell)\in\E(\L(s_\ell)^F,[s_{\ell'}])$ such that
\[\lambda=\epsilon_\L\epsilon_{\L(s_\ell)}\R_{\L(s_\ell)}^\L\left(\wh{s_\ell}\cdot \lambda(s_\ell)\right),\]
where $\wh{s_\ell}$ is the linear character corresponding to $s_\ell\in\z(\c_{\L^*}(s_\ell)^{F^*})$. To conclude, notice that $\L(s_\ell)$ is an $E_{q,\ell}$-split Levi subgroup of $\G(s_\ell)$. Since $\G(s_\ell)$ is $E_{q,\ell}$-split in $\G$ we deduce that $\L(s_\ell)$ is $E_{q,\ell}$-split in $\G$.
\end{proof}

Next, we show that the relation $\ll_e$ is preserved under the construction of Lemma \ref{lem:Jordan decomposition for l-elements}. Moreover we consider $e$-cuspidality and a property related to the \emph{Jordan criterion} (J) introduced by Cabanes and Enguehard (see \cite[Proposition 1.10]{Cab-Eng99}). The final part of the following statement could be seen as a partial converse of \cite[Proposition 1.10]{Cab-Eng99}.

\begin{lem}
\label{lem:Jordan decomposition for l-elements, order relation}
Assume Hypothesis \ref{hyp:Brauer--Lusztig blocks} (i). Let $(\L,\lambda)$ and $(\K,\kappa)$ be two $(e,s)$-pairs and consider the corresponding $(e,s_{\ell'})$-pairs $(\L(s_\ell),\lambda(s_\ell))$ and $(\K(s_\ell),\kappa(s_\ell))$ given by Lemma \ref{lem:Jordan decomposition for l-elements}. Then $(\L,\lambda)\leq_e(\K,\kappa)$ if and only if $(\L(s_\ell),\lambda(s_\ell))\leq_e(\K(s_\ell),\kappa(s_\ell))$. In particular $(\L,\lambda)\ll_e(\K,\kappa)$ if and only if $(\L(s_\ell),\lambda(s_\ell))\ll_e(\K(s_\ell),\kappa(s_\ell))$. Moreover $(\L,\lambda)$ is $e$-cuspidal in $\G$ if and only if $(\L(s_\ell),\lambda(s_\ell))$ is $e$-cuspidal in $\G(s_\ell)$ and $\z^\circ(\L)_{\Phi_e}=\z^\circ(\L(s_\ell))_{\Phi_e}$.
\end{lem}

\begin{proof}
First, using \cite[Proposition 2.12]{Hol22} notice that every $e$-split Levi subgroup of $\G(s_\ell)$ is of the form $\M(s_\ell)$ for some $e$-split Levi subgroup $\M$ of $\G$. Then, since $\ll_e$ is a pre-order on a finite poset, proceeding by induction it is enough to show that $(\L,\lambda)\leq_e(\K,\kappa)$ if and only if $(\L(s_\ell),\lambda(s_\ell))\leq_e(\K(s_\ell),\kappa(s_\ell))$.

Suppose first that $(\L,\lambda)\leq_e(\K,\kappa)$ and assume without loss of generality that $s\in\L^*$. We know that $\kappa$ is an irreducible constituent of $\R_{\L}^{\K}(\lambda)$. By the transitivity of Deligne--Lusztig induction (see \cite[Theorem 3.3.6]{Gec-Mal20}), we have
\[\R_{\L}^{\K}(\lambda)=\epsilon_{\L}\epsilon_{\L(s_\ell)}\R_{\L(s_\ell)}^{\K}\left(\wh{s_\ell}\cdot\lambda(s_\ell)\right)=\epsilon_\L\epsilon_{\L(s_\ell)}\R_{\K(s_\ell)}^{\K}\left(\R_{\L(s_\ell)}^{\K(s_\ell)}\left(\wh{s_\ell}\cdot\lambda(s_\ell)\right)\right).\]
Moreover, by \cite[Proposition 15.7]{Cab-Eng04}, every irreducible constituent of $\R_{\L(s_\ell)}^{\K(s_\ell)}(\wh{s_\ell}\cdot\lambda(s_\ell))$ is contained in $\E(\K(s_\ell)^F,[s])$. Then, since
\[\epsilon_{\K}\epsilon_{\K(s_\ell)}\R_{\K(s_\ell)}^{\K}:\E(\K(s_\ell)^F,[s])\to \E(\K^F,[s])\]
is a bijection, we deduce that $\wh{s_\ell}\cdot\kappa(s_\ell)$ is an irreducible constituent of $\R_{\L(s_\ell)}^{\K(s_\ell)}(\wh{s_\ell}\cdot \lambda(s_\ell))$. It follows that $\kappa(s_\ell)$ is an irreducible constituent of $\R_{\L(s_\ell)}^{\K(s_\ell)}(\lambda(s_\ell))$ (see \cite[10.2]{Bon06}). The reverse implication follows from a similar argument.

Next suppose that $(\L,\lambda)$ is $e$-cuspidal in $\G$ and let $(\M_{s_\ell},\mu_{s_\ell})\ll_e(\L(s_\ell),\lambda(s_\ell))$. By \cite[Proposition 2.12]{Hol22} we can find an $e$-split Levi subgroup $\M$ of $\G$ such that $\M_{s_\ell}=\M(s_\ell)$ and we define $\mu:=\epsilon_{\M(s_\ell)}\epsilon_{\M}\R_{\M(s_\ell)}^\M(\mu_{s_\ell}\wh{s_\ell})$. Then $(\M_{s_\ell},\mu_{s_\ell})=(\M(s_\ell),\mu(s_\ell))$ and the above paragraph implies $(\M,\mu)\ll_e(\L,\lambda)$. Since $(\L,\lambda)$ is $e$-cuspidal, we deduce that $(\M,\mu)=(\L,\lambda)$ and therefore $(\M_{s_\ell},\mu_{s_\ell})=(\L(s_\ell),\lambda(s_\ell))$. This shows that $(\L(s_\ell),\lambda(s_\ell))$ is $e$-cuspidal. Furthermore if $\z^\circ(\L)_{\Phi_e}<\z^\circ(\L(s_\ell))_{\Phi_e}$, then $\L(s_\ell)_e:=\c_\G^\circ(\z^\circ(\L(s_\ell))_{\Phi_e})$ is an $e$-split Levi subgroup of $\G$ strictly contained in $\L$ and containing $\L(s_\ell)$. In this case $(\L(s_\ell)_e,\lambda(s_\ell)_e)\leq_e(\L,\lambda)$ for $\lambda(s_\ell)_e:=\epsilon_{\L(s_\ell)}\epsilon_{\L(s_\ell)_e}\R_{\L(s_\ell)}^{\L(s_\ell)_e}(\wh{s_\ell}\cdot\lambda(s_\ell))$ which is a contradiction. Hence $\z^\circ(\L)_{\Phi_e}=\z^\circ(\L(s_\ell))_{\Phi_e}$.

Conversely, assume that $(\L(s_\ell),\lambda(s_\ell))$ is $e$-cuspidal in $\G(s_\ell)$ and that $\z^\circ(\L)_{\Phi_e}=\z^\circ(\L(s_\ell))_{\Phi_e}$. Let $(\M,\mu)\ll_e(\L,\lambda)$ and observe that $(\M(s_\ell),\mu(s_\ell))\ll_e(\L(s_\ell),\lambda(s_\ell))$. Since $(\L(s_\ell),\lambda(s_\ell))$ is $e$-cuspidal, it coincides with $(\M(s_\ell),\mu(s_\ell))$. In particular $\M(s_\ell)=\L(s_\ell)$ and, recalling that $\z^\circ(\M)\leq \z^\circ(\M(s_\ell))$ and hence $\z^\circ(\M)_{\Phi_e}\leq \z^\circ(\M(s_\ell))_{\Phi_e}$, we deduce that $\L=\c_\G^\circ(\z^\circ(\L)_{\Phi_e})=\c_\G^\circ(\z^\circ(\L(s_\ell))_{\Phi_e})=\c_\G^\circ(\z^\circ(\M(s_\ell))_{\Phi_e})\leq \c_\G^\circ(\z^\circ(\M)_{\Phi_e})=\M$. Thus $\L=\M$ and $(\L,\lambda)$ is $e$-cuspidal.
\end{proof}

The argument used in the proofs of Lemma \ref{lem:Jordan decomposition for l-elements} and Lemma \ref{lem:Jordan decomposition for l-elements, order relation} can be used to obtain the following result of independent interest.

\begin{prop}
\label{prop:Transitivity in good series}
Let $s\in\G^{*F^*}$ be a semisimple element such that $o(s)$ is not divisible by bad primes and $(o(s),|\z(\G)^F:\z^\circ(\G)^F|)=1$. Then Conjecture \ref{conj:Transitive closure} holds in $\E(\G^F,[s])$ for every $e\geq 1$.
\end{prop}

\begin{proof}
Let $(\L,\lambda)\ll_e(\K,\kappa)$ be $(e,s)$-pairs. Replacing $s$ with a $\G^F$-conjugate, we can assume that $s\in\L^*$. Since no bad prime divide $o(s)$, a repeated application of \cite[Proposition 13.16 (i)]{Cab-Eng04} shows that $\c_{\G^*}^\circ(s)$ is an $F$-stable Levi subgroup. Moreover, since $(o(s),|\z(\G)^F:\z^\circ(\G)^F|)=1$, it follows that $\c_{\G^*}^\circ(s)^F=\c_{\G^*}(s)^F$ (see \cite[Lemma 11.2.1 (iii)]{Dig-Mic20}). Now proceeding as in Lemma \ref{lem:Jordan decomposition for l-elements} we construct unipotent $e$-pairs $(\L(s),\lambda(s))$ and $(\K(s),\kappa(s))$ of $\G(s)$. Arguing as in Lemma \ref{lem:Jordan decomposition for l-elements, order relation} we deduce from $(\L,\lambda)\ll_e(\K,\kappa)$ that $(\L(s),\lambda(s))\ll_e(\K(s),\kappa(s))$. Applying \cite[Theorem 3.11]{Bro-Mal-Mic93} we get $(\L(s),\lambda(s))\leq_e(\K(s),\kappa(s))$ and again proceeding as in the proof of Lemma \ref{lem:Jordan decomposition for l-elements, order relation} we conclude that $(\L,\lambda)\leq_e(\K,\kappa)$.
\end{proof}

The following lemma is a fundamental ingredient to understand the distribution of characters into blocks. The proof is based on an idea used first in \cite{Cab-Eng94} in order to deal with unipotent blocks. Notice that, if $\ell\in\Gamma(\G,F)$ and $\L$ is an $e$-split Levi subgroup of $\G$, then $\L^F=\c_{\G^F}(Q)$ for some abelian $\ell$-subgroup $Q\leq \G^F$ by Lemma \ref{lem:e-split Levi} together with Lemma \ref{prop:Good primes} (i). Therefore, block induction from $\L^F$ to $\G^F$ is defined by \cite[Theorem 4.14]{Nav98}.

\begin{lem}
\label{lem:Jordan decomposition for l-elements, blocks}
Assume Hypothesis \ref{hyp:Brauer--Lusztig blocks} (i). Let $(\K,\kappa)$ be an $(e,s)$-pair of $\G$ and consider the $(e,s_{\ell'})$-pair $(\K(s_\ell),\kappa(s_\ell))$ given by Lemma \ref{lem:Jordan decomposition for l-elements}. Consider an $(e,s_{\ell'})$-cuspidal pair $(\L,\lambda)$ of $\K(s_\ell)$ such that $\bl(\kappa(s_\ell))=b_{\K(s_\ell)^F}(\L,\lambda)$. Then $\bl(\kappa)=\bl(\lambda)^{\K^F}$.
\end{lem}

\begin{proof}
Since $\K(s_\ell)$ is an $E_{q,\ell}$-split Levi subgroup of $\K$, \cite[Theorem 2.5]{Cab-Eng99} implies that all irreducible constituents of $\R_{\K(s_\ell)}^\K(\kappa(s_\ell))$ are contained in a unique block $b$ of $\K^F$. Moreover, under our assumption, Lemma \ref{lem:e-split Levi} implies that $\K(s_\ell)=\c_\K^\circ(\z(\K(s_\ell))^F_\ell)$ and therefore $b=\bl(\kappa(s_\ell))^{\K^F}$. Similarly $\bl(\kappa(s_\ell))=\bl(\lambda)^{\K(s_\ell)^F}$ and so $b=\bl(\lambda)^{\K^F}$ by the transitivity of block induction.
It remains to show that $b=\bl(\kappa)$. In order to do so, we apply Brauer's second Main Theorem (see \cite[Theorem 5.8]{Cab-Eng04}). Then, it suffices to show that $d^1(\R_{\K(s_\ell)}^\K(\kappa(s_\ell)))$ has an irreducible constituent in $\bl(\kappa)$. By \cite[Proposition 21.4]{Cab-Eng04} and since $\R_{\K(s_\ell)}^{\K}$ and $^*\R_{\K(s_\ell)}^\K$ are adjoint, it follows that
\begin{align*}
d^1\left(\R_{\K(s_\ell)}^\K(\kappa(s_\ell))\right)&=\R_{\K(s_\ell)}^\K\left(d^1(\kappa(s_\ell))\right)
\\
&=\R_{\K(s_\ell)}^\K\left(d^1(\wh{s_\ell}\cdot\kappa(s_\ell))\right)
\\
&=\epsilon_\K\epsilon_{\K(s_\ell)}d^1(\kappa).
\end{align*}
Since by Brauer's second Main Theorem $d^1(\kappa)\in\mathbb{N}\irr(\bl(\kappa))$, the proof is now complete.
\end{proof}

As a corollary we deduce that the construction given in Lemma \ref{lem:Jordan decomposition for l-elements} preserves the decomposition of characters into blocks.

\begin{cor}
\label{cor:Jordan decomposition for l-elements, blocks}
Assume Hypothesis \ref{hyp:Brauer--Lusztig blocks} (i). Let $\L$ be an $e$-split Levi subgroup of $\G$ and consider $s\in\L^{*F^*}_{ss}$. For $i=1,2$, let $\lambda_i\in\E(\L^F,[s])$ and consider $\lambda_i(s_\ell)\in\E(\L(s_\ell)^F,[s_{\ell'}])$ given by Lemma \ref{lem:Jordan decomposition for l-elements}. If $\lambda_1(s_\ell)$ and $\lambda_2(s_\ell)$ are in the same block of $\L(s_\ell)^F$, then $\lambda_1$ and $\lambda_2$ are in the same block of $\L^F$.
\end{cor}

\begin{proof}
Let $c$ be the block of $\L(s_\ell)$ containing $\lambda_1(s_\ell)$ and $\lambda_2(s_\ell)$ and consider an $e$-cuspidal pair $(\M,\mu)$ such that $c=b_{\L(s_\ell)^F}(\M,\mu)$. Then, Lemma \ref{lem:Jordan decomposition for l-elements, blocks} implies that $\bl(\lambda_1)=\bl(\mu)^{\L^F}=\bl(\lambda_2)$.
\end{proof}

%{\color{red}{
%\begin{rmk}
%\label{rmk:Reverse implication and blocks}\marginnote{This is false by Radha's example}
%We believe that the reverse implication of Corollary \ref{cor:Jordan decomposition for l-elements, blocks} also holds. Namely we believe that, if $\lambda_1$ and $\lambda_2$ are in the same block, then $\lambda_1(s_\ell)$ and $\lambda_2(s_\ell)$ are in the same block. We point out that this is true when $\c_{\G^*}^\circ(s_{\ell'})\c_{\G^*}(s_{\ell'})^{F^*}\leq \L(s_\ell)^*$ by results of Broué on perfect isometries (see \cite[Theorem 2.3]{Bro90}). This and related questions will be investigated in an upcoming paper.
%\end{rmk}}}

The next result can be seen as an extension of \cite[Theorem 2.5]{Cab-Eng99} to $(e,s)$-pairs with $s$ not necessarily $\ell$-regular.

\begin{prop}
\label{prop:e-Harish-Chandra series and blocks}
Assume Hypothesis \ref{hyp:Brauer--Lusztig blocks} (i). Let $\K$ be an $e$-split Levi subgroup of $\G$ and $(\L,\lambda)$ an $e$-pair of $\K$. Then there exists a block $b$ of  $\hspace{1pt}$ $\K^F$ such that $\R_\L^\K(\lambda)\in\mathbb{Z}\irr(b)$. Moreover $b=\bl(\lambda)^{\K^F}$.
\end{prop}

\begin{proof}
Let $s\in\L^{*F^*}$ such that $(\L,\lambda)$ is an $(e,s)$-pair. Consider the $(e,s_{\ell'})$-pair $(\L(s_\ell),\lambda(s_\ell))$ given by Lemma \ref{lem:Jordan decomposition for l-elements}. By \cite[Theorem 2.5]{Cab-Eng99}, there exists a block $b(s_\ell)$ of $\K(s_\ell)$ such that $\R_{\L(s_\ell)}^{\K(s_\ell)}(\lambda(s_\ell))\in\mathbb{Z}\irr(b(s_\ell))$. Furthermore $b(s_\ell)=\bl(\lambda(s_\ell))^{\K(s_\ell)^F}$ by Lemma \ref{lem:e-split Levi}. If we denote by $\wh{s_\ell}\cdot b(s_\ell)$ the block of $\K(s_\ell)$ consisting of those characters of the form $\wh{s_\ell}\cdot \xi$, for $\xi\in\irr(b(s_\ell))$, then
\begin{equation}
\label{eq:e-Harish-Chandra seres and blocks, 1}
\R_{\L(s_\ell)}^{\K(s_\ell)}\left(\wh{s_\ell}\cdot \lambda(s_\ell)\right)=\wh{s_\ell}\cdot \R_{\L(s_\ell)}^{\K(s_\ell)}(\lambda(s_\ell))\in\mathbb{Z}\irr\left(\wh{s_\ell}\cdot b(s_\ell)\right).
\end{equation}
By Corollary \ref{cor:Jordan decomposition for l-elements, blocks} and \eqref{eq:e-Harish-Chandra seres and blocks, 1} it follows that there exists a unique block $b$ of $\K^F$ such that
\[\R_\L^\K(\lambda)=\R_\L^\K\left(\R_{\L(s_\ell)}^\L\left(\wh{s_\ell}\cdot \lambda(s_\ell)\right)\right)=\R_{\L(s_\ell)}^\L\left(\R_{\L(s_\ell)}^{\K(s_\ell)}\left(\wh{s_\ell}\cdot \lambda(s_\ell)\right)\right)\in\mathbb{Z}\irr(b).\]
Next, set $c:=\bl(\lambda(s_\ell))$. Consider an $(e,\ell')$-cuspidal pair $(\M,\mu)$ such that $c=b_{\L(s_\ell)^F}(\M,\mu)$. Since $c=\bl(\mu)^{\L(s_\ell)^F}$ and $b(s_\ell)=c^{\K(s_\ell)^F}$ it follows that
\[b(s_\ell)=c^{\K(s_\ell)^F}=\bl(\mu)^{\K(s_\ell)^F}=b_{\K(s_\ell)^F}(\M,\mu).\]
Now, Lemma \ref{lem:Jordan decomposition for l-elements, blocks} implies that $\bl(\lambda)=\bl(\mu)^{\L^F}$ and that $b=\bl(\mu)^{\K^F}$. We conclude that $b=\bl(\mu)^{\K^F}=(\bl(\mu)^{\L^F})^{\K^F}=\bl(\lambda)^{\K^F}$ and this concludes the proof.
\end{proof}

The next corollary is basically a restatement of Proposition \ref{prop:e-Harish-Chandra series and blocks}.

\begin{cor}
\label{cor:e-Harish-Chandra series and blocks}
Assume Hypothesis \ref{hyp:Brauer--Lusztig blocks} (i), let $\L$ be an $e$-split Levi subgroup of $\G$ and $b_\L$ an $\ell$-block of $\L^F$. Then there exists a unique $\ell$-block $b_\G$ of $\G^F$ such that for every $\lambda\in\irr(b_\L)$, all irreducible constituents of $\R_\L^\G(\lambda)$ lie in $b_\G$. Moreover $b_\G=b_\L^{\G^F}$ via Brauer induction.
\end{cor}

Finally, we show that for every $e$-pair $(\K,\kappa)$ there exists a unique $e$-cuspidal pair $(\L,\lambda)$ up to $\K^F$-conjugation satisfying $(\L,\lambda)\leq_e(\K,\kappa)$. Observe that our next results also extends Proposition \ref{prop:e-HC theory for connected reductive} to $e$-pairs associated with $\ell$-singular semisimple elements.

\begin{prop}
\label{prop:e-Harish-Chandra series, disjointness}
Assume Hypothesis \ref{hyp:Brauer--Lusztig blocks}. Let $(\L,\lambda)\ll_e(\K,\kappa)$ be $e$-pairs of $\G$ such that $(\L,\lambda)$ is $e$-cuspidal. Then $(\L,\lambda)\leq_e(\K,\kappa)$. Moreover, if $(\L',\lambda')$ is another $e$-cuspidal pair of $\G$ satisfying $(\L',\lambda')\ll_e(\K,\kappa)$, then $(\L,\lambda)$ and $(\L',\lambda')$ are $\K^F$-conjugate.
\end{prop}

\begin{proof}
Assume that $\lambda\in\E(\L^F,[s])$ for a semisimple element $s\in\L^{*F^*}$ and consider the $(e,s_{\ell'})$-pairs $(\L(s_\ell),\lambda(s_\ell))$ and $(\K(s_\ell),\kappa(s_\ell))$ given by Lemma \ref{lem:Jordan decomposition for l-elements}. Applying Lemma \ref{lem:Jordan decomposition for l-elements, order relation} together with our assumption, we deduce that $(\L(s_\ell),\lambda(s_\ell))\ll_e(\K(s_\ell),\kappa(s_\ell))$ and that $(\L(s_\ell),\lambda(s_\ell))$ is $e$-cuspidal in $\G(s_\ell)$. Now, Condition \ref{cond:Transitivity} shows that $(\L(s_\ell),\lambda(s_\ell))\leq_e(\K(s_\ell),\kappa(s_\ell))$ and we obtain $(\L,\lambda)\leq_e(\K,\kappa)$ by applying Lemma \ref{lem:Jordan decomposition for l-elements, order relation} one more time.

Next consider another $(e,s)$-cuspidal pair $(\L',\lambda')\ll_e(\K,\kappa)$. Let $\lambda'\in\E(\L',[s'])$ and notice that $s$ and $s'$ are $\K^{*F^*}$-conjugate by \cite[Proposition 15.7]{Cab-Eng04}. Replacing $(\L',\lambda')$ with a $\K^F$-conjugate we may assume that $s=s'$. As before consider the $(e,s_{\ell'})$-cuspidal pair $(\L'(s_\ell),\lambda'(s_\ell))$ and observe that $(\L'(s_\ell),\lambda'(s_\ell))\ll_e(\K(s_\ell),\kappa(s_\ell))$. By \cite[Theorem 4.1]{Cab-Eng99} it follows that $(\L(s_\ell),\lambda(s_\ell))$ and $(\L'(s_\ell),\lambda'(s_\ell))$ are $\K(s_\ell)^F$-conjugate. Since $\wh{s_\ell}$ is $\K(s_\ell)^F$-invariant, we deduce that $(\L(s_\ell),\wh{s_\ell}\cdot\lambda(s_\ell))$ and $(\L'(s_\ell),\wh{s_\ell}\cdot\lambda'(s_\ell))$ are $\K(s_\ell)^F$-conjugate. Write
\begin{equation}
\label{eq:e-Harish-Chandra series, disjointness}
\left(\L(s_\ell),\lambda(s_\ell)\cdot\wh{s}_\ell\right)=\left(\L'(s_\ell),\lambda'(s_\ell)\cdot\wh{s}_\ell\right)^x
\end{equation}
for some $x\in\K(s_\ell)^F$. An application of \cite[Proposition 1.10]{Cab-Eng99} with respect to the $e$-cuspidal characters $\lambda$ of $\L^F$ and $\lambda(s_\ell)\cdot\wh{s}_\ell$ of $\L(s_\ell)^F$ shows that $\z^\circ(\L^*)_{\Phi_e}=\z^\circ(\c_{\L^*}^\circ(s))_{\Phi_e}$ and that $\z^\circ(\c_{\L^*}^\circ(s_\ell))_{\Phi_e}=\z^\circ(\c_{\c_{\L^*}^\circ(s_\ell)}^\circ(s))_{\Phi_e}$ respectively. However $\c_{\L^*}^\circ(s)\leq\c_{\L^*}^\circ(s_\ell)$ and therefore we obtain $\z^\circ(\L^*)_{\Phi_e}=\z^\circ(\c_{\L^*}^\circ(s_\ell))_{\Phi_e}$. This shows that $\L$ is the smallest $e$-split Levi subgroup of $\K$ containing $\L(s_\ell)$. Since $\L'^x$ is an $e$-split Levi subgroup of $\K$ containing $\L'(s_\ell)^x=\L(s_\ell)$ we deduce that $\L'^x=\L$. Now, \eqref{eq:e-Harish-Chandra series, disjointness} implies $(\L,\lambda)=(\L',\lambda')^x$.
\end{proof}

Proposition \ref{prop:e-Harish-Chandra series, disjointness} shows in particular that Conjecture \ref{conj:Transitive closure} holds when working above $e$-cuspidal pairs.

\begin{cor}
\label{cor:Transitive closure}
Assume Hypothesis \ref{hyp:Brauer--Lusztig blocks}. Then the relation $\leq_e$ and $\ll_e$ from Conjecture \ref{conj:Transitive closure} coincide on the subset $\CP_e(\G,F)\times \mathcal{P}_e(\G,F)\subseteq\mathcal{P}_e(\G,F)\times \mathcal{P}_e(\G,F)$.
\end{cor}

\begin{proof}
If $((\L,\lambda),(\K,\kappa))\in\CP_e(\G,F)\times\mathcal{P}_e(\G,F)$, then Proposition \ref{prop:e-Harish-Chandra series, disjointness} implies that $(\L,\lambda)\leq_e(\K,\kappa)$ if and only if $(\L,\lambda)\ll_e(\K,\kappa)$.
\end{proof}

As an immediate consequence of Proposition \ref{prop:e-Harish-Chandra series, disjointness} we deduce that the set $\irr(\K^F)$ is a disjoint union of $e$-Harish-Chandra series. This should be compared with the classical Harish-Chandra theory (see \cite[Corollary 3.1.17]{Gec-Mal20}) and with the analogous result for unipotent characters \cite[Theorem 4.6.20]{Gec-Mal20}. These two results, can be recovered by considering $(1,s)$-pairs and $(e,1)$-pairs respectively.

\begin{cor}
\label{cor:e-Harish-Chandra, disjointness}
Let $\G$ be a connected reductive group with a Frobenius endomorphism $F$ defining an $\mathbb{F}_q$-structure on $\G$ and consider an integer $e\geq 1$. For every $e$-split Levi subgroup $\K$ of $\G$ there is a partition
\[\irr(\K^F)=\coprod\limits_{(\L,\lambda)}\E\left(\K^F,(\L,\lambda)\right),\]
where the union runs over a $\K^F$-transversal in the set of $e$-cuspidal pairs of $\K$, provided that there exists a prime $\ell$ such that Hypothesis \ref{hyp:Brauer--Lusztig blocks} is satisfied with respect to $(\G,F,q,\ell,e)$.
\end{cor}

Observe that although the proof of the above corollary depends on the choice of a certain prime $\ell$, its statement does not. This is due to the fact that our result is obtained as a consequence of $\ell$-modular representation theoretic techniques. Nonetheless, a partition of characters into $e$-Harish-Chandra series is expected to hold without the restrictions considered here.

Next, combining Corollary \ref{cor:e-Harish-Chandra, disjointness} and Proposition \ref{prop:e-Harish-Chandra series and blocks} we can describe all the characters in the blocks of $\K^F$ in terms of $e$-Harish-Chandra series.

\begin{theo}
\label{thm:Blocks are unions of e-HC series}
Assume Hypothesis \ref{hyp:Brauer--Lusztig blocks}. Let $\K$ be an $e$-split Levi subgroup of $\G$ and $b$ a block of $\K^F$. Then
\[\irr\left(b\right)=\coprod\limits_{(\L,\lambda)}\E\left(\K^F,(\L,\lambda)\right),\]
where the union runs over the $\K^F$-conjugacy classes of $e$-cuspidal pairs $(\L,\lambda)$ of $\K$ such that $\bl(\lambda)^{\K^F}=b$.
\end{theo}

\begin{proof}
Proposition \ref{prop:e-Harish-Chandra series and blocks} shows that $\E(\K^F,(\L,\lambda))\subseteq \irr(b)$ for every $e$-cuspidal pair $(\L,\lambda)$ such that $\bl(\lambda)^{\K^F}=b$. On the other hand, if $k\in\irr(b)$, then by Corollary \ref{cor:e-Harish-Chandra, disjointness} there exists an $e$-cuspidal pair $(\L,\lambda)$ of $\K$ such that $k\in\E(\K^F,(\L,\lambda))$. Moreover, applying Proposition \ref{prop:e-Harish-Chandra series and blocks} once more, it follows that $b=\bl(\kappa)=\bl(\lambda)^{\K^F}$. Finally, the union is disjoint by Proposition \ref{prop:e-Harish-Chandra series, disjointness}.
\end{proof}

In the following remark we compare Theorem \ref{thm:Blocks are unions of e-HC series} and \cite[Theorem (iii)]{Cab-Eng94}.

\begin{rmk}
Let $B$ be a unipotent block and consider a unipotent $e$-cuspidal pair $(\L_0,\lambda_0)$ such that $B=b_{\G^F}(\L_0,\lambda_0)$. In this situation, \cite[Theorem (iii)]{Cab-Eng94} shows that every irreducible character of $B$ occurs in some $\R_{\G(s)}^\G(\chi_s\cdot\wh{s})$ where $s\in\G^{*F^*}$ is an $\ell$-element, $\chi_s\in\E(\G(s)^F,[1])$ and $(\L_s,\lambda_s)\leq_e(\G(s),\chi_s)$ for some $e$-cuspidal pair $(\L_s,\lambda_s)$ of $\G(s)$ satisfying $(\L_s,\lambda_s)\sim (\L_0,\lambda_0)$ (see \cite[Definition 3.4]{Cab-Eng94}). As usual, here we denote by $\G(s)$ a Levi subgroup in duality with $\c^\circ_{\G^*}(s)$. This description can be recovered from the partition obtained in Theorem \ref{thm:Blocks are unions of e-HC series} under Hypothesis \ref{hyp:Brauer--Lusztig blocks}. To prove our claim, let $\chi\in\irr(B)$ and notice that by Theorem \ref{thm:Blocks are unions of e-HC series} there exists an $e$-cuspidal pair $(\L,\lambda)$ of $\G$ such that $\bl(\lambda)^{\G^F}=B$ and $\chi\in\E(\G^F,(\L,\lambda))$. If $s\in\L^{*F^*}$ and $\lambda\in\E(\L^F,[s])$, then $s$ is an $\ell$-element by \cite{Bro-Mic89}. Consider the unipotent $e$-pairs $(\L(s),\lambda(s))$ and $(\G(s),\chi(s))$ of $\G(s)$ given by Lemma \ref{lem:Jordan decomposition for l-elements}. By definition, $\chi$ occurs in $\R_{\G(s)}^\G(\chi(s)\cdot \wh{s})$. Next, applying Lemma \ref{lem:Jordan decomposition for l-elements, order relation} we deduce that $(\L(s),\lambda(s))\leq_e(\G(s),\chi(s))$ and that $(\L(s),\lambda(s))$ is $e$-cuspidal in $\G(s)$. By \cite[Proposition 3.5 (i)]{Cab-Eng94} there exists a unipotent $e$-cuspidal pair $(\M,\mu)$ of $\G$ such that $(\L(s),\lambda(s))\sim(\M,\mu)$. Proceeding as at the end of the proof of \cite[Theorem 4.4 (iii)]{Cab-Eng94}, we conclude that $(\M,\mu)$ is $\G^F$-conjugate to $(\L_0,\lambda_0)$ and by replacing $(\L,\lambda)$ with a $\G^F$-conjugate we may assume $(\L(s),\lambda(s))\sim (\L_0,\lambda_0)$ as required by \cite[Theorem (iii)]{Cab-Eng94}.

Conversely, given a unipotent $e$-cuspidal pair $(\L_s,\lambda_s)\leq (\G(s),\chi_s)$ and an irreducible constituent $\chi$ of $\R_{\G(s)}^\G(\chi_s\cdot \wh{s})$ as in \cite[Theorem (iii)]{Cab-Eng94}, we obtain an $e$-cuspidal pair $(\L,\lambda)$ of $\G$ such that $\bl(\lambda)^{\G^F}=B$ and $\chi\in\E(\G^F,(\L,\lambda))$ as follows. Consider $\L:=\c_\G^\circ(\z^\circ(\L_s)_{\Phi_e})$ and define $\lambda:=\epsilon_{\L(s_\ell)}\epsilon_\L\R_{\L(s_\ell)}^\L(\lambda_s\cdot \wh{s})$. Then $(\L_s,\lambda_s)$ coincides with the pair $(\L(s),\lambda(s))$ obtained from $(\L,\lambda)$ as in Lemma \ref{lem:Jordan decomposition for l-elements}. Moreover, $\z^\circ(\L(s))_{\Phi_e}=\z^\circ(\L)_{\Phi_e}$ and since $(\L_s,\lambda_s)$ is $e$-cuspidal in $\G(s)$ Lemma \ref{lem:Jordan decomposition for l-elements, order relation} implies that $(\L,\lambda)$ is an $e$-cuspidal pair of $\G^F$. Under our assumption, $(\G(s),\chi_s)=(\G(s),\chi(s))$ and the relation $(\L_s,\lambda_s)\leq_e(\G(s),\chi_s)$ implies $(\L,\lambda)\leq_e(\G,\chi)$ thanks to Lemma \ref{lem:Jordan decomposition for l-elements, order relation}. This shows that $\chi\in\E(\G^F,(\L,\lambda))$. Finally, since by assumption $\chi\in\irr(B)$ we conclude that $\bl(\lambda)^{\G^F}=B$ by applying Corollary \ref{cor:e-Harish-Chandra series and blocks}.
\end{rmk}

\subsection{Brauer--Lusztig blocks}

We now extend Theorem \ref{thm:Blocks are unions of e-HC series} in order to obtain Theorem \ref{thm:Main Brauer--Lusztig blocks}. To start, following Broué, Fong and Srinivasan, we define the Brauer--Lusztig blocks of $\G^F$. 

\begin{defin}[Broué--Fong--Srinivasan]
\label{def:Brauer-Lusztig blocks}
A \emph{Brauer--Lusztig block} of $\G^F$ is any non-empty set of the form
\[\E\left(\G^F,B,[s]\right):=\E\left(\G^F,[s]\right)\cap \irr(B),\]
where $B$ is a block of $\G^F$ and $s$ is a semisimple element of $\G^{*F^*}$. In this case, we say that $(\G,B,[s])$ is the associated \emph{Brauer--Lusztig triple} of $\G^F$. Moreover, we denote by $\BL(\G,F)$ the set of all Brauer--Lusztig triples of $\G^F$. We also define the set
\[\BL^\vee(\G,F):=\coprod\limits_{\L\leq \G}\BL(\L,F),\]
where $\L$ runs over all $e$-split Levi subgroups of $\G$.
\end{defin}

Next, assume $\ell\in\Gamma(\G,F)$. Recall from the discussion preceding Lemma \ref{lem:Jordan decomposition for l-elements, blocks} that, for every $e$-split Levi subgroup $\L$ of $\G$ and $b\in\Bl(\L^F)$, the Brauer induced block $b^{\G^F}$ is defined. Then, we can introduce a partial order relation on $\BL^\vee(\G,F)$ by defining
\[(\L,b,[s])\leq (\K,c,[t])\]
if $\L\leq \K$, $b^{\K^F}=c$ and the semisimple elements $s$ and $t$ are conjugate by an element of $\K^{*F^*}$. If $(\L,b,[s])$ is a minimal element of the poset $(\BL^\vee(\G,F),\leq)$, then we say that $(\L,b,[s])$ is a \emph{cuspidal} Brauer--Lusztig triple.

In the next lemma we compare the relation $\leq$ on Brauer--Lusztig triples with the relations $\ll_e$ and $\leq_e$ on $e$-pairs.

\begin{lem}
\label{lem:Brauer--Lusztig triples vs pairs for all elements}
Assume Hypothesis \ref{hyp:Brauer--Lusztig blocks}. Let $\L$ and $\K$ be $e$-split Levi subgroups of $\G$ and consider semisimple elements $s\in\L^{*F^*}$ and $t\in\K^{*F^*}$. 
\begin{enumerate}
\item Let $\lambda\in\E(\L^F,b,[s])$ and $\kappa\in\E(\K^F,c,[t])$. If $(\L,\lambda)\ll_e(\K,\kappa)$, then $(\L,b,[s])\leq (\K,c,[t])$.
\item Let $\lambda\in\E(\L^F,b,[s])$. If $(\L,b,[s])$ is cuspidal, then $(\L,\lambda)$ is $e$-cuspidal.
\item If $(\L,b,[s])\leq (\K,c,[t])$, then for every $\lambda\in\E(\L^F,b,[s])$ there exists $\kappa\in\E(\K^F,c,[t])$ such that $(\L,\lambda)\leq_e(\K,\kappa)$.
\end{enumerate}
\end{lem}

\begin{proof}
We start by proving (i). Let $(\L,\lambda)\ll_e(\K,\kappa)$. By \cite[Proposition 15.7]{Cab-Eng04}, we may assume $s=t$ and it is enough to show that $\bl(\lambda)^{\K^F}=\bl(\kappa)$. To see this, choose an $e$-cuspidal pair $(\M,\mu)\ll_e(\L,\lambda)$ and notice that $(\M,\mu)\ll_e(\K,\kappa)$. By Proposition \ref{prop:e-Harish-Chandra series, disjointness} we deduce that $(\M,\mu)\leq_e(\L,\lambda)$ and $(\M,\mu)\leq_e(\K,\kappa)$. Then, Proposition \ref{prop:e-Harish-Chandra series and blocks} implies that $\bl(\lambda)=\bl(\mu)^{\L^F}$ and $\bl(\kappa)=\bl(\mu)^{\K^F}$. By the transitivity of block induction, we conclude that $\bl(\kappa)=\bl(\lambda)^{\K^F}$. This proves (i) and (ii) is an immediate consequence. In fact, if $(\L,b,[s])$ is a cuspidal Brauer--Lusztig triple and we consider an $e$-cuspidal $(\M,\mu)\ll_e(\L,\lambda)$, then (i) shows that $(\M,\bl(\mu),[r])\leq (\L,b,[s])$, where $\mu\in\E(\M^F,[r])$. It follows that $\L=\M$ and that $(\L,\lambda)=(\M,\mu)$ is $e$-cuspidal.

Finally, let $(\L,b,[s])\leq (\K,c,[t])$ and consider $\lambda\in\E(\L^F,b,[s])$. Let $\kappa$ be an irreducible constituent of $\R_\L^\K(\lambda)$ so that $(\L,\lambda)\leq_e(\K,\kappa)$. We need to show that $\kappa\in\E(\K^F,c,[t])$. By \cite[Proposition 15.7]{Cab-Eng04} we have $\kappa\in\E(\K^F,[s])=\E(\K^F,[t])$. Moreover, applying Proposition \ref{prop:e-Harish-Chandra series and blocks}, we obtain $\bl(\kappa)=\bl(\lambda)^{\K^F}=b^{\K^F}=c$. We conclude that $\kappa\in\E(\K^F,c,[t])$.
\end{proof}

Finally, we are able to prove the main result of this section which provides a slightly more general version of Theorem \ref{thm:Main Brauer--Lusztig blocks}.

\begin{theo}
\label{thm:Brauer-Lusztig blocks are unions of e-HC series}
Assume Hypothesis \ref{hyp:Brauer--Lusztig blocks}. Let $(\K,c,[t])\in\BL^\vee(\G,F)$. Then
\begin{equation}
\label{eq:Brauer-Lusztig blocks are unions of e-HC series}
\E\left(\K^F,c,[t]\right)=\coprod\limits_{(\L,\lambda)}\E\left(\K^F,(\L,\lambda)\right),
\end{equation}
where the union runs over the $\K^F$-conjugacy classes of $(e,t)$-cuspidal pairs $(\L,\lambda)$ of $\K$ with $\lambda\in\E(\L^F,[s_\lambda])$ such that $(\L,\bl(\lambda),[s_\lambda])\leq(\K,c,[t])$.
\end{theo}

\begin{proof}
Consider an $e$-cuspidal pair $(\L,\lambda)$ such that $(\L,\bl(\lambda),[s])\leq (\K,c,[t])$, where $s\in\L^{*F^*}_{\rm ss}$ and $\lambda\in\E(\L^F,[s])$. Since $s$ and $t$ are $\K^{*F^*}$-conjugate, \cite[Proposition 15.7]{Cab-Eng04} implies that $\E(\K^F,(\L,\lambda))\subseteq \E(\K^F,[t])$. Moreover, using the fact that $c=\bl(\lambda)^{\K^F}$, Proposition \ref{prop:e-Harish-Chandra series and blocks} shows that the $e$-Harish-Chandra series $\E(\K^F,(\L,\lambda))$ is contained in $\irr(c)$. This shows that the union on the right hand side of \eqref{eq:Brauer-Lusztig blocks are unions of e-HC series} is contained in the Brauer--Lusztig block $\E(\K^F,c,[t])$. Moreover the union is disjoint by Proposition \ref{prop:e-Harish-Chandra series, disjointness}. To conclude, let $\kappa\in\E(\K^F,c,[t])$ and notice that there exists an $e$-cuspidal pair $(\L,\lambda)$ of $\K$ such that $\kappa\in\irr(\R_\L^\K(\lambda))$ by Corollary \ref{cor:e-Harish-Chandra, disjointness}. If $\lambda\in\E(\L^F,[s])$, then $s$ and $t$ are $\K^{*F^*}$-conjugate by \cite[Proposition 15.7]{Cab-Eng04}. Moreover, $c=\bl(\kappa)=\bl(\lambda)^{\K^F}$ by Proposition \ref{prop:e-Harish-Chandra series and blocks}. It follows that $(\L,\bl(\lambda),[s])\leq (\K,c,[t])$.
\end{proof}

%{\color{red}{\marginnote{False by Radha's example. Can we include a counterexample?}We conclude this section with a remark concerning Theorem \ref{thm:Brauer-Lusztig blocks are unions of e-HC series}. Here we have shown that Brauer--Lusztig blocks are disjoint unions of $e$-Harish-Chandra series. However, we believe that there exists a unique (up to $\K^F$-conjugation) $(e,t)$-cuspidal pair $(\L,\lambda)$ such that $(\L,\bl(\lambda),[s_\lambda])\leq (\K,c,[t])$. By \cite[Theorem 4.1]{Cab-Eng99} this is true when $t$ is $\ell$-regular. In particular this would show that the concepts of Brauer--Lusztig block and $e$-Harish-Chandra series coincide, at least under the above restrictions on primes. It can be seen that to prove such a statement it is enough to show that the reverse implication of Corollary \ref{cor:Jordan decomposition for l-elements, blocks} would hold true (see Remark \ref{rmk:Reverse implication and blocks}). These questions will be the subject of an upcoming paper.}}

As we have mentioned before, the results obtained by Cabanes and Enguehard have been extended to all primes by Kessar and Malle in \cite{Kes-Mal15} and the reader might wonder why we are not considering this more general situation. Unfortunately, many of the techniques used in this section fail for bad primes and a different proof needs to be found in this case.

\subsection{Defect zero characters and $e$-cuspidal pairs}

Recall that for an irreducible character $\chi$ of a finite group $G$, the $\ell$-defect of $\chi$ is the non-negative integer $d(\chi)$ defined by $\ell^{d(\chi)}\chi(1)_\ell=|G|_\ell$. Our next result shows a necessary condition for an $e$-cuspidal character of $\G$ to be of $\ell$-defect zero.

\begin{prop}
\label{prop:Minimal Brauer--Lusztig triple general}
Suppose that $\ell$ is large for $\G$ and $(\G,F,e)$-adapted with $\z(\G^*)^{F^*}_\ell=1$. If $\chi$ is an $e$-cuspidal pair of $\G$, then $\chi$ has $\ell$-defect zero. In particular $\chi\in\E(\G^F,\ell')$.
\end{prop}

\begin{proof}
Let $s\in\G^{*F^*}$ such that $\chi\in\E(\G^F,[s])$. By Jordan decomposition (see \cite[Theorem 2.6.22 and Remark 2.6.26]{Gec-Mal20}), $\chi$ corresponds to a unique $\chi(s)\in\E(\c_{\G^*}(s)^{F^*},1)$ lying over some unipotent character $\chi^\circ(s)\in\E(\c_{\G^*}^\circ(s)^{F^*},1)$. Notice that
\[\chi(1)_\ell=\dfrac{|\G^F|_\ell}{|\c_{\G^*}(s)^{F^*}|_\ell}\chi(s)(1)_\ell.\]
Since $\ell$ is large for $\G$ we deduce that $\ell$ does not divide $|\z(\G^*)^{F^*}:\z^\circ(\G^*)^{F^*}|$ by \cite[Lemma 2.1]{Mal14}. Now, since $|\c_{\G^*}(s)^F:\c_{\G^*}^\circ(s)^F|$ divides $|\z(\G^*)^{F^*}:\z^\circ(\G^*)^{F^*}|$ by \cite[Lemma 11.2.1 (iii)]{Dig-Mic20}, Clifford's theorem implies
\begin{equation}
\label{eq:Minimal Brauer--Lusztig triple general 1}
\chi(1)_\ell=\dfrac{|\G^F|_\ell}{|\c_{\G^*}^\circ(s)^{F^*}|_\ell}\chi^\circ(s)(1)_\ell.
\end{equation}
Set $\H:=\c_{\G^*}^\circ(s)$ and notice that, by \cite[Theorem (ii)]{Cab-Eng94}, the block $\bl(\chi^\circ(s))$ has defect group $D\in\syl_\ell(\c_{\H}^\circ([\H,\H])^{F^*})$. Since $\H=\z^\circ(\H)[\H,\H]$, it follows that $\c_{\H}^\circ([\H,\H])=\z^\circ(\H)$. Thus $D\leq \z(\H)^{F^*}\leq \z(\H^{F^*})$ and, using \cite[Theorem 9.12]{Nav98}, we obtain
\[\chi^\circ(s)(1)_\ell=|\H^{F^*}:D|_\ell.\]
This implies
\begin{equation}
\label{eq:Minimal Brauer--Lusztig triple general 2}
\chi^\circ(s)(1)_\ell=|\H^{F^*}:\z^\circ(\H)^{F^*}|_\ell.
\end{equation}
Combining \eqref{eq:Minimal Brauer--Lusztig triple general 1} and \eqref{eq:Minimal Brauer--Lusztig triple general 2} we see that it is enough to show that $Z:=\z^\circ(\H)^{F^*}_\ell=1$. To do so, observe that $\z^\circ(\G^*)_{\Phi_e}=\z^\circ(\H)_{\Phi_e}$ by \cite[Proposition 1.10]{Cab-Eng99}. In particular, for every $e$-split Levi subgroup $\K^*$ of $\G^*$ containing $\H$, we have $\K^*=\G^*$. Notice that $\H\leq \c_{\G^*}(Z)$ and that $\c_{\G^*}(Z)$ is an $e$-split Levi subgroup of $\G^*$ by Proposition \ref{prop:e-split Levi large primes} (i). Therefore $\c_{\G^*}(Z)=\G^*$ and $Z\leq \z(\G^*)^{F^*}_\ell=1$. This shows that $\chi$ has defect zero. To conclude, we notice that $\chi$ is the only character in its block $B$ while using \cite{His90} we obtain $\irr(B)\cap \E(\G^F,\ell')\neq \emptyset$. It follows that $\chi\in\E(\G^F,\ell')$.
\end{proof}

\section{New conjectures for finite reductive groups}
\label{sec:New conjectures}

For any prime number $\ell$, it is expected that the $\ell$-modular representation theory of a finite group is strongly determined by its $\ell$-local structure. This belief is supported by numerous results and conjectural evidences. If $\G^F$ is a finite reductive group defined over $\mathbb{F}_q$ with $\ell$ not dividing $q$, then its $\ell$-local structure is closely related to its $e$-local structure where $e$ is the order of $q$ modulo $\ell$ (modulo $4$ if $\ell=2$). For instance, under suitable restrictions on the prime $\ell$, every $e$-split Levi subgroup $\L$ of $\G$ gives rise to an $\ell$-local subgroup of $\G^F$. Namely, $\L^F=\c_{\G^F}(\z(\L)^F_\ell)$ is the centraliser of an $\ell$-subgroup in $\G^F$ (see Lemma \ref{lem:e-split Levi} and Lemma \ref{prop:Good primes} (i)). Using this idea, we can then try to determine a link between the $\ell$-modular representation theory of $\G^F$ and the $e$-local structure of $\G$. In this section we propose new conjectures that can be seen as analogues for finite reductive groups of the so-called counting conjectures. In Section \ref{sec:Towards Dade's Projective Conjecture} we compare our statements with the counting conjectures and show that they coincide whenever the prime $\ell$ is large enough.

\subsection{Counting characters $e$-locally}
\label{sec:Counting characters e-locally}

Let $\G$, $F$, $q$, $\ell$ and $e$ be as in Notation \ref{notation}. Denote by $\CL(\G,F)$, or simply by $\CL(\G)$ when $F$ is clear from the context, the set of descending chains of $e$-split Levi subgroups $\sigma=\{\G=\L_0>\L_1\>\dots>\L_n\}$ and define the length of $\sigma$ as $|\sigma|:=n$. We denote by $\L(\sigma)$ the smallest term of the chain $\sigma$. Consider the subset $\CL(\G)_{>0}$ of all chains of positive length. Notice that $\G^F$ acts on $\CL(\G)$ and on $\CL(\G)_{>0}$ and denote by $\G^F_\sigma$ the stabiliser in $\G^F$ of any chain $\sigma$.

Next, using \cite[Theorem 2.5]{Cab-Eng99} and \cite[Theorem 3.4]{Kes-Mal15}, we associate to every block $b$ of $\G_\sigma^F$ a uniquely determined block of $\G^F$. In order to apply \cite[Theorem 3.4]{Kes-Mal15}, we assume that $\G$ is an $F$-stable Levi subgroup of a simple simply connected group whenever $\ell$ is bad for $\G$. Let $\L:=\L(\sigma)$ be the smallest term of $\sigma$ and consider a block $b_\L$ of $\L^F$ covered by $b$. By \cite[Theorem 2.5]{Cab-Eng99} and \cite[Theorem 3.4]{Kes-Mal15}, we deduce that there exists a unique block of $\G^F$, denoted by $\R_{\L}^\G(b_\L)$ (see \cite[Notation 2.6]{Cab-Eng99}), containing all irreducible constituents of $\R_{\L\leq \P}^\G(\lambda)$ for every parabolic subgroup $\P$ of $\G$ with Levi complement $\L$ and every $\lambda\in\irr(b_\L)\cap \E(\L^F,\ell')$. Since $b_\L$ is uniquely determined by $b$ up to $\G_\sigma^F$-conjugation, the block
\[\R_{\G_\sigma}^\G(b):=\R_\L^\G(b_\L)\]
of $\G^F$ is well defined. Now, for every $\ell$-block $B$ of $\G^F$ and $d\geq 0$, we denote by $\k^d(B_\sigma)$ the cardinality of the set
\[\irr^d(B_\sigma):=\left\lbrace\vartheta\in\irr\left(\G^F_\sigma\right)\hspace{1pt}\middle|\hspace{1pt} d(\vartheta)=d, \R_{\G_\sigma}^\G\left(\bl(\vartheta)\right)=B\right\rbrace.\]
Moreover, we denote by $\k^d(B)$ the cardinality of the set $\irr^d(B)$ consisting of the irreducible characters of $B$ with defect $d$ and by $\k_{\rm c}^d(B)$ be the number of $e$-cuspidal characters in $\irr^d(B)$.

With the above notation, we can present our conjecture which proposes a formula to count the number of characters of a given defect in a block in terms of $e$-local data. Notice that, since $e$-cuspidal characters are minimal with respect to the $e$-structure of $\G$, they should be interpreted as $e$-local objects. Here, we use the term $e$-local structure to indicate the collection of (proper) $e$-split Levi subgroups of $\G$ together with their normalisers and their intersections. 

\begin{conj}
\label{conj:Dade reductive}
Let $B$ be an $\ell$-block of $\G^F$ and $d\geq 0$. Then
\[\k^d(B)=\k_{\rm c}^d(B)+\sum\limits_\sigma(-1)^{|\sigma|+1}\k^d(B_\sigma)\]
where $\sigma$ runs over a set of representatives for the action of $\G^F$ on $\CL(\G)_{>0}$.
\end{conj}

Our statement can be considered as an adaptation of Dade's Conjecture to finite reductive groups. In fact, in Section \ref{sec:Towards Dade's Projective Conjecture} we show the two statements coincide when the prime $\ell$ is large for $\G$ (see Proposition \ref{prop:Equivalences Dade}). We also notice that by considering the contribution of characters of any defect, we can state a weak version of Conjecture \ref{conj:Dade reductive} which could be interpreted as an analogue of Alperin's Weight Conjecture in the formulation given by Kn\"orr and Robinson \cite[Theorem 3.8 and Theorem 4.6]{Kno-Rob89} (see also \cite[Theorem 10.7.1 (iv)]{Lin18II}). For this, we define $\k(B):=\sum_d\k^d(B)$, $\k_{\rm c}^d(B)=\sum_d\k_{\rm c}^d(B)$ and $\k(B_\sigma):=\sum_d\k^d(B_\sigma)$.

\begin{conj}
\label{conj:AWC reductive}
Let $B$ be an $\ell$-block of $\G^F$. Then
\[\k(B)=\k_{\rm c}(B)+\sum\limits_\sigma(-1)^{|\sigma|+1}\k(B_\sigma)\]
where $\sigma$ runs over a set of representatives for the action of $\G^F$ on $\CL(\G)_{>0}$.
\end{conj}

In Section \ref{sec:Towards Dade's Projective Conjecture} we show that Conjecture \ref{conj:AWC reductive} is equivalent to Alperin's Weight Conjecture for finite reductive groups and large primes. Indeed, in this case these two statment coincide with both Conjecture \ref{conj:Dade reductive} and Dade's Conjecture (see Proposition \ref{prop:Equivalences Dade}). As it was the case the Kn\"orr--Robinson reformulation of Alperin's Weight Conjecture, it is natural to ask whether the alternating sums presented in Conjecture \ref{conj:Dade reductive} and Conjecture \ref{conj:AWC reductive} can be expressed as the Euler characteristic of a chain complex. Such a complex has been constructed by using Bredon cohomology for Alperin's Weight Conjecture (see \cite{Sym05} and \cite{Lin05}).

\begin{rmk}
Assume that $\k^d_{\rm c}(B)\neq 0$ for some $d\geq 0$ and let $\chi$ be an $e$-cuspidal character in $\irr^d(B)$. If $\ell$ is large for $\G$ and $(\G,F,e)$-adapted with $\z(\G^*)^{F^*}_\ell=1$, then Proposition \ref{prop:Minimal Brauer--Lusztig triple general} implies that $\chi$ has defect zero and therefore $d=0$ and $\irr^d(B)=\irr(B)=\{\chi\}$. As a consequence $\k^d(B)=\k^d_{\rm c}(B)=1=\k(B)=\k_{\rm c}(B)$ and, we must have
\[\sum\limits_\sigma(-1)^{|\sigma|}\k^d(B_\sigma)=0.\]
We claim that in this case we even have $\k(B_\sigma)=0$ for every $\sigma\in\CL(\G)_{>0}$. In fact, if $\sigma$ has smallest term $\L:=\L(\sigma)<\G$ and $\k^d(B_\sigma)\neq 0$, then there exists a block $b_\L$ of $\L^F$ such that $\irr(B)$ contains all constituent of $\R_\L^\G(\lambda)$ for any $\lambda\in\irr(b_\L)\cap \E(\L^F,\ell')$. Since $\chi$ is the only character in $B$, it follows that $(\L,\lambda)\leq_e(\G,\chi)$. Since $\chi$ is $e$-cuspidal, this is a contradiction and hence $\k^d(B_\sigma)=0$ for every $\sigma\in\CL(\G)_{>0}$. Notice that, if $\ell$ is $(\G,F,e_0)$-adapted for some $e_0\neq e$, then $\G$ has no proper $e$-split Levi subgroups and Conjecture \ref{conj:Dade reductive} holds trivially.
\end{rmk}

%The above remark shows that Conjecture \ref{conj:Dade reductive} (resp. Conjecture \ref{conj:AWC reductive}) holds when $\ell$ is large, $\z(\G^*)^{F^*}_\ell=1$ and $\k^d_{\rm c}(B)\neq 0$ (resp. $\k_{\rm c}(B)\neq 0$). In Section \ref{sec:Dade reductive for large primes} we will remove the assumption on cuspidal characters and prove Conjecture \ref{conj:Dade reductive} whenever $\ell$ is large and $\z(\G^*)^{F^*}_\ell=1$ (see Corollary \ref{cor:Dade and AWC reductive for large primes}).

%{\color{red}{
%\begin{ex}
%Find an example in which $\k^d_{\rm c}(B)\neq 0$, $\z(\G^*)^{F^*}_\ell=1$ and $d>0$. There should be one! In fact, as there exists $e$-cuspidal pairs that are not $e$-Brauer--Lusztig-cuspidal, it is likely that we can find an $e$-cuspidal $\chi\in\irr(\G^F)$ that is not $e$-Brauer--Lusztig-cuspidal in $\G^F$. In this case we would have $\{\chi\}\subsetneq\irr(B)\cap \E(\G^F,s)$ (for $B=\bl(\chi)$ and some $s\in\G^{*F^*}$) and therefore $\{\chi\}\subsetneq\irr(B)$ and $d(\chi)\neq 0$.
%\end{ex}
%}}

\subsection{Introducing $\G^F$-block isomorphisms of character triples}
\label{sec:CTC reductive}

Counting conjectures for finite groups provide numerical evidence that is believed to be consequence of an underlying structural theory. In this regard, Broué's Abelian Defect Group Conjecture \cite{Bro90} proposes a structural explanation when considering the case of blocks with abelian defect groups. In a similar fashion, although perhaps on a more superficial level, the introduction by Isaacs, Malle and Navarro \cite{Isa-Mal-Nav07} of the so-called inductive conditions for the counting conjectures has initiated a study of stronger conjectures which suggest a way to control Clifford theory via the use of relations on the set of character triples. This idea has been exploited further in \cite[Theorem 7.1]{Nav-Spa14I} for the Alperin--McKay Conjecture (see also \cite{Ros-iMcK} for the simpler McKay Conjecture) and in \cite[Conjecture 1.2]{Spa17} for Dade's Conjecture. In this section, we introduce $\G^F$-block isomorphisms of character triples in the context of Conjecture \ref{conj:Dade reductive} and Conjecture \ref{conj:AWC reductive}. The equivalence relation on character triples that we consider here is denoted by $\iso{\G^F}$ and has been introduced in \cite[Definition 3.6]{Spa17}. We refer the reader to that paper for further details. Before proceeding further, we mentioned that in order to obtain the conditions on defect groups necessary to define $\G^F$-block isomorphisms of character triples we must assume $\ell\in\Gamma(\G,F)$.

Recall from Section \ref{sec:transitivity} that, for a finite reductive group $\G^F$, we denote by $\CP_e(\G,F)$ the set of all $e$-cuspidal pairs $(\L,\lambda)$ of $(\G,F)$ and by $\CP_e(\G,F)_<$ the subset of $e$-cuspidal pairs $(\L,\lambda)$ with $\L<\G$. When $\ell\in\Gamma(\G,F)$ and $B$ is a block of $\G^F$, we define the subsets $\CP_e(B)$ and $\CP_e(B)_<$ consisting of those pairs $(\L,\lambda)$ in $\CP_e(\G,F)$ and $\CP_e(\G,F)_<$ respectively such that $\bl(\lambda)^{\G^F}=B$. Recall that block induction is defined in this case as explained in the discussion preceding Lemma \ref{lem:Jordan decomposition for l-elements, blocks}. As in \cite[Definition 2.18]{Bro-Mal-Mic93}, let $\abirr(\L^F)$ be the set of (linear) characters of $\L^F$ containing $[\L,\L]^F$ in their kernel and, for a fixed character $\lambda\in\irr(\L^F)$, define
\[\ab(\lambda):=\left\lbrace\lambda\eta\hspace{1pt}\middle|\hspace{1pt}\eta\in\abirr\left(\L^F\right)\right\rbrace\]
and
\[\E\left(\G^F,(\L,\ab(\lambda))\right):=\bigcup\limits_{\lambda'\in\ab(\lambda)}\E\left(\G^F,(\L,\lambda')\right).\]
By \cite[Proposition 12.1]{Bon06}, observe that if $\lambda$ is $e$-cuspidal then every character in $\ab(\lambda)$ is $e$-cuspidal. Finally, for every $B\in \Bl(\G^F)$, $d\geq 0$ and $\epsilon\in\{+,-\}$ we define
\[\CL^d(B)_\epsilon:=\left\lbrace(\sigma,\M,\ab(\mu),\vartheta)\hspace{1pt}\middle|\hspace{1pt} \substack{\sigma\in\CL(\G)_\epsilon, (\M,\mu)\in\CP_e(B)_<\text{ with } \M\leq \L(\sigma),\\ \vartheta\in\irr^d\left(B_\sigma\hspace{1pt}\middle|\hspace{1pt} \E\left(\L(\sigma)^F,(\M,\ab(\mu))\right)\right)}\right\rbrace,\]
where $\CL(\G)_\epsilon$ is the subset of $\CL(\G)$ consisting of those chains $\sigma$ satisfying $(-1)^{|\sigma|}=\epsilon 1$ and the set $\irr^d(B_\sigma\mid \E(\L(\sigma)^F,(\M,\ab(\mu))))$ consists of those characters $\vartheta\in\irr(\G^F_\sigma)$ lying over some character in $\E(\L(\sigma)^F,(\M,\ab(\mu)))$ and such that $d(\vartheta)=d$ and $\bl(\vartheta)^{\G^F}=B$. Notice that the group $\G^F$ acts by conjugation on $\CL^d(B)_\epsilon$ and denote by $\CL^d(B)_\epsilon/\G^F$ the corresponding set of $\G^F$-orbits. As usual, for $(\sigma,\M,\ab(\mu),\vartheta)\in\CL^d(B)_\epsilon$ we denote the corresponding $\G^F$-orbit by $\overline{(\sigma,\M,\ab(\mu),\vartheta)}$. Moreover, for every $\omega\in\CL^d(B)_\epsilon/\G^F$ we denote by $\omega^\bullet$ the $\G^F$-orbit of pairs $(\sigma,\vartheta)$ such that $(\sigma,\M,\ab(\mu),\vartheta)\in\omega$ for some $e$-cuspidal pair $(\M,\mu)$.

With the notation introduced above, we can now present a more conceptual framework for the conjectures presented in the previous section by considering bijections inducing $\G^F$-block isomorphisms of character triples.

\begin{conj}
\label{conj:CTC reductive}
Let $\ell\in\Gamma(\G,F)$ and consider a block $B$ of $\G^F$ and $d\geq 0$. There exists an $\aut_\mathbb{F}(\G^F)_B$-equivariant bijection
\[\Lambda:\CL^d\left(B\right)_+/\G^F\to\CL^d\left(B\right)_-/\G^F\]
such that
\[\left(X_{\sigma,\vartheta},\G^F_\sigma,\vartheta\right)\iso{\G^F}\left(X_{\rho,\chi},\G^F_\rho,\chi\right)\]
for every $\omega\in\CL^d(B)_+/\G^F$, any $(\sigma,\vartheta)\in\omega^\bullet$, $(\rho,\chi)\in\Lambda(\omega)^\bullet$ and where $X:=\G^F\rtimes \aut_\mathbb{F}(\G^F)$.
\end{conj}

In analogy with the inductive conditions for the counting conjectures, the above statement should be understood as a version of Conjecture \ref{conj:Dade reductive} compatible with Clifford theory and with the action of automorphisms. Although not completely satisfactory from a structural point of view, Conjecture \ref{conj:CTC reductive} suggest a deeper explanation for the numerical phenomena proposed in Conjecture \ref{conj:Dade reductive}. By considering the contribution given by characters of any defect, we could introduce $\G^F$-block isomorphisms in the context of Conjecture \ref{conj:AWC reductive}. In Section \ref{sec:Towards Dade's Projective Conjecture}, we show that Conjecture \ref{conj:CTC reductive} implies Sp\"ath's Character Triple Conjecture for finite reductive groups and large primes (see Proposition \ref{prop:Equivalence CTC}).

We now explain in more details the connection between Conjecture \ref{conj:Dade reductive} and Conjecture \ref{conj:CTC reductive}. First, we provide a more explicit description of the blocks and irreducible characters of stabilisers of chains.

\begin{lem}
\label{lem:Characters and blocks of chains normalizers}
Consider a chain of $e$-split Levi subgroups $\sigma\in \CL(\G)$ with final term $\L:=\L(\sigma)$. If $\ell\in\Gamma(\G,F)$, then:
\begin{enumerate}
\item every block of $\G^F_\sigma$ is $\L^F$-regular (see \cite[p.210]{Nav98}). In particular, for $b\in \Bl(\L^F)$, the induced block $b^{\G^F_\sigma}$ is defined and is the unique block of $\G^F_\sigma$ that covers $b$;
\item if $\vartheta\in\irr(\G_\sigma^F)$, then $\bl(\vartheta)^{\G^F}$ is defined and $\R_{\G_\sigma}^\G(\bl(\vartheta))=\bl(\vartheta)^{\G^F}$;
\item assume Hypothesis \ref{hyp:Brauer--Lusztig blocks}. There is a partition of the irreducible characters of $\G^F_\sigma$ given by
\[\irr\left(\G^F_\sigma\right)=\coprod\limits_{(\M,\mu)/\sim}\irr\left(\G^F_\sigma\hspace{1pt}\middle|\hspace{1pt} \E(\L^F,(\M,\mu))\right),\]
where the union runs over the $e$-cuspidal pairs $(\M,\mu)$ of $\L$ up to $\G^F_\sigma$-conjugation.
\end{enumerate}
\end{lem}

\begin{proof}
To prove the first statement, set $X:=\z^\circ(\L)^F_\ell$ and observe that $\L^F=\c^\circ_{\G}(X)^F=\c_\G(X)^F$ by Lemma \ref{lem:e-split Levi} and Lemma \ref{prop:Good primes} (i). In particular $X\leq \O_\ell(\G^F_\sigma)$. If $B\in\Bl(\G^F_\sigma)$ has defect group $D$, then $X\leq D$ by \cite[Theorem 4.8]{Nav98}. Thus $\c_{\G^F_\sigma}(D)\leq \c_{\G^F}(X)=\L^F$ and \cite[Lemma 9.20]{Nav98} shows that $B$ is $\L^F$-regular. In particular, if the block $B$ covers $b\in\Bl(\L^F)$, then $B=b^{\G^F_\sigma}$ by \cite[Theorem 9.19]{Nav98}. This proves (i). Moreover, by \cite[Theorem 2.5]{Cab-Eng99} we know that $\R_\L^\G(b)=b^{\G^F}$ and so, if $\vartheta\in\irr(\G_\sigma^F)$ and $\bl(\vartheta)$ covers $b$, we deduce that $\R_{\G_\sigma}^\G(\bl(\vartheta))=\R_\L^\G(b)=b^{\G^F}$. 

Next, as $\irr(\L^F)$ is the union of the $e$-Harish-Chandra series $\E(\L^F,(\M,\mu))$ by Corollary \ref{cor:e-Harish-Chandra, disjointness}, we deduce that every character $\chi\in \irr(\G^F_\sigma)$ lies over some character of an $e$-Harish-Chandra series $\E(\L^F,(\M,\mu))$, where $(\M,\mu)$ is an $e$-cuspidal pair of $\L$. To conclude we have to show that, if $(\M',\mu')$ is another $e$-cuspidal pair of $\L$, then $\irr(\G^F_\sigma\mid \E(\L^F,(\M,\mu)))$ and $\irr(\G^F_\sigma\mid \E(\L^F,(\M',\mu')))$ are disjoint unless $(\M,\mu)$ and $(\M',\mu')$ are $\G^F_\sigma$-conjugate. Suppose that $\chi$ is a character belonging to the intersection of $\irr(\G^F_\sigma\mid \E(\L^F,(\M,\mu)))$ and $\irr(\G^F_\sigma\mid \E(\L^F,(\M',\mu')))$. Let $\psi\in\E(\L^F,(\M,\mu))$ and $\psi'\in\E(\L^F,(\M',\mu'))$ lie below $\chi$ and consider $g\in \G^F_\sigma$ such that $\psi=\psi'^g$. Then, $\psi\in\E(\L^F,(\M,\mu))\cap \E(\L^F,(\M',\mu')^g)$ and Corollary \ref{cor:e-Harish-Chandra, disjointness} implies that $(\M,\mu)=(\M',\mu')^{gx}$, for some $x\in \L^F$. Since $gx\in\G^F_\sigma$ the proof is now complete.
\end{proof}

We can now prove Theorem \ref{thm:Main CTC reductive implies Dade reductive} as a consequence of the above considerations.

\begin{theo}
\label{thm:CTC reductive imples Dade reductive}
Assume Hypothesis \ref{hyp:Brauer--Lusztig blocks} and consider $B\in\Bl(\G^F)$ and $d\geq 0$. If Conjecture \ref{conj:CTC reductive} holds for $B$ and $d$, then Conjecture \ref{conj:Dade reductive} holds for $B$ and $d$.
\end{theo}

\begin{proof}
If there exists a bijection between the sets $\CL^d(B)_+/\G^F$ and $\CL^d(B)_-/\G^F$, then
\begin{equation}
\label{eq:CTC reductive imples Dade reductive 1}
\sum\limits_{(\sigma,\M,\ab(\mu))/\G^F}(-1)^{|\sigma|}\left|\hspace{1pt}\irr^d\left(B_\sigma\hspace{1pt}\middle|\hspace{1pt}\E\left(\L(\sigma)^F,(\M,\ab(\mu))\right)\right)\hspace{1pt}\right|=0
\end{equation}
where the sum runs over $\G^F$-conjugacy classes of triples $(\sigma,\M,\ab(\mu))$ with $\sigma\in\CL(\G)$ and $(\M,\mu)\in\CP_e(B)_<$ with $\M\leq \L(\sigma)$. By Lemma \ref{lem:Characters and blocks of chains normalizers} (ii)-(iii) we deduce that
\begin{equation}
\label{eq:CTC reductive imples Dade reductive 2}
\sum\limits_{(\M,\ab(\mu))/\G^F_\sigma}\left|\hspace{1pt}\irr^d\left(B_\sigma\hspace{1pt}\middle|\hspace{1pt}\E\left(\L(\sigma)^F,(\M,\ab(\mu))\right)\right)\hspace{1pt}\right|=\k^d(B_\sigma)
\end{equation}
whenever $\sigma\in\CL(\G)_{>0}$. On the other hand, since we are only considering $(\M,\mu)\in\CP_e(B)_<$, the contribution given by the trivial chain $\{\G\}$ is
\begin{equation}
\label{eq:CTC reductive imples Dade reductive 3}
\sum\limits_{(\M,\ab(\mu))/\G^F}\left|\hspace{1pt}\irr^d\left(B\hspace{1pt}\middle|\hspace{1pt}\E\left(\G^F,(\M,\ab(\mu))\right)\right)\hspace{1pt}\right|=\k^d(B)-\k_{\rm c}^d(B).
\end{equation}
Using \eqref{eq:CTC reductive imples Dade reductive 2} and \eqref{eq:CTC reductive imples Dade reductive 3}, we deduce that \eqref{eq:CTC reductive imples Dade reductive 1} may be rewritten as
\begin{align*}
0&=\sum\limits_{\sigma/\G^F}(-1)^{|\sigma|}\sum\limits_{(\M,\ab(\mu))/\G^F_\sigma}\left|\hspace{1pt}\irr^d\left(B_\sigma\hspace{1pt}\middle|\hspace{1pt}\E\left(\L(\sigma)^F,(\M,\ab(\mu))\right)\right)\hspace{1pt}\right|
\\
&=\k^d(B)-\k^d_{\rm c}(B)+\sum\limits_{\sigma\in\CL(\G)_{>0}/\G^F}(-1)^{|\sigma|}\k^d(B_\sigma)
\end{align*}
and therefore Conjecture \ref{conj:Dade reductive} holds for $B$ and $d$.
\end{proof}

In the following concluding remark we discuss the definition of the sets of quadruples $\CL^d(B)_\epsilon$ and show that some of the conditions imposed above are redundant.

\begin{rmk}
If $(\M,\mu)\in\CP_e(B)$ and $\mu'\in\mathcal{Y}(\mu)$, then we have $\mathcal{Y}(\mu)=\mathcal{Y}(\mu')$ although it might happen that $(\M,\mu')\nin\CP_e(B)$. On the other hand, let $\sigma\in\CL(\G)$ with last term $\L(\sigma)$ and consider an $e$-cuspidal pair $(\M,\mu)$ of $\L(\sigma)$. If $\vartheta\in\irr(\G_\sigma^F\mid \E(\L(\sigma)^F,(\M,\mathcal{Y}(\mu))))$ and $\bl(\vartheta)^{\G^F}=B$, then there exists $\mu'\in\mathcal{Y}(\mu)$, so that $\mathcal{Y}(\mu)=\mathcal{Y}(\mu')$, such that $(\M,\mu')\in\CP_e(B)$. In fact, there exists $\mu'\in\mathcal{Y}(\mu)$ such that $\vartheta\in\irr(\G^F_\sigma\mid\E(\L(\sigma)^F,(\M,\mu')))$. By Proposition \ref{prop:e-Harish-Chandra series and blocks} every character of $\E(\L(\sigma)^F,(\M,\mu'))$ is contained in $\bl(\mu')^{\L(\sigma)^F}$. Then, applying Lemma \ref{lem:Characters and blocks of chains normalizers} (i) and using the transitivity of block induction, it follows that $\bl(\vartheta)=(\bl(\mu')^{\L(\sigma)^F})^{\G^F_\sigma}=\bl(\mu')^{\G^F_\sigma}$. We deduce that $\bl(\mu')^{\G^F}=\bl(\vartheta)^{\G^F}=B$ and hence $(\M,\mu')\in\CP_e(B)$. In particular, we have
\[\CL^d(B)_\epsilon=\left\lbrace(\sigma,\M,\mathcal{Y}(\mu),\vartheta)\hspace{1pt}\middle|\hspace{1pt} \substack{\sigma\in\CL(\G)_\epsilon, (\M,\mu)\in\CP_e(\L(\sigma),F)\text{ with }\M<\G,\\ \vartheta\in\irr^d\left(\G^F_\sigma\hspace{1pt}\middle|\hspace{1pt} \E\left(\L(\sigma)^F,(\M,\mathcal{Y}(\mu))\right)\right)\text{ with }\bl(\vartheta)^{\G^F}=B}\right\rbrace.\]
\end{rmk}

\subsection{A parametrisation of $e$-Harish-Chandra series}
\label{sec:Parametrisation of e-HC series}

In section \ref{sec:Brauer-Lusztig blocks and e-HC series} we have shown how to describe the characters in a block of a finite reductive group in terms of $e$-Harish-Chandra theory. More precisely, Theorem \ref{thm:Blocks are unions of e-HC series} shows that the set of characters of a block $B$ of $\G^F$ is the disjoint union of $e$-Harish-Chandra series $\E(\G^F,(\L,\lambda))$ associated to certain $e$-cuspidal pairs $(\L,\lambda)$. The next natural step to understand the distribution of characters in the block $B$ is to find a parametrisation of the characters in each series. Inspired by the results of \cite{Bro-Mal-Mic93} and by classical Harish-Chandra theory, we propose a parametrisation of the series $\E(\G^F,(\L,\lambda))$ in terms of data analogue to the one encoded in the relative Weyl group $W_\G(\L,\lambda)^F$. At the same time, this parametrisation suggests an explanation for the Clifford theoretic and cohomological requirements imposed by the inductive conditions for the counting conjectures.

\begin{para}
\label{para:iEBC}
Let $\ell\in\Gamma(\G,F)$ and consider an $e$-cuspidal pair $(\L,\lambda)$ of $\G$. There exists a defect preserving $\aut_\mathbb{F}(\G^F)_{(\L,\lambda)}$-equivariant bijection
\[\Omega^\G_{(\L,\lambda)}:\E\left(\G^F,(\L,\lambda)\right)\to\irr\left(\n_\G(\L)^F\hspace{1pt}\middle|\hspace{1pt} \lambda\right)\]
such that
\[\left(X_\vartheta,\G^F,\vartheta\right)\iso{\G^F}\left(\n_{X_\vartheta}(\L),\n_{\G^F}(\L),\Omega^\G_{(\L,\lambda)}(\vartheta)\right)\]
for every $\vartheta\in\E\left(\G^F,(\L,\lambda)\right)$ and where $X:=\G^F\rtimes \aut_\mathbb{F}(\G^F)$.
\end{para}

We say that Parametrisation \ref{para:iEBC} holds for $(\G,F)$ at the prime $\ell$ if it holds for every $e$-cuspidal pair $(\L,\lambda)$ of $\G$ where $q$ is the prime power associated to $F$ and $e$ is the order of $q$ modulo $\ell$.

As we have said before, the above parametrisation should provide an explanation for the inductive conditions for the counting conjectures for finite reductive groups. Analogously, in Section \ref{sec:Final} we show that Conjecture \ref{conj:CTC reductive}, and hence Conjecture \ref{conj:Dade reductive} and Conjecture \ref{conj:AWC reductive}, holds once we assume the existence of Parametrisation \ref{para:iEBC} (see Theorem \ref{thm:Reduction for CTC for groups of Lie type}). We are then left to prove Parametrisation \ref{para:iEBC}. The results of \cite{Ros-Clifford_automorphisms_HC} shows that the bijections $\Omega_{(\L,\lambda)}^\G$ can be constructed once we assume certain technical conditions on the extendibility of characters of $e$-split Levi subgroups. These conditions also appear in the proofs of the inductive conditions for the McKay, the Alperin--McKay and the Alperin Weight conjectures (see \cite{Mal-Spa16}, \cite{Cab-Spa17I}, \cite{Cab-Spa17II}, \cite{Cab-Spa19}, \cite{Bro-Spa20}, \cite{Spa21}, \cite{Bro22}).

\begin{rmk}
\label{rmk:Consequences of parametrisation}
To conclude this section we derive an interesting consequence of Parametrisation \ref{para:iEBC}. For every $e$-cuspidal pair $(\L,\lambda)$ of $\G$ and any $d\geq 0$ we denote by $\k^d(\G^F,(\L,\lambda))$ the number of characters $\chi\in\E(\G^F,(\L,\lambda))$ with $d(\chi)=d$ and by $\k^d(\n_\G(\L)^F,\lambda)$ the number of characters $\psi\in\irr(\n_\G(\L)^F\mid \lambda)$ with $d(\psi)=d$. Since the bijection $\Omega_{(\L,\lambda)}^\G$ preserves the defect of characters, assuming Parametrisation \ref{para:iEBC} we obtain
\begin{equation}
\label{eq:para consequence 1}
\k^d\left(\G^F,(\L,\lambda)\right)=\k^d\left(\n_\G(\L)^F,\lambda\right)
\end{equation}
On the other hand, under Hypothesis \ref{hyp:Brauer--Lusztig blocks}, the partition given by Theorem \ref{thm:Blocks are unions of e-HC series} implies that
\begin{equation}
\label{eq:para consequence 2}
\k^d(B)=\k^d_{\rm c}(B)+\sum\limits_{(\L,\lambda)}\k^d\left(\G^F,(\L,\lambda)\right)
\end{equation}
for any block $B$ of $\G^F$ and where $(\L,\lambda)$ runs over a set of representatives for the action of $\G^F$ on $\CP_e(B)_<$. Then, combining \eqref{eq:para consequence 1} and \eqref{eq:para consequence 2} we obtain
\begin{equation}
\label{eq:para consequence 3}
\k^d(B)=\k^d_{\rm c}(B)+\sum\limits_{(\L,\lambda)}\k^d\left(\n_\G(\L)^F,\lambda\right)
\end{equation}
where as before $(\L,\lambda)$ runs over a set of representatives for the action of $\G^F$ on $\CP_e(B)_<$. The formula given in \eqref{eq:para consequence 3} suggests another way of counting the number $\k^d(B)$ in terms of $e$-local data. In particular, if we believe Conjecture \ref{conj:Dade reductive}, then under the above hypothesis we must have
\begin{equation}
\sum\limits_{(\L,\lambda)}\k^d\left(\n_\G(\L)^F,\lambda\right)=\sum\limits_\sigma(-1)^{|\sigma|+1}\k^d(B_\sigma)
\end{equation}
where $(\L,\lambda)$ and $\sigma$ run over a set of representatives for the action of $\G^F$ on $\CP_e(B)_<$ and on $\CL(\G)_{>0}$ respectively.
\end{rmk}

\section{Counting characters via $e$-Harish-Chandra theory}
\label{sec:Final}

The results obtained in Section \ref{sec:Brauer-Lusztig blocks and e-HC series} together with the parametrisation proposed in Section \ref{sec:Parametrisation of e-HC series} constitute crucial properties of $e$-Harish-Chandra theory. As an application of these powerful tools we show that Conjecture \ref{conj:CTC reductive}, and hence Conjecture \ref{conj:Dade reductive} and Conjecture \ref{conj:AWC reductive}, holds if we assume Parametrisation \ref{para:iEBC} (see Theorem \ref{thm:Reduction for CTC for groups of Lie type}). First, we prove some preliminary showing how to lift isomorphisms of character triples.

\subsection{Bijections and $N$-block isomorphic character triples}

The following proposition is an adaptation of \cite[Proposition 2.10]{Ros22} to finite reductive groups. Recall that, for $Y\unlhd X$ and $\mathcal{S}\subseteq \irr(Y)$, we denote by $\irr(X\mid \mathcal{S})$ the set of irreducible characters of $X$ whose restriction to $Y$ has an irreducible constituent contained in $\mathcal{S}$. Moreover, we define $X_{\mathcal{S}}:=\{x\in X\mid \mathcal{S}^x=\mathcal{S}\}$. 

\begin{prop}
\label{prop:Constructing bijections over bijections with central quotients}
Let $K\leq G\leq A$ be finite groups with $G\unlhd A$, consider $A_0\leq A$. and set $H_0:=H\cap A_0$ for every $H\leq A$. Consider $\mathcal{S}\subseteq \irr(K)$ and $\mathcal{S}_0\subseteq \irr(K_0)$ and suppose there exists $K\leq V\leq X\leq \n_A(K)$ and $U\leq X_0$ such that:
\begin{enumerate}
\item $V\leq X_{\mathcal{S}}$. Moreover, if $x\in X$ and $\mathcal{S}\cap \mathcal{S}^x\neq \emptyset$, then $x\in V$;
\item $U\leq X_{0,\mathcal{S}_0}$. Moreover, if $x\in X_0$ and $\mathcal{S}_0\cap\mathcal{S}_0^x\neq\emptyset$, then $x\in K_0U$; 
\item $V=KU$.
\end{enumerate}
Assume there exists a $U$-equivariant bijection
\[\Psi:\mathcal{S}\to\mathcal{S}_0\]
such that
\[\left(X_\vartheta,K,\vartheta\right)\iso{K}\left(X_{0,\vartheta},K_0,\Psi(\vartheta)\right)\]
for every $\vartheta\in \mathcal{S}$. If $K\leq J\leq X\cap G$ and $\c_{X}(Q)\leq X_0$ for every radical $\ell$-subgroup $Q$ of $J_0$, then there exists an $\n_U(J)$-equivariant bijection
\[\Phi_J:\irr\left(J\hspace{1pt}\middle|\hspace{1pt} \mathcal{S}\right)\to\irr\left(J_0\hspace{1pt}\middle|\hspace{1pt} \mathcal{S}_0\right)\]
such that
\[\left(\n_X(J)_\chi,J,\chi\right)\iso{J}\left(\n_{X_0}(J)_\chi,J_0,\Phi_J(\chi)\right)\]
for every $\chi\in\irr(J\mid \mathcal{S})$.
\end{prop}

\begin{proof}
Consider an $\n_U(J)$-transversal $\mathbb{S}$ in $\mathcal{S}$ and define $\mathbb{S}_0:=\{\Psi(\vartheta)\mid \vartheta\in\mathbb{S}\}$. Since $\Psi$ is $U$-equivariant, it follows that $\mathbb{S}_0$ is an $\n_U(J)$-transversal in $\mathcal{S}_0$. For every $\vartheta\in\mathbb{S}$, with $\vartheta_0:=\Psi(\vartheta)\in\mathbb{S}_0$, we fix a pair of projective representations $(\mathcal{P}^{(\vartheta)},\mathcal{P}^{(\vartheta_0)}_0)$ giving $(X_\vartheta,K,\vartheta)\iso{K}(X_{0,\vartheta},K_0,\vartheta_0)$. Now, let $\mathbb{T}$ be an $\n_U(J)$-transversal in $\irr(J\mid \mathcal{S})$ such that every character $\chi\in\mathbb{T}$ lies above a character $\vartheta\in\mathbb{S}$ (this can be done by the choice of $\mathbb{S}$). Moreover, using Clifford's theorem together with hypotheses (i) and (iii), it follows that every $\chi\in\mathbb{T}$ lies over a unique $\vartheta\in\mathbb{S}$.

For $\chi\in\mathbb{T}$ lying over $\vartheta\in\mathbb{S}$, let $\psi\in\irr(J_\vartheta\mid \vartheta)$ be the Clifford correspondent of $\chi$ over $\vartheta$. Set $\vartheta_0:=\Psi(\vartheta)\in\mathbb{S}_0$ and consider the $\n_U(J)_\vartheta$-equivariant bijection $\sigma_{J_\vartheta}:\irr(J_\vartheta\mid \vartheta)\to\irr(J_{0,\vartheta}\mid \vartheta_0)$ induced by our choice of projective representations $(\mathcal{P}^{(\vartheta)},\mathcal{P}_0^{(\vartheta_0)})$. Let $\psi_0:=\sigma_{J_\vartheta}(\psi)$. Observe that $J_{0,\vartheta_0}=J_{0,\vartheta}$. To see this, notice that $U_\vartheta=U_{\vartheta_0}$ since $\Psi$ is $U$-equivariant and that $J_{0,\vartheta_0}\leq K_0U$ by (ii) above. Therefore $J_{0,\vartheta_0}\leq J_{0,\vartheta}$. On the other hand, since $(J\cap U)_\vartheta=(J\cap U)_{\vartheta_0}$ because $\Psi$ is $U$-equivariant and noticing that $J_{0,\vartheta}\leq J_0\cap V=K_0(J\cap U)$ by using (iii), it follows that $J_{0,\vartheta}\leq J_{0,\vartheta_0}$. Now $\Phi_J(\chi):=\psi^{J_0}$ is irreducible by the Clifford correspondence. We define
\[\Phi_J\left(\chi^x\right):=\Phi_J(\chi)^x\]
for every $\chi\in\mathbb{T}$ and $x\in\n_U(J)$. This defines an $\n_U(J)$-equivariant bijection $\Psi:\irr(J\mid \mathcal{S})\to\irr(J_0\mid \mathcal{S}_0)$.

To prove the condition on character triples, consider $\chi\in\irr(J\mid \mathcal{S})$, $\vartheta\in\irr(\chi_K)\cap \mathcal{S}$, $\psi\in\irr(J_\vartheta\mid \vartheta)$ and $\vartheta_0:=\Psi(\vartheta)$, $\psi_0:=\sigma_{J_\vartheta}(\psi)$ and $\chi_0:=\Phi_J(\chi)$ as in the previous paragraph. Since $(X_\vartheta,K,\vartheta)\iso{K}(X_{0,\vartheta},K_0,\vartheta_0)$, \cite[Proposition 2.9 (ii)]{Ros22} implies that
\[\left(\n_{X_\vartheta}(J_\vartheta)_\psi,J_\vartheta,\psi\right)\iso{J_\vartheta}\left(\n_{X_{0,\vartheta}}(J_\vartheta)_\psi,J_{0,\vartheta},\psi_0\right)\]
and, because $\n_X(J)_\vartheta\leq \n_{X_\vartheta}(J_\vartheta)$, \cite[Lemma 3.8]{Spa17} implies
\begin{equation}
\label{eq:Constructing bijections over bijections with central quotients 1}
\left(\n_X(J)_{\vartheta,\psi},J_\vartheta,\psi\right)\iso{J_\vartheta}\left(\n_{X_0}(J)_{\vartheta,\psi},J_{0,\vartheta},\psi_0\right).
\end{equation}
To conclude, observe that by hypothesis we have
\[\c_{\n_X(J)_\chi}(Q)\leq \n_{X_0}(J)_\chi\]
for every $\chi_0\in\irr(J_0\mid \mathcal{S})$ and $Q\in\delta(\bl(\chi_0))$ and therefore we can apply \cite[Proposition 2.8]{Ros22} which, together with \eqref{eq:Constructing bijections over bijections with central quotients 1}, yields
\[\left(\n_X(J)_\chi,J,\chi\right)\iso{J}\left(\n_{X_0}(J)_\chi,J_0,\chi_0\right).\]
The proof is now complete.
\end{proof}

\begin{rmk}
\label{rmk:Defect preservation}
Consider the setup of Proposition \ref{prop:Constructing bijections over bijections with central quotients}. Then, the bijection $\Phi_J$ is defect preserving if and only if $\Psi$ is defect preserving.
\end{rmk}

\begin{proof}
For $\chi\in\irr(J\mid \mathcal{S})$, let $\psi$ be the Clifford correspondent of $\chi$ over some $\vartheta\in\irr(\chi_K)\cap \mathcal{S}$ and let $\psi_0:=\sigma_{J_\vartheta}(\psi)$ and $\vartheta_0:=\Psi(\vartheta)$. If $\chi_0:=\Phi_J(\chi)=\psi_0^{J_0}$, then $d(\chi)=d(\psi)$ and $d(\chi_0)=d(\psi_0)$. By \cite[Proposition 2.9 (iii)]{Ros22} we deduce that $d(\psi)-d(\psi_0)=d(\vartheta)-d(\vartheta_0)$.
\end{proof}

In \cite[Lemma 3.8 (c)]{Spa17} it is shown that $N$-block isomorphisms of character triples are compatible with the action of inner automorphisms. It is straightforward to extend this compatibility to arbitrary automorphisms.

\begin{lem}
\label{lem:Basic properties of N-block isomorphism}
Suppose that $(H_1,M_1,\vartheta_1)\iso{N}(H_2,M_2,\vartheta_2)$ with $G=H_1N=H_2N$. If $\gamma\in\aut(G)$, then $(H_1^\gamma,M_1^\gamma,\vartheta_1^\gamma)\iso{N^\gamma}(H_2^\gamma,M_2^\gamma,\vartheta_2^\gamma)$.
\end{lem}

\begin{proof}
The claim follows directly from the definition of $\iso{N}$ (see \cite[Definition 3.6]{Spa17}).
\end{proof}

\subsection{Proof of Theorem \ref{thm:Main reduction}}

We now start working towards a proof of Theorem \ref{thm:Main reduction}. In order to apply the results on $e$-Harish-Chandra theory obtain in Section \ref{sec:Brauer-Lusztig blocks and e-HC series}, we assume throughout this section that Hypothesis \ref{hyp:Brauer--Lusztig blocks} holds for our choice of $\G$, $F$, $q$, $\ell$ and $e$ as in Notation \ref{notation}.

\begin{defin}
\label{def:Irreducible rational component}
Let $\G$ be a connected reductive group with Frobenius endomorphism $F:\G\to \G$. Recall that $[\G,\G]$ is the product of simple algebraic groups $\G_1,\dots, \G_n$ and that $F$ acts on the set $\{\G_1,\dots, \G_n\}$. For any orbit $\mathcal{O}$ of $F$, we denote by $\G_\mathcal{O}$ the product of those simple algebraic groups in the orbit $\mathcal{O}$. Notice that $\G_\mathcal{O}$ is $F$-stable and, by abuse of notation, denote by $F$ the restriction of $F$ to $\G_\mathcal{O}$. Then, we say that $(\G_\mathcal{O},F)$ is an \emph{irreducible rational component} of $(\G,F)$.
\end{defin}

Recall that a connected reductive group $\G$ is called simply connected if the semisimple algebraic group $[\G,\G]$ is simply connected.

\begin{prop}
\label{prop:From rational components to K_0}
Assume Hypothesis \ref{hyp:Brauer--Lusztig blocks} and suppose that $\G$ is simply connected. Consider an $e$-split Levi subgroup $\K$ of $\G$ and suppose that Parametrisation \ref{para:iEBC} holds at the prime $\ell$ for every irreducible rational component of $(\K,F)$. Let $\K_0:=[\K,\K]$ and consider an $e$-cuspidal pair $(\L_0,\lambda_0)$ of $\K_0$. Then there exists a defect preserving $\aut_\mathbb{F}(\K_0^F)_{(\L_0,\lambda_0)}$-equivariant bijection
\[\Omega^{\K_0}_{(\L_0,\lambda_0)}:\E\left(\K_0^F,(\L_0,\lambda_0)\right)\to\irr\left(\n_{\K_0}(\L_0)^F\hspace{1pt}\middle|\hspace{1pt} \lambda_0\right)\]
such that
\[\left(Y_\vartheta,\K_0^F,\vartheta\right)\iso{\K_0^F}\left(\n_{Y_\vartheta}(\L_0),\n_{\K_0}(\L_0),\Omega^{\K_0}_{(\L_0,\lambda_0)}(\vartheta)\right)
\]
for every $\vartheta\in\E(\K_0^F,(\L_0,\lambda_0))$ and where $Y:=\K_0^F\rtimes \aut_\mathbb{F}(\K_0^F)$.
\end{prop}

\begin{proof}
Since $\G$ is simply connected, we deduce that $\K_0$ is a simply connected semisimple group (see \cite[Proposition 12.14]{Mal-Tes}). By \cite[Proposition 1.4.10]{Mar91}, $\K_0$ is the direct product of simple algebraic groups $\K_1,\dots, \K_n$ and the action of $F$ induces a permutation on the set of simple components $\K_i$. For every orbit of $F$ we denote by $\H_j$, $j=1,\dots, t$, the direct product of the simple components in such an orbit. Then $\H_j$ is $F$-stable and
\[\K_0^F=\H_1^F\times \cdots\times \H_t^F,\]
where by abuse of notation we denote the restriction of $F$ to $\H_j$ again by $F$. Observe that the $(\H_j, F)$'s are the irreducible rational components of $(\K,F)$. Define $\M_j:=\L_0\cap \H_j$ and observe that $\M_j$ is an $e$-split Levi subgroup of $\H_j$ and that
\[\L_0^F=\M_1^F\times \cdots \times\M_t^F.\]
Then, we can write $\lambda_0=\mu_1\times \cdots \times\mu_t$ with $\mu_j\in\irr(\M_j^F)$. As $\R_{\L_0}^{\K_0}=\R_{\M_1}^{\H_1}\times\cdots\times \R_{\M_t}^{\H_t}$ (see \cite[Proposition 10.9 (ii)]{Dig-Mic91}), it follows that $(\M_j,\mu_j)$ is an $e$-cuspidal pair of $\H_j$ for every $j=1,\dots, t$ and, using our assumption, there exist bijections
\begin{equation}
\label{eq:From rational components to K_0, 1}
\Omega^{\H_j}_{(\M_j,\mu_j)}:\E\left(\H_j^F,(\M_j,\mu_j)\right)\to\irr\left(\n_{\H_j}(\M_j)^F\hspace{1pt}\middle|\hspace{1pt}\mu_j\right)
\end{equation}
as in Parametrisation \ref{para:iEBC}. Since $\R_{\L_0}^{\K_0}=\R_{\M_1}^{\H_1}\times\cdots\times \R_{\M_t}^{\H_t}$, we deduce that $\E(\K_0^F,(\L_0,\lambda_0))$ coincides with the set of characters of the form $\vartheta_1\times\dots\times\vartheta_t$ with $\psi_j\in\E(\H_j^F,(\M_j,\mu_j))$, while it is not hard to see that $\irr(\n_{\K_0}(\L_0)^F\mid \lambda_0)$ coincides with the set of characters of the form $\xi_1\times \dots\times \xi_t$ with $\xi_j\in\irr(\n_{\H_j}(\M_j)^F\mid \mu_j)$. Hence, we obtain a bijection
\begin{align*}
\Omega_{(\L_0,\lambda_0)}^{\K_0}:\E\left(\K_0^F,(\L_0,\lambda_0)\right)&\to\irr\left(\n_{\K_0}(\L_0)^F\hspace{1pt}\middle|\hspace{1pt} \lambda_0\right)
\\
\vartheta_1\times\dots\times \vartheta_t&\mapsto\Omega_{(\M_1,\mu_1)}^{\H_1}(\vartheta_1)\times \dots\times \Omega_{(\M_t,\mu_t)}^{\H_t}(\vartheta_t).
\end{align*}
We now show that $\Omega^{\K_0}_{(\L_0,\lambda_0)}$ satisfies the required properties. 

First, consider the partition $\{1,\dots, t\}=\coprod_lA_l$ given by $j,k\in A_l$ if there exists a bijective morphism $\varphi:\H_j\to\H_k$ commuting with $F$ such that $\varphi(\M_j,\mu_j)=(\M_k,\mu_k)$. Fix $j_l\in A_l$. By Lemma \ref{lem:Basic properties of N-block isomorphism}, we may assume without loss of generality that
\[\K_0^F=\bigtimes\limits_l\H_{A_l}^F\]
and
\[\Omega_{(\L_0,\lambda_0)}^{\K_0}=\bigtimes\limits_l\Omega^{\H_{A_l}}_{(\M_{A_l},\mu_{A_l})}\]
where $\H_{A_l}:=\H_{j_l}^{|A_l|}$, $\M_{A_l}:=\M_{j_l}^{|A_l|}$, $\mu_{A_l}=\mu_{j_l}^{\otimes|A_l|}$ and $\Omega^{\H_{A_l}}_{(\M_{A_l},\lambda_{A_l})}:=(\Omega^{\H_{j_l}}_{(\M_{j_l},\lambda_{j_l})})^{\otimes|A_l|}$. Fix $\vartheta=\times_l\vartheta_{A_l}$, with $\vartheta_{A_l}\in\E(\H_{A_l}^F,(\M_{A_l},\mu_{A_l}))$, and write $\xi:=\Omega^{\K_0}_{(\L_0,\lambda_0)}(\vartheta)=\times_l\xi_{A_l}$ with $\xi_{A_l}=\Omega^{\H_{A_l}}_{(\M_{A_l},\mu_{A_l})}(\vartheta_{A_l})$. Then, noticing that $\aut_{\mathbb{F}}(\K_0^F)=\bigtimes_l\aut_{\mathbb{F}}(\H_{A_l}^F)$, by \cite[Theorem 5.1]{Spa17} it is enough to check that
\begin{equation}
\label{eq:From rational components to K_0, 2}
\left(Y_{A_l,\vartheta_{A_l}},\H_{A_l}^F,\vartheta_{A_l}\right)\iso{\H_{A_l}^F}\left(\n_{Y_{A_l,\vartheta_{A_l}}}(\M_{A_l}),\n_{\H_{A_l}}(\M_{A_l})^F,\xi_{A_l}\right)
\end{equation}
where $Y_{A_l}:=\H_{A_l}^F\rtimes \aut_\mathbb{F}(\H_{A_l}^F)$. 

To prove \eqref{eq:From rational components to K_0, 2}, observe that $\vartheta_{A_l}$ is $\aut_{\mathbb{F}}(\H_{A_l}^F)_{(\M_{A_l},\mu_{A_l})}$-conjugate to a character of the form $\times_u\vartheta_u$ such that for every $u, v$ we have either $\vartheta_u=\vartheta_v$ or $\vartheta_u$ and $\vartheta_v$ are not $\aut_\mathbb{F}(\H_{A_l}^F)$-conjugate. By Lemma \ref{lem:Basic properties of N-block isomorphism}, we may assume without loss of generality that $\vartheta_{A_l}=\times_u\nu_u^{m_u}$, where for every $u\neq v$ the characters $\nu_u$ and $\nu_v$ are distinct and not $\aut_\mathbb{F}(\H_{A_l})$-conjugate while $m_u$ are non-negative integers such that $|A_l|=\sum_u m_u$. Then
\[\aut_\mathbb{F}(\H_{A_l}^F)_{\vartheta_{A_l}}=\bigtimes\limits_u \left(\aut_\mathbb{F}(\H_{j_l})_{\nu_u}\wr S_{m_u}\right)\]
and hence \eqref{eq:From rational components to K_0, 2} follows from the properties of the bijections \eqref{eq:From rational components to K_0, 1} by applying \cite[Theorem 5.2]{Spa17}. A similar argument shows that the bijection $\Omega^{\K_0}_{(\L_0,\lambda_0)}$ is $\aut_\mathbb{F}(\K_0^F)_{(\L_0,\lambda_0)}$-equivariant. Moreover $\Omega^{\K_0}_{(\L_0,\lambda_0)}$ preserves the defect of characters by the analogous property of the bijections \eqref{eq:From rational components to K_0, 1}.
\end{proof}

We now prove an easy lemma which we use to combine bijections $\Omega^{\K_0}_{(\L_0,\lambda_0)}$ given by Proposition \ref{prop:From rational components to K_0} for various $e$-cuspidal pairs $(\L_0,\lambda_0)$.

\begin{lem}
\label{lem:Stabilizer of character set}
Let $X\leq Y\leq Z$ be finite groups with $X,Y\unlhd Z$ and $Y/X$ abelian. Consider $\eta\in\irr(Y)$ and define the set $\mathcal{Y}:=\{\eta\nu\mid \nu\in\irr(Y/X)\}$. If $z\in Z$ and $\mathcal{Y}^z\cap \mathcal{Y}\neq \emptyset$, then $\mathcal{Y}^z=\mathcal{Y}$.
\end{lem}

\begin{proof}
Suppose that $\eta\nu\in\mathcal{Y}^z\cap \mathcal{Y}$, then there exists $\nu_1\in\irr(Y/X)$ such that $\eta\nu=(\eta\nu_1)^z$. Since $Y/X$ is abelian we deduce that $\eta^z=\eta\nu(\nu_1^z)^{-1}$. Now, if $\eta\nu_2\in\mathcal{Y}$, then $(\eta\nu_2)^z=\eta^z\nu_2^z=\eta\nu(\nu_1^z)^{-1}\nu_2^z$. Noticing that $\nu(\nu_1^z)^{-1}\nu_2^z\in\irr(Y/X)$, we conclude that $\mathcal{Y}^z\subseteq\mathcal{Y}$ and the result follows.
\end{proof}

For every $e$-split Levi subgroup $\L_0$ of a connected reductive group $\K_0$ and every subset $\mathcal{Y}_0\subseteq \irr(\L_0^F)$ of $e$-cuspidal characters, we define $\E(\K_0^F,(\L_0,\mathcal{Y}_0))$ to be the union of the $e$-Harish-Chandra series $\E(\K_0^F,(\L_0,\nu))$ where $\nu\in\mathcal{Y}_0$.

\begin{cor}
\label{cor:From rational components to K_0}
Assume Hypothesis \ref{hyp:Brauer--Lusztig blocks} and suppose that $\G$ is simply connected. Consider an $e$-split Levi subgroup $\K$ of $\G$ and suppose that Parametrisation \ref{para:iEBC} holds at the prime $\ell$ for every irreducible rational component of $(\K,F)$. Let $(\L,\lambda)$ be an $e$-cuspidal pair of $\K$, set $\K_0:=[\K,\K]$ and $\L_0:=\L\cap \K_0$ and consider $\lambda_0\in\irr(\lambda_{\L_0^F})$. Define $\mathcal{Y}_0:=\{\lambda_0\xi\mid\xi\in\irr(\L_0^F/[\L,\L]^F)\}$. Then there exists a defect preserving $\aut_\mathbb{F}(\G^F)_{\K,\L,\mathcal{Y}_0}$-equivariant bijection
\[\Psi^{\K_0}_{(\L_0,\lambda_0)}:\E\left(\K_0^F,(\L_0,\mathcal{Y}_0)\right)\to\irr\left(\n_{\K_0}(\L_0)^F\hspace{1pt}\middle|\hspace{1pt} \mathcal{Y}_0\right)\]
such that
\[\left(Y_\vartheta,\K_0^F,\vartheta\right)\iso{\K_0^F}\left(\n_{Y_\vartheta}(\L_0),\n_{\K_0}(\L_0)^F,\Psi^{\K_0}_{(\L_0,\lambda_0)}(\vartheta)\right)
\]
for every $\vartheta\in\E(\K_0^F,(\L_0,\mathcal{Y}_0))$ and where $Y:=\K_0^F\rtimes \aut_\mathbb{F}(\K_0^F)$.
\end{cor}

\begin{proof}
First observe that for every $\lambda_0\xi\in\mathcal{Y}_0$ the pair $(\L_0,\lambda_0\xi)$ is $e$-cuspidal in $\K_0$ (see \cite[Proposition 12.1]{Bon06}). Moreover, notice that $\L=\z(\K)\L_0$ and therefore $\n_\K(\L_0)=\n_\K(\L)$. Let $\mathbb{T}$ be an $\n_{\K_0}(\L)_{\mathcal{Y}_0}^F\rtimes \aut_\mathbb{F}(\G^F)_{\K,\L,\mathcal{Y}_0}$-transversal in $\mathcal{Y}_0$. For each $\lambda_0\xi\in\mathbb{T}$ consider an $\aut_{\mathbb{F}}(\G^F)_{\K,\L,\lambda_0\xi}$-transversal $\mathcal{T}_{\lambda_0\xi}$ in $\E(\K_0^F,(\L_0,\lambda_0\xi))$ and define $\mathcal{T}$ as the union of the sets $\mathcal{T}_{\lambda_0\xi}$ with $\lambda_0\xi\in\mathbb{T}$.

We claim that $\mathcal{T}$ is an $\aut_\mathbb{F}(\G^F)_{\K,\L, \mathcal{Y}_0}$-transversal in
\[\E\left(\K_0^F,(\L_0,\mathcal{Y}_0)\right).\]
Let $\chi\in\E(\K_0^F,(\L_0,\lambda_0\xi))$ with $\xi\in\irr(\L_0^F/[\L,\L]^F)$ and consider the unique $\lambda_0\wh{\xi}\in\mathbb{T}$ such that $(\lambda_0\xi)^{xy}=\lambda_0\wh{\xi}$ for some $x\in\n_{\K_0}(\L)_{\mathcal{Y}_0}^F$ and $y\in\aut_\mathbb{F}(\G^F)_{\K,\L,\mathcal{Y}_0}$. Noticing that $\chi^y=\chi^{xy}\in\E(\K_0^F,(\L_0,\lambda_0\wh{\xi}))$ we can find a unique $\wh{\chi}\in\mathcal{T}_{\lambda_0\wh{\xi}}$ such that $\chi^{yz}=\wh{\chi}$ for some $z\in\aut_\mathbb{F}(\G^F)_{\K,\L,\lambda_0\wh{\xi}}$. By Lemma \ref{lem:Stabilizer of character set} we obtain $\aut_\mathbb{F}(\G^F)_{\K,\L,\lambda_0\wh{\xi}}\leq \aut_\mathbb{F}(\G^F)_{\K,\L,\mathcal{Y}_0}$ and hence $yz\in\aut_\mathbb{F}(\G^F)_{\K,\L,\mathcal{Y}_0}$. Next, for $i=1,2$ consider $\chi_i\in\mathcal{T}_{\lambda_0\xi_i}$ with $\lambda\xi_i\in\mathbb{T}$ such that $\chi_1=\chi_2^y$ with $y\in\aut_\mathbb{F}(\G^F)_{\K,\L,\mathcal{Y}_0}$. In particular $\chi_1\in\E(\K_0^F,(\L_0,\lambda_0\xi_1))\cap \E(\K_0^F,(\L_0,\lambda_0\xi_2)^y)$ and Proposition \ref{prop:e-Harish-Chandra series, disjointness} implies that $\lambda_0\xi_1=(\lambda_0\xi_2)^{yx}$ for some $x\in\n_{\K_0}(\L)^F$. Moreover, Lemma \ref{lem:Stabilizer of character set} yields $x\in\n_{\K_0}(\L)^F_{\mathcal{Y}_0}$ and by the choice of $\mathbb{T}$ it follows that $\lambda_0\xi_1=\lambda_0\xi_2$. Now $yx\in\aut_\mathbb{F}(\G^F)_{\K,\L,\lambda_0\xi_1}$ satisfies $\chi_1=\chi_2^{yx}$ and the choice of $\mathcal{T}_{\lambda_0\xi_1}$ implies that $\chi_1=\chi_2$. This proves the claim.

Next, using Proposition \ref{prop:From rational components to K_0}, for every $\lambda_0\xi\in\mathbb{T}$, $\chi\in\mathcal{T}_{\lambda_0\xi}$ and $x\in\aut_\mathbb{F}(\G^F)_{\K,\L,\mathcal{Y}_0}$ we define
\[\Psi^{\K_0}_{(\L_0,\lambda_0)}\left(\chi^x\right):=\Omega_{(\L_0,\lambda_0\xi)}^{\K_0}(\chi)^x.\]
Noticing that $\Psi_{(\L_0,\lambda_0)}^{\K_0}(\mathcal{T})$ is an $\aut_\mathbb{F}(\G^F)_{\K,\L,\mathcal{Y}_0}$-transversal in
\[\irr\left(\n_{\K_0}(\L_0)^F\hspace{1pt}\middle|\hspace{1pt} \mathcal{Y}_0\right)\]
we deduce that $\Psi_{(\L_0,\lambda_0)}^{\K_0}$ is an $\aut_\mathbb{F}(\G^F)_{\K,\L,\mathcal{Y}_0}$-equivariant bijection. The remaining properties follow directly from the corresponding properties of the bijections $\Omega_{(\K_0,\lambda_0\xi)}^{\K_0}$ given by Proposition \ref{prop:From rational components to K_0}.
\end{proof}

Using \cite[Theorem 5.3]{Spa17} we rewrite the relations on character triples given by Corollary \ref{cor:From rational components to K_0} replacing $\K_0\rtimes \aut_\mathbb{F}(\K_0^F)$ with $(\G^F\rtimes\aut_\mathbb{F}(\G^F))_\K$.

\begin{cor}
\label{cor:From rational components to K_0, with G-automorphism}
Consider the setup of Corollary \ref{cor:From rational components to K_0}. Then
\[\left(X_\vartheta,\K_0^F,\vartheta\right)\iso{\K_0^F}\left(\n_{X_\vartheta}(\L_0),\n_{\K_0}(\L_0),\Psi^{\K_0}_{(\L_0,\lambda_0)}(\vartheta)\right)
\]
for every $\vartheta\in\E(\K_0^F,(\L_0,\mathcal{Y}_0))$ and where $X:=(\G^F\rtimes \aut_\mathbb{F}(\G^F))_\K$.
\end{cor}

\begin{proof}
Fix $\vartheta$ as in the statement, let $Y:=\K_0^F\rtimes \aut_\mathbb{F}(\K_0^F)$ and consider the canonical maps
\[\epsilon:Y_\vartheta\to\aut(\K_0^F)\]
and
\[\wh{\epsilon}:X_\vartheta\to\aut(\K_0^F).\]
Define $U:=\wh{\epsilon}^{-1}(\epsilon(X_{\vartheta}))\leq Y_\vartheta$. By Corollary \ref{cor:From rational components to K_0} we know that
\[\left(Y_\vartheta,\K_0^F,\vartheta\right)\iso{\K_0^F}\left(\n_{Y_\vartheta}(\L_0),\n_{\K_0}(\L_0),\Psi^{\K_0}_{(\L_0,\lambda_0)}(\vartheta)\right)
\]
and applying \cite[Lemma 3.8]{Spa17} we obtain
\[\left(U,\K_0^F,\vartheta\right)\iso{\K_0^F}\left(\n_{U}(\L_0),\n_{\K_0}(\L_0),\Psi^{\K_0}_{(\L_0,\lambda_0)}(\vartheta)\right).
\]
Now \cite[Theorem 5.3]{Spa17} implies that
\[\left(X_\vartheta,\K_0^F,\vartheta\right)\iso{\K_0^F}\left(\n_{X_\vartheta}(\L_0),\n_{\K_0}(\L_0),\Psi^{\K_0}_{(\L_0,\lambda_0)}(\vartheta)\right)
\]
and this concludes the proof.
\end{proof}

Our next goal is to lift the bijection $\Psi^{\K_0}_{(\L_0,\lambda_0)}$ to a similar bijection $\Psi^{\K}_{(\L,\lambda)}$. To do so we need the following preliminary result.

\begin{lem}
\label{lem:From K_0 to K, with lambda}
Consider the setup of Corollary \ref{cor:From rational components to K_0} and recall that $\ab(\lambda)=\{\lambda\eta\mid \eta\in\abirr(\L^F)\}$. Then
\[\irr\left(\K^F\hspace{1pt}\middle|\hspace{1pt} \E(\K_0^F,(\L_0,\mathcal{Y}_0))\right)=\E\left(\K^F,(\L,\ab(\lambda))\right)\]
and
\[\irr\left(\n_\K(\L)^F\hspace{1pt}\middle|\hspace{1pt}\mathcal{Y}_0\right)=\irr\left(\n_\K(\L)^F\hspace{1pt}\middle|\hspace{1pt}\ab(\lambda)\right).\]
\end{lem}

\begin{proof}
Let $\lambda_0\xi\in\mathcal{Y}_0$ and consider $\chi\in\irr(\K^F\mid \E(\K_0^F,(\L_0,\lambda_0\xi)))$. Since $\L^F/[\L,\L]^F$ is abelian, $\xi$ has an extension $\wh{\xi}\in\abirr(\L^F)$. By \cite[Corollary 3.3.25]{Gec-Mal20} and \cite[Problem 5.3]{Isa76} we obtain
\[\ind_{\K_0^F}^{\K^F}\left(\R_{\L_0}^{\K_0}(\lambda_0\xi)\right)=\R_\L^\K\left(\ind_{\L_0^F}^{\L^F}(\lambda_0\xi)\right)=\R_\L^\K\left(\ind_{\L_0^F}^{\L^F}(\lambda_0)\wh{\xi}\right).\]
Then, by \cite[Problem 6.2]{Isa76} there exists $\eta\in\irr(\L^F/\L_0^F)$ such that $\chi\in\E(\K^F,(\L,\lambda\eta\wh{\xi}))$ with $\eta\wh{\xi}\in\abirr(\L^F)$. Conversely, assume that $\chi\in\E(\K^F,(\L,\lambda\eta))$ with $\lambda\eta\in\ab(\lambda)$. Applying \cite[Corollary 3.3.25]{Gec-Mal20}, we obtain
\[\res^{\K^F}_{\K_0^F}\left(\R_\L^\K(\lambda\eta)\right)=\R_{\L_0}^{\K_0}\left(\res_{\L_0^F}^{\L^F}(\lambda\eta)\right).\]
By Clifford's theorem we deduce that $\res_{\K_0^F}^{\K^F}(\chi)$ has an irreducible constituent in $\R_{\L_0}^{\K_0}(\lambda_0^g\xi)$ for some $g\in \L^F$ and $\xi:=\eta_{\L_0^F}\in\irr(\L_0^F/[\L,\L]^F)$. This proves the first equality.

Next, consider $\psi\in\irr(\n_\K(\L)^F\mid \lambda\eta)$ with $\lambda\eta\in\ab(\lambda)$. Since $\lambda\eta$ lies above $\lambda_0\xi$, with $\xi:=\eta_{\L_0^F}\in\irr(\L_0^F/[\L,\L]^F)$, we deduce that $\psi\in\irr(\n_\K(\L)^F\mid \mathcal{Y}_0)$. Conversely, suppose that $\psi\in\irr(\n_\K(\L)^F\mid \lambda_0\xi)$ with $\lambda_0\xi\in\mathcal{Y}_0$ and consider an extension $\eta_1\in\abirr(\L^F)$ of $\xi$. By \cite[Problem 5.3 and Problem 6.2]{Isa76}, we conclude that there exists $\eta_2\in\irr(\L^F/\L_0^F)$ such that $\psi$ lies above $\lambda\eta_1\eta_2$. Since $\eta:=\eta_1\eta_2\in\abirr(\L^F)$ the result follows.
\end{proof}

\begin{cor}
\label{cor:From K_0 to K}
Assume Hypothesis \ref{hyp:Brauer--Lusztig blocks} and suppose that $\G$ is simply connected, let $\K$ be an $e$-split Levi subgroup of $\G$ and suppose that Parametrisation \ref{para:iEBC} holds at the prime $\ell$ for every irreducible rational component of $(\K,F)$. Let $(\L,\lambda)$ be an $e$-cuspidal pair of $\K$ and consider $\ab(\lambda)$ as defined in Section \ref{sec:CTC reductive}. Then there exists a defect preserving $\aut_\mathbb{F}(\G^F)_{\K,\L,\ab(\lambda)}$-equivariant bijection
\[\Psi^{\K}_{(\L,\lambda)}:\E\left(\K^F,(\L,\ab(\lambda))\right)\to\irr\left(\n_\K(\L)^F\hspace{1pt}\middle|\hspace{1pt} \ab(\lambda)\right)\]
such that
\[\left(X_\vartheta,\K^F,\vartheta\right)\iso{\K^F}\left(\n_{X_\vartheta}(\L),\n_\K(\L)^F,\Psi^{\K}_{(\L,\lambda)}(\vartheta)\right)\]
for every $\vartheta\in\E(\K^F,(\L,\ab(\lambda)))$ and where $X:=(\G^F\rtimes \aut_\mathbb{F}(\G^F))_\K$.
\end{cor}

\begin{proof}
Define $\K_0:=[\K,\K]$, $\L_0:=\L\cap \K_0$, fix an irreducible constituent $\lambda_0$ of $\lambda_{\L_0^F}$ and set $\mathcal{Y}_0:=\{\lambda_0\xi\mid\xi\in\irr(\L_0^F/[\L,\L]^F)\}$. We apply Proposition \ref{prop:Constructing bijections over bijections with central quotients} with $A:=\G^F\rtimes \aut_\mathbb{F}(\G^F)$, $A_0:=\n_A(\L)$, $K:=\K_0^F$, $K_0=\n_{\K_0}(\L)^F=\n_{\K_0}(\L_0)^F$, $G:=\G^F$, $X:=(\G^F\rtimes \aut_\mathbb{F}(\G^F))_\K$, $\mathcal{S}:=\E(\K_0^F,(\L_0,\mathcal{Y}_0))$, $\mathcal{S}_0:=\irr(\n_{\K_0}(\L_0)^F\mid \mathcal{Y}_0)$, $V:=(\G^F\rtimes \aut_\mathbb{F}(\G^F))_{\K,\mathcal{S}}$ and $U:=(\G^F\rtimes \aut_\mathbb{F}(\G^F))_{\K,\L,\mathcal{Y}_0}$. Observe that assumptions (ii) and (iii) of Proposition \ref{prop:Constructing bijections over bijections with central quotients} are satisfied by Proposition \ref{prop:e-Harish-Chandra series, disjointness} and Lemma \ref{lem:Stabilizer of character set}. Consider the bijection between $\mathcal{S}$ and $\mathcal{S}_0$ given by Corollary \ref{cor:From rational components to K_0} and Corollary \ref{cor:From rational components to K_0, with G-automorphism}. In order to apply Proposition \ref{prop:Constructing bijections over bijections with central quotients} with $J:=\K^F$ we need to show that $\c_X(Q)\leq X_0$ for every radical $\ell$-subgroup $Q$ of $J_0=\n_{\K}(\L)^F$. By Lemma \ref{lem:e-split Levi} (ii), we know that $\L=\c_\G^\circ(E)$ with $E:=\z^\circ(\L)_\ell^F$ and hence $E\leq \O_\ell(\n_\K(\L)^F)$. Since $Q$ is a radical $\ell$-subgroup of $J_0$, it follows that $E\leq Q$ (see \cite[Proposition 1.4]{Dad92}) and therefore $\c_X(Q)\leq \c_X(E)\leq \n_X(E)=\n_X(\L)=X_0$. We can thus apply Proposition \ref{prop:Constructing bijections over bijections with central quotients} together with Corollary \ref{cor:From rational components to K_0}, Corollary \ref{cor:From rational components to K_0, with G-automorphism} and Lemma \ref{lem:From K_0 to K, with lambda} to obtain an $\aut_\mathbb{F}(\G^F)_{\K,\L,\mathcal{Y}_0}$-equivariant bijection
\[\Psi^{\K}_{(\L,\lambda)}:\E\left(\K^F,(\L,\ab(\lambda))\right)\to\irr\left(\n_\K(\L)^F\hspace{1pt}\middle|\hspace{1pt} \ab(\lambda)\right)\]
such that
\[\left(X_\vartheta,\K^F,\vartheta\right)\iso{\K^F}\left(\n_{X_\vartheta}(\L),\n_\K(\L)^F,\Psi^{\K}_{(\L,\lambda)}(\vartheta)\right)\]
for every $\vartheta\in\E(\K^F,(\L,\ab(\lambda)))$. Moreover, $\Psi^{\K}_{(\L,\lambda)}$ preserves the defect of characters by Remark \ref{rmk:Defect preservation}. To conclude, notice that by a Frattini argument and using Clifford's theorem and Lemma \ref{lem:Stabilizer of character set} we have
\[\aut_\mathbb{F}(\G^F)_{\K,\L,\ab(\lambda)}\leq \L^F\aut_\mathbb{F}(\G^F)_{\K,\L,\mathcal{Y}_0}\]
and therefore the bijection $\Psi^\K_{(\L,\lambda)}$ is $\aut_\mathbb{F}(\G^F)_{\K,\L,\ab(\lambda)}$-equivariant.
\end{proof}

Now, applying Proposition \ref{prop:Constructing bijections over bijections with central quotients}, we show how to lift the bijection given by Corollary \ref{cor:From K_0 to K} to a bijection
\[\Omega_{(\L,\lambda)}^{\K,H}:\irr\left(H\hspace{1pt}\middle|\hspace{1pt} \E(\K^F,(\L,\ab(\lambda)))\right)\to\irr\left(\n_H(\L)\hspace{1pt}\middle|\hspace{1pt}\ab(\lambda)\right)\]
for every $\K^F\leq H\leq \n_\G(\K)^F$. Notice that if $\sigma\in\CL(\G)$ has final term $\K$ and $H:=\G_\sigma^F$, then the above bijections gives a parametrisation of the characters considered in the definition of the set $\CL^d(B)_\epsilon$ of Conjecture \ref{conj:CTC reductive}.

\begin{prop}
\label{prop:iEBC going up}
Consider the setup of Corollary \ref{cor:From K_0 to K} and let $\K^F\leq H\leq \n_\G(\K)^F$. Then there exists a defect preserving $\aut_\mathbb{F}(\G^F)_{H,\K,\L,\ab(\lambda)}$-equivariant bijection
\[\Omega_{(\L,\lambda)}^{\K,H}:\irr\left(H\hspace{1pt}\middle|\hspace{1pt} \E\left(\K^F,(\L,\ab(\lambda))\right)\right)\to\irr\left(\n_H(\L)\hspace{1pt}\middle|\hspace{1pt} \ab(\lambda)\right)\]
such that
\[\left(\n_X(H)_\chi,H,\chi\right)\iso{H}\left(\n_X(H,\L)_\chi,\n_H(\L),\psi\right)\]
for every $\chi\in\irr(H\mid\E(\K^F,(\L,\ab(\lambda))))$ and where $X:=(\G^F\rtimes \aut_\mathbb{F}(\G^F))_\K$.
\end{prop}

\begin{proof}
We apply Proposition \ref{prop:Constructing bijections over bijections with central quotients} to the bijection given by Corollary \ref{cor:From K_0 to K}. We consider $A:=\G^F\rtimes \aut_\mathbb{F}(\G^F)$, $G:=\G^F$, $K:=\K^F$, $A_0:=\n_A(\L)$, $X:=\n_A(\K)$, $\mathcal{S}:=\E(\K^F,(\L,\ab(\lambda)))$, $\mathcal{S}_0:=\irr(\n_\K(\L)^F\mid \ab(\lambda))$, $U:=X_{0,\ab(\lambda)}$, $V:=X_{\mathcal{S}}$ and $J:=H$. By Proposition \ref{prop:e-Harish-Chandra series, disjointness} and Lemma \ref{lem:Stabilizer of character set} we deduce that conditions (ii) and (iii) of Proposition \ref{prop:Constructing bijections over bijections with central quotients} hold. Next, let $Q$ be a radical $\ell$-subgroup of $\n_H(\L)$. Set $E:=\z^\circ(\L)^F_\ell$ and notice that under our assumptions $\L=\c_\G^\circ(E)$ by Lemma \ref{lem:e-split Levi}. Then $E\leq \O_\ell(\n_H(\L))\leq Q$ because $Q$ is radical and we conclude that
$\c_X(Q)\leq \c_X(E)\leq \n_X(\L)=X_0$. We can therefore apply Proposition \ref{prop:Constructing bijections over bijections with central quotients} to obtain an $\aut_\mathbb{F}(\G^F)_{H,\K,\L,\ab(\lambda)}$-equivariant bijection $\Omega_{(\L,\lambda)}^{\K,H}$ as in the statement. Moreover $\Omega_{(\L,\lambda)}^{\K,H}$ preserves the defect of characters by Remark \ref{rmk:Defect preservation}.
\end{proof}

\begin{rmk}
It is worth pointing out that, in the previous results, the assumption that $\K$ is an $e$-split Levi subgroup of $\G$ can be weakened by only requiring $\K$ to be an $F$-stable Levi subgroup of $\G$.
\end{rmk}

We can finally prove Theorem \ref{thm:Main reduction}.

\begin{theo}
\label{thm:Reduction for CTC for groups of Lie type}
Assume Hypothesis \ref{hyp:Brauer--Lusztig blocks} and suppose that $\G$ is simply connected. If Parametrisation \ref{para:iEBC} holds at the prime $\ell$ for every irreducible rational component of any $e$-split Levi subgroup of $(\G,F)$, then Conjecture \ref{conj:CTC reductive} holds with respect to the prime $\ell$.
\end{theo}

\begin{proof}
Consider an $\ell$-block $B$ of $\G^F$, $d\geq 0$ and define $A:=\aut_\mathbb{F}(\G^F)$ and $X:=\G^F\rtimes A$. Let $\mathcal{T}_{1,+}$ be an $A_B$-transversal in the set
\[\mathcal{S}_{1,+}:=\{(\sigma,\M,\ab(\mu))\mid \sigma\in\CL(\G)_+, (\M,\mu)\in\CP_e(B)_<\text{ with }\M\leq\L(\sigma)\}\]
and fix an $A_{(\sigma,\M,\ab(\mu))}$-transversal $\mathcal{T}_{2,+}^{(\sigma,\M,\ab(\mu))}$ in $\irr^d(B_\sigma\mid \E(\L(\sigma)^F,(\M,\ab(\mu))))$ for each $(\sigma,\M,\ab(\mu))\in \mathcal{T}_{1,+}$. By our choices
\[\mathcal{T}_+:=\left\lbrace\overline{(\sigma,\M,\ab(\mu),\vartheta)}\hspace{1pt}\middle|\hspace{1pt} (\sigma,\M,\ab(\mu))\in\mathcal{T}_{1,+}, \vartheta\in\mathcal{T}_{2,+}^{(\sigma,\M,\ab(\mu))}\right\rbrace\]
is an $A_B$-transversal in $\CL^d(B)_+/\G^F$.

We now fix $(\sigma,\M,\ab(\mu))\in\mathcal{S}_{1,+}$. If $\L(\sigma)=\M$, then define $\rho$ to be the chain obtained by deleting $\L(\sigma)$ from $\sigma$. Since $(\M,\mu)\in\CP_e(B)_<$ we have $\M<\G$ and hence the chain $\rho$ is non-empty. On the other hand if $\M<\L(\sigma)$, then define $\rho$ to be the chain obtained by adding $\M$ to $\sigma$. In this case the last term $\L(\rho)$ of $\rho$ coincides with $\M$. This construction yields an $A_B$-equivariant bijection
\[\Delta:\mathcal{S}_{1,+}\to\mathcal{S}_{1,-}\]
where
\[\mathcal{S}_{1,-}:=\{(\rho,\n,\ab(\nu))\mid \rho\in\CL(\G)_-, (\n,\nu)\in\CP_e(B)_< \text{ with } \n\leq \L(\rho)\}.\]
In particular, the image $\mathcal{T}_{1,-}:=\Delta(\mathcal{T}_{1,+})$ is an $A_B$-transversal in $\mathcal{S}_{1,-}$. Moreover, notice that if $\Delta((\sigma,\M,\ab(\mu)))=(\rho,\n,\ab(\nu))$, then we have $(\n,\ab(\nu))=(\M,\ab(\mu))$ and
\begin{equation}
\label{eq:Reduction for CTC for groups of Lie type,1}
A_{(\sigma,\M,\ab(\mu))}=A_{(\rho,\n,\ab(\nu))}.
\end{equation}

Next, consider $(\sigma,\M,\ab(\mu))\in\mathcal{T}_{1,+}$ and $(\rho,\M,\ab(\mu)):=\Delta((\sigma,\M,\ab(\mu)))\in\mathcal{T}_{1,-}$. Assume first that $\L(\sigma)=\M$. By Proposition \ref{prop:iEBC going up} applied with $H=\G^F_\rho$, we obtain a bijection
\[\Omega_{(\M,\mu)}^{\L(\rho),\G^F_\rho}:\irr\left(\G^F_\rho\hspace{1pt}\middle|\hspace{1pt} \E\left(\L(\rho)^F,(\M,\ab(\mu))\right)\right)\to\irr\left(\n_{\G^F_\rho}(\M)\hspace{1pt}\middle|\hspace{1pt} \ab(\mu)\right).\]
Since $\M=\L(\sigma)$, notice that $\n_{\G^F_\rho}(\M)=\G^F_\sigma$ and that $\irr(\G^F_\sigma\mid \E(\L(\sigma)^F,(\M,\ab(\mu))))=\irr(\n_{\G^F_\rho}(\M)\mid \ab(\mu))$. We define
\[\mathcal{T}_{2,-}^{(\rho,\M,\ab(\mu))}:=\left(\Omega_{(\M,\mu)}^{\L(\rho),\G^F_\rho}\right)^{-1}\left(\mathcal{T}_{2,+}^{(\sigma,\M,\ab(\mu))}\right).\]
Similarly, if $\M<\L(\sigma)$, then Proposition \ref{prop:iEBC going up} applied with $H=\G^F_\sigma$ yields a bijection
\[\Omega_{(\M,\mu)}^{\L(\sigma),\G^F_\sigma}:\irr\left(\G^F_\sigma\hspace{1pt}\middle|\hspace{1pt} \E\left(\L(\sigma)^F,(\M,\ab(\mu))\right)\right)\to\irr\left(\n_{\G^F_\sigma}(\M)\hspace{1pt}\middle|\hspace{1pt} \ab(\mu)\right).\]
Noticing that $\n_{\G^F_\sigma}(\M)=\G^F_\rho$ and recalling that the last term $\L(\rho)$ of $\rho$ coincides with $\M$, it follows that $\irr(\n_{\G^F_\sigma}(\M)\mid \ab(\mu))=\irr(\G^F_\rho\mid \E(\L(\rho)^F,(\M,\ab(\mu))))$. In this case we define
\[\mathcal{T}_{2,-}^{(\rho,\M,\ab(\mu))}:=\Omega_{(\M,\mu)}^{\L(\sigma),\G^F_\sigma}\left(\mathcal{T}_{2,+}^{(\sigma,\M,\ab(\mu))}\right).\]
Since $\mathcal{T}_{2,+}^{(\sigma,\M,\ab(\mu))}$ is an $A_{(\sigma,\M,\ab(\mu))}$-transversal in $\irr^d(B_\sigma\mid \E(\L(\sigma)^F,(\M,\ab(\mu))))$, it follows from Proposition \ref{prop:iEBC going up} and \eqref{eq:Reduction for CTC for groups of Lie type,1} that $\mathcal{T}_{2,-}^{(\rho,\M,\ab(\mu))}$ is an $A_{(\rho,\M,\ab(\mu))}$-transversal in $\irr^d(B_\rho\mid \E(\L(\rho)^F,(\M,\ab(\mu))))$. In particular, the set
\[\mathcal{T}_-:=\left\lbrace\overline{(\rho,\M,\ab(\mu),\chi)}\hspace{1pt}\middle|\hspace{1pt} (\rho,\M,\ab(\mu))\in\mathcal{T}_{1,-}, \chi\in\mathcal{T}_{2,-}^{(\rho,\M,\ab(\mu))}\right\rbrace\]
is an $A_B$-transversal in $\CL^d(B)_-/\G^F$ in bijection with $\mathcal{T}_+$. By setting
\[\Lambda\left(\overline{(\sigma,\M,\ab(\mu),\vartheta)}^x\right):=\overline{(\rho,\M,\ab(\mu),\chi)}^x,\]
for every $x\in A_B$ and every $(\sigma,\M,\ab(\mu),\vartheta)\in\mathcal{T}_+$ corresponding to $(\rho,\M,\ab(\mu),\chi)\in\mathcal{T}_-$, we obtain an $A_B$-equivariant bijection
\[\Lambda:\CL^d(B)_+/\G^F\to\CL^d(B)_-/\G^F.\]

It remains to show that $\Lambda$ yields $\G^F$-block isomorphisms of characters triples. Let $(\sigma,\M,\ab(\mu),\vartheta)$ and $(\rho,\M,\ab(\mu),\chi)$ be as above. Without loss of generality we may assume that $\M<\L(\sigma)$ and so $\L(\rho)=\M$. By the construction given in the previous paragraph and using Proposition \ref{prop:iEBC going up} and \cite[Lemma 3.8(b)]{Spa17}, we know that
\begin{equation}
\label{eq:Reduction for CTC for groups of Lie type,2}
\left(X_{\sigma,\vartheta},\G^F_\sigma,\vartheta\right)\iso{\G^F_\sigma}\left(X_{\rho,\chi},\G^F_\rho,\chi\right).
\end{equation}
First we show that
\begin{equation}
\label{eq:Reduction for CTC for groups of Lie type,3}
\left(X_{\sigma,\vartheta},\G^F_\sigma,\vartheta\right)\iso{\G^F}\left(X_{\rho,\chi},\G^F_\rho,\chi\right).
\end{equation}
To do so, applying \cite[Lemma 2.11]{Ros22}, it is enough to check that 
\begin{equation}
\label{eq:Reduction for CTC for groups of Lie type,4}
\c_{\G^FX_{\sigma,\vartheta}}(D)\leq X_{\rho,\chi}
\end{equation}
for some defect group $D$ of $\bl(\chi)$. By \eqref{eq:Reduction for CTC for groups of Lie type,2} we already know that $\c_{X_{\sigma,\vartheta}}(D)\leq X_{\rho,\chi}$ and hence it suffices to prove $\c_{\G^FX_{\sigma,\vartheta}}(D)\leq X_{\sigma,\vartheta}$. Write $\sigma=\{\G=\L_0>\cdots >\L_n=\L(\sigma)\}$ and set $E_i:=\z^\circ(\L_i)_\ell^F$. By the argument used at the end of the proof of Proposition \ref{prop:iEBC going up} and noticing that $\G^F_{\rho}\leq \G^F_\sigma$, we have $E_i\leq D$ and hence $\c_{\G^FX_{\sigma,\vartheta}}(D)\leq \c_{\G^FX_{\sigma,\vartheta}}(E_i)$ for every $i=0,\dots,n$. This implies that $\c_{\G^FX_{\sigma,\vartheta}}(D)\leq (\G^FX_{\sigma,\vartheta})_{\sigma}=X_{\sigma,\vartheta}$ and so we obtain \eqref{eq:Reduction for CTC for groups of Lie type,4}. We can now complete the proof by applying \cite[Lemma 2.11]{Ros22} to \eqref{eq:Reduction for CTC for groups of Lie type,2} in order to obtain \eqref{eq:Reduction for CTC for groups of Lie type,3}.
\end{proof}

\section{A connection with the local-global counting conjectures}
\label{sec:Towards Dade's Projective Conjecture}

In this section we show that if $\ell$ is large, $(\G,F,e)$-adapted and $\z(\G^*)^{F^*}_\ell=1$, then Conjecture \ref{conj:Dade reductive} is equivalent to Dade's Conjecture. Moreover, since in this case we are dealing with blocks of abelian defect, the main result of \cite{Kes-Mal13} implies that Conjecture \ref{conj:Dade reductive} and Dade's Conjecture are equivalent to Conjecture \ref{conj:AWC reductive} and Alperin's Weight Conjecture (in the Kn\"orr--Robinson reformulation) respectively. Notice that the above equivalences are not merely logical equivalences but holds block by block. Furthermore, under the same assumptions, we show that the Alperin--McKay Conjecture holds for a block $B$ of $\G^F$ if and only if
\[\k(B)=\sum\limits_{(\L,\lambda)}\k(\n_\G(\L)^F,\lambda)\]
where $(\L,\lambda)$ runs over a set of representatives for the action of $\G^F$ on $\CP_e(B)$ and $\k(\n_\G(\L)^F,\lambda)$ is the number of characters of $\irr(\n_\G(\L)^F\mid \lambda)$. Since the Alperin--McKay Conjecture holds in this case as a consequence of \cite[Theorem 5.24]{Bro-Mal-Mic93}, the above equality provides evidence for the validity of Parametrisation \ref{para:iEBC} as discussed in Remark \ref{rmk:Consequences of parametrisation}.

%{\footnotesize \color{red}{\marginnote{It is not clear if the AM coincides with DPC for $\G^F$ and abelian defect. In particular even for $\ell$ large Dade does not follow from the results of \cite{Bro-Mal-Mic93}!}}}
We need to make a small remark: in the discussion following \cite[Theorem 5.24]{Bro-Mal-Mic93} the authors claim that Dade's Conjecture holds for finite reductive groups with respect to large primes as a consequence of Broué's Isotypy Conjecture. Unfortunately, although Broue's Isotypy Conjecture logically implies Dade's Conjectures, it is not known whether this implication holds block by block (see also \cite[Remark 10.7.2]{Lin18II}). Therefore, at the time of writing Dade's Conjecture remains open for finite reductive groups and large primes. On the other hand, the connection established here with Conjecture \ref{conj:Dade reductive} could lead to a proof of such a result.

Under the above assumptions, we also show that Conjecture \ref{conj:CTC reductive} implies Sp\"ath's Character Triple Conjecture and, when $\G^F/\z(\G^F)$ is a nonabelian simple group with universal covering group $\G^F$, the inductive condition for Dade's Conjecture. Finally, we show that Parametrisation \ref{para:Main iEBC} implies the inductive Alperin--McKay condition.

\subsection{Statements of the conjectures}

Let $G$ be a finite group and $\ell$ a prime. We denote by $\irr^d(G)$ the set of irreducible characters $\chi\in\irr(G)$ with $\ell$-defect $d(\chi)=d$ and define $\irr^d(B):=\irr^d(G)\cap \irr(B)$ for any $\ell$-block $B$ of $G$. If $D$ is a defect group of $B$, then we define the defect of $B$ as the integer $d(B)$ such that $|D|=\ell^{d(B)}$. The set $\irr_0(B):=\irr^{d(B)}(B)$ consists of the characters of height zero in $B$. We denote by $\k(B)$, $\k^d(B)$ and $\k_0(B)$ the number of characters in the sets $\irr(B)$, $\irr^d(B)$ and $\irr_0(B)$ respectively. In \cite{Alp76}, Alperin introduced the following blockwise version of the McKay Conjecture.

\begin{conj}[Alperin--McKay Conjecture]
\label{conj:AM}
Let $G$ be a finite group and $B$ an $\ell$-block of $G$ with defect group $D$. Then
\[\k_0(B)=\k_0(b)\]
where $b$ is the Brauer correspondent of $B$ in $\n_G(D)$.
\end{conj}

Inspired by \cite{Isa-Mal-Nav07} and by the inductive Alperin--McKay condition for quasi-simple groups introduced in \cite{Spa13II}, a stronger version of the above conjecture has been considered in \cite{Nav-Spa14I}.

\begin{conj}[Inductive Alperin--McKay condition]
\label{conj:iAM}
Let $G$ be a finite group and $B$ an $\ell$-block of $G$ with defect group $D$ and Brauer correspondent $b$ in $\n_G(D)$. If $G\unlhd X$, then there exists an $X_{D,B}$-equivariant bijection
\[\Omega:\irr_0(B)\to\irr_0(b)\]
such that
\[\left(X_\chi,G,\chi\right)\iso{G}\left(\n_X(D)_\chi,\n_G(D),\Omega(\chi)\right)\]
for every $\chi\in\irr_0(B)$.
\end{conj}

The main result of \cite{Nav-Spa14I} shows that Conjecture \ref{conj:iAM} reduces to quasi-simple groups and, together with \cite{Kes-Mal13}, implies Brauer's Height Zero Conjecture.

Now, consider the set $\mathfrak{P}(G)$ of $\ell$-chains of $G$ with initial term $\O_\ell(G)$. These are the $\ell$-chains $\d=\{D_0=\O_\ell(G)<D_1<\dots<D_n\}$ where $D_i$ is an $\ell$-subgroup of $G$ and $n$ is a non-negative integer. If we denote by $|\d|$ the integer $n$, called the \emph{length} of $\d$, then we obtain a partition of $\mathfrak{P}(G)$ into the sets $\mathfrak{P}(G)_+$ and $\mathfrak{P}(G)_-$ consisting of $\ell$-chains of even and odd length respectively. Notice that $G$ acts by conjugation on the sets $\mathfrak{P}(G)$, $\mathfrak{P}(G)_+$ and $\mathfrak{P}(G)_-$ and we denote by $G_\d=\bigcap_i\n_G(D_i)$ the stabiliser in $G$ of $\d\in\mathfrak{P}(G)$ and by $\mathfrak{P}(G)/G$ a set of representatives for the $G$-orbits on $\mathfrak{P}(G)$. For any $\ell$-block $B$ of $G$, we the set $\irr(B_\d)=\{\vartheta\in\irr(G_\d)\mid \bl(\vartheta)^G=B\}$ and denote its cardinality by $\k(B_\d)$. Notice that the induced block $\bl(\vartheta)^G$ is well defined according to \cite[Lemma 3.2]{Kno-Rob89}. Moreover, $\k^d(B_\d)$ denotes the cadrinality of $\irr^d(B_\d):=\irr^d(G_\d)\cap \irr(B_\d)$ for any $d\geq 0$. Then Kn\"orr-Robinson reformulation of Alperin's Weight Conjecture (see \cite{Alp87} and \cite[Theorem 3.8 and Theorem 4.6]{Kno-Rob89}) can be stated as follows.

\begin{conj}[Alperin's Weight Conjecture]
\label{conj:AWC}
Let $G$ be a finite group such that $\O_\ell(G)\leq \z(G)$. Then
\[\sum\limits_{\d\in\mathfrak{P}(G)/G}(-1)^{|\d|}\k(B_\d)=0\]
for every $\ell$-block $B$ of $G$ with defect groups strictly containing $\O_\ell(G)$.
\end{conj}

Inspired by the reformulation introduced by Kn\"orr and Robinson, In \cite{Dad92} (see also \cite{Dad94}) Dade proposed a refinement of Alperin's Weight Conjecture by considering characters of any fixed defect.
 
\begin{conj}[Dade's Conjecture]
\label{conj:Dade}
Let $G$ be a finite group such that $\O_\ell(G)\leq \z(G)$. Then
\[\sum\limits_{\d\in\mathfrak{P}(G)/G}(-1)^{|\d|}\k^d(B_\d)=0\]
for every $\ell$-block $B$ of $G$ with defect groups strictly containing $\O_\ell(G)$ and any $d\geq 0$.
\end{conj}

Next, we recall the statements of the Character Triple Conjecture and of the inductive condition for Dade's Conjecture. For $\epsilon\in\{+,-\}$ and $B$ an $\ell$-block of $G$, define
\[\C^d(B)_\epsilon:=\left\lbrace(\d,\vartheta)\hspace{1pt}\middle|\hspace{1pt}\d\in\mathfrak{P}(G)_\epsilon,\vartheta\in\irr^d(B_\d)\right\rbrace\]
 Observe that $G$ acts on $\C^d(B)_\epsilon$ and denote by $\overline{(\d,\vartheta)}$ the $G$-orbit of $(\d,\vartheta)\in\C^d(B)_\epsilon$ and by $\C^d(B)_\epsilon/G$ a set of representatives for the $G$-orbits on $\C^d(B)_\epsilon$. The following statement has been proposed by Sp\"ath's in \cite[Conjecture 6.3]{Spa17} and can be seen as an analogue of Conjecture \ref{conj:iAM} to Dade's Conjecture.

\begin{conj}[Sp\"ath's Character Triple Conjecture]
\label{conj:CTC}
Let $G$ be a finite group such that $\O_\ell(G)\leq \z(G)$ and consider a block $B\in\Bl(G)$ with defect groups strictly larger than $\O_\ell(G)$. Suppose that $G\unlhd X$. Then, for every $d\geq 0$, there exists an $X_B$-equivariant bijection
\[\Omega:\C^d(B)_+/G \to \C^d(B)_-/G\]
such that
\[\left(X_{\d,\vartheta},G_\d,\vartheta\right)\iso{G}\left(X_{\e,\chi},G_\e,\chi\right)\]
for every $(\d,\vartheta)\in\C^d(B)_+$ and $(\e,\chi)\in\Omega(\overline{(\d,\vartheta)})$.
\end{conj}

We now show that the Character Triple Conjecture plays the role of an inductive condition for Dade's Conjecture. Notice that we could also consider a version of the above statement in the context of Alperin's Weight Conjecture (Conjecture \ref{conj:AWC}) by removing the restriction on the defect of characters, however the reduction theorem for Alperin's Weight Conjecture does not take into account the Kn\"orr--Robinson reformulation and deals with Brauer characters and weights instead.

\begin{lem}
\label{lem:iDade universal vs covering}
Let $S$ be a non-abelian simple group with universal covering group $\wh{S}$ and consider $B\in\Bl(\wh{S})$ with non-central defect groups. Then the following are equivalent:
\begin{enumerate}
\item the inductive condition for Dade's Conjecture (in the sense of \cite[Definition 6.7]{Spa17}) holds for $B$;
\item for each $Z\leq \z(G)$, Conjecture \ref{conj:CTC} holds with respect to $G:=\wh{S}/Z$, $X:=G\rtimes \aut(G)$ and every block of $G$ dominated by $B$;
\item there exists a bijection $\Omega:\C^d(B)_+/\wh{S}\to\C^d(B)_-/\wh{S}$ as in Conjecture \ref{conj:CTC} satisfying $\ker(\vartheta_{\z(\wh{S})})=\ker(\chi_{\z(\wh{S})})=:Z$ and
\begin{equation}
\label{eq:CTC with kernels}
\left(X_{\d,\vartheta}/Z,\wh{S}_\d/Z,\overline{\vartheta}\right)\iso{\wh{S}/Z}\left(X_{\e,\chi}/Z,\wh{S}_\e/Z,\overline{\chi}\right)
\end{equation}
for every $(\d,\vartheta)\in\C^d(B)_+$ and $(\e,\chi)\in\Omega(\overline{(\d,\vartheta)})$ and where $\overline{\vartheta}$ and $\overline{\chi}$ correspond, via inflation of characters, to $\vartheta$ and $\chi$ respectively and $X:=\wh{S}\rtimes \aut(\wh{S})$.
\end{enumerate}
\end{lem}

\begin{proof}
This is just a restatement of \cite[Proposition 6.8]{Spa17}.
\end{proof}

The main theorem of \cite{Spa17} shows that if the above conditions are satisfied for every simple group, then Dade's Conjecture holds for every finite group. Sp\"ath's results have been improved in \cite{Ros-Reduction_CTC} where Conjecture \ref{conj:CTC} is shown to reduce to quasi-simple groups. The requirement presented in \eqref{eq:CTC with kernels} is more restrictive than the one stated in Conjecture \ref{conj:CTC}. In fact, $G$-block isomorphisms of character truokes can be lifted from quotients with respect to central subgroups (see \cite[Corollary 4.4]{Spa17}). The above lemma tells us that proving the inductive condition for Dade's Conjecture (as defined in \cite[Definition 6.7]{Spa17}) for a non-abelian simple group $S$ is equivalent to show that Conjecture \ref{conj:CTC} holds, with this more restrictive requirement \eqref{eq:CTC with kernels}, for its universal covering group $\wh{S}$ with respect to $\wh{S}\unlhd \wh{S}\rtimes \aut(\wh{S})$. This remark is helpful since, in the majority of cases, the universal covering group of a simple group of Lie type is a finite reductive group $\G^F$ where $\G$ is a simple simply connected algebraic group with a Frobenius endomorphism $F$.

\subsection{An idea of Broué, Fong and Srinivasan}
\label{sec:Reduction first part}

For every finite group $G$, recall that the set $\mathfrak{E}(G)$ of $\ell$-elementary abelian chains of $G$ (starting at $\O_\ell(G)$) consists of those chains $\e=\{E_0=\O_\ell(G)<E_1<\dots<E_n\}$ such that $E_n/\O_\ell(G)$ is $\ell$-elementary abelian. Consider $\G$, $F$, $q$, $\ell$ and $e$ as in Notation \ref{notation}.

\begin{defin}[Broué--Fong--Srinivasan]
\label{def:Good chains}
Let $E$ be an $\ell$-elementary abelian subgroup of $\G^F$. Then $E$ is said to be \emph{good} if
\[E=\Omega_1\left(\O_\ell\left(\z^\circ\left(\c_\G^\circ\left(E\right)\right)^F\right)\right),\]
and \emph{bad} otherwise. An $\ell$-elementary abelian chain $\e\in\mathfrak{E}(\G^F)$ is said to be \emph{good} if $E_i$ is good for every $i$, while it is \emph{bad} otherwise. The set of good and bad $\ell$-elementary abelian chains of $\G^F$ is denote by $\mathfrak{E}_g(\G^F)$ and $\mathfrak{E}_b(\G^F)$ respectively.
\end{defin}

When $\ell$ is large and $(\G,F,e)$-adapted there exists a bijection between chains of $e$-split Levi subgroups of $\G$ and good $\ell$-elementary abelian chains of $\G^F$. Recall from Section \ref{sec:Automorphisms} that every automorphism $\alpha\in\aut_\mathbb{F}(\G^F)$ extends to a bijective endomorphism of $\G$ commuting with $F$. Then $\aut_\mathbb{F}(\G^F)$ acts on the set of $F$-stable closed connected subgroups of $\G$.

\begin{lem}
\label{lem:Elementary abelian vs Levi}
Suppose that $\ell$ is large, $(\G,F,e)$-adapted and $\O_\ell(\G^F)=1$. Then the maps
\begin{align*}
\CL(\G)&\to \mathfrak{E}_g\left(\G^F \right)
\\
\sigma=(\L_i) &\mapsto \e=\left(\Omega_1\left(\O_\ell\left(\z^\circ(\L_i)^F\right)\right)\right)
\end{align*}
and 
\begin{align*}
\mathfrak{E}_g(\G^F) &\to \CL(\G)
\\
\e=(E_i)&\mapsto \sigma:=\left(\c^\circ_\G(E_i)\right)
\end{align*}
are mutually inverse $\aut_\mathbb{F}(\G^F)$-equivariant length preserving bijections.
\end{lem}

\begin{proof}
First, consider a chain of $e$-split Levi subgroups $\sigma=(\G=\L_0>\dots >\L_n)$. Since $\ell$ is large for $\G$, Proposition \ref{prop:e-split Levi large primes} (iii) implies that $E_i:=\Omega_1(\O_\ell(\z^\circ(\L_i)^F))$ is a good $\ell$-elementary abelian subgroup such that $\L_i=\c^\circ_\G(E_i)$. Since $\L_i>\L_{i+1}$, this also shows that $E_i<E_{i+1}$ for every $i=0,\dots, n-1$. Moreover, as $\O_\ell(\G^F)=1$, we deduce that $E_0=\O_\ell(\G^F)$. On the other hand, if $\d=(\O_\ell(\G^F)=D_0<\dots <D_n)$ is a good $\ell$-elementary abelian chain, then all terms $D_i$ are elementary abelian (since $\O_\ell(\G^F)=1$) and Proposition \ref{prop:e-split Levi large primes} (i) shows that $\K_i:=\c_\G^\circ(D_i)$ is an $e$-split Levi subgroup. Furthermore $D_i=\Omega_1(\O_\ell(\z^\circ(\K_i)^F))$, because $D_i$ is good in the sense of Definition \ref{def:Good chains}, and $K_0=\G$. As a consequence, since $D_i<D_{i+1}$, we obtain that $\K_i>\K_{i+1}$ for every $i=0,\dots, n-1$. It follows that the above maps are inverses of each other and preserve the length of chains. To show that the maps are $\aut_\mathbb{F}(\G^F)$-equivariant, observe that $\Omega_1(\O_\ell(\z^\circ(\L)^F))^\alpha=\Omega_1(\O_\ell(\z^\circ(\L^\alpha)^F))$ and $\c_\G^\circ(E)^\alpha=\c_\G^\circ(E^\alpha)$ for every $e$-split Levi subgroup $\L$ of $\G$, every $\ell$-elementary abelian subgroup $E$ of $\G^F$ and every $\alpha\in\aut_\mathbb{F}(\G^F)$.
\end{proof}

Next, we show that there exists a self inverse $\aut_\mathbb{F}(\G^F)$-equivariant bijection on the set of bad $\ell$-elementary abelian chains.

\begin{lem}
\label{lem:Bad chains avoider}
If $\ell\in\Gamma(\G,F)$ and $\O_\ell(\G^F)=1$, then there exists an $\aut_\mathbb{F}(\G^F)$-equivariant bijection
\[\mathfrak{E}_b(\G^F)\to\mathfrak{E}_b(\G^F)\]
such that, if $\e$ is mapped to $\e'$, then $|\e|=|\e'|\pm 1$.
\end{lem}

\begin{proof}
Let $\e=(E_0<\dots<E_n)\in\mathfrak{E}_b(\G^F)$ and set $D_i:=\Omega_1(\O_\ell(\z^\circ(\c_\G^\circ(E_i))^F))$. Since $\O_\ell(\G^F)=1$, notice that $E_i$ is elementary abelian so that $E_i\leq D_i$ by Lemma \ref{prop:Good primes} (ii) and hence $\c_\G^\circ(D_i)\leq \c_\G^\circ(E_i)$. On the other hand, as $D_i\leq \z^\circ(\c_\G^\circ(E_i))^F$, we have $\c_\G^\circ(E_i)\leq \c_\G^\circ(D_i)$. Thus $\c_\G^\circ(E_i)=\c_\G^\circ(D_i)$ and we conclude that $D_i$ is a good $\ell$-elementary abelian subgroup. Now, since $\e$ is a bad chain, there exists a maximal index $j$ such that $E_j<D_j$. If $j=n$, then we define $\e'$ by adding $D_n$ to the chain $\e$. Assume $j<n$. In this case we claim that $D_j\leq E_{j+1}$ and we define $\e'$ to be the chain obtained from $\e$ by adding or removing $D_j$ to $\e$ if $D_j<E_{j+1}$ or $D_j=E_{j+1}$ respectively. To prove the claim, notice that $E_{j+1}\leq \c_\G(E_{j+1})^F\leq \c_\G(E_j)^F=\c_\G^\circ(E_j)^F$ by Lemma \ref{prop:Good primes} (i). As $D_j$ centralises $\c_\G^\circ(E_j)$, we deduce that $D_j\leq \c_\G(E_{j+1})^F=\c_\G^\circ(E_{j+1})^F$ and that $D_j$ centralises $\c_\G^\circ(E_{j+1})$. Thus $D_j\leq \z(\c_\G^\circ(E_{j+1}))$. Observing that $\c^\circ(E_{j+1})$ is an $F$-stable Levi subgroup of $\G$ (see \cite[Proposition 13.16]{Cab-Eng04}) we obtain $\ell\in\Gamma(\c_\G^\circ(E_{j+1}),F)$ by Remark \ref{rmk:Very good primes go to Levi subgroups} and hence $D_j\leq \z^\circ(\c_\G^\circ(E_{j+1}))$. This proves our claim that $D_j\leq D_{j+1}=E_{j+1}$.
\end{proof}

\subsection{New and old conjectures for large primes}
\label{sec:New and old conjecture for large primes}

Using the results described in the previous section, we can now show the equivalence between the Conjectures proposed in Section \ref{sec:Counting characters e-locally} and the local-global counting conjectures for finite reductive groups with respect to large primes. First, we recall that building on ideas of Bouc, Quillen, Thevenaz and Webb, \cite[Proposition 3.3]{Kno-Rob89} and \cite[Proposition 6.10]{Spa17} show that it is no loss of generality to restrict our attention to $\ell$-elementary abelian chains when working with Conjecture \ref{conj:Dade} and Conjecture \ref{conj:CTC} respectively. 

\begin{prop}
\label{prop:Equivalences Dade}
Let $\ell$ be large, $(\G,F,e)$-adapted and suppose that $\O_\ell(\G^F)=1=\z(\G^*)^{F^*}_\ell$. Then the following are equivalent for any $\ell$-block $B$ of $\G^F$ with non-trivial defect:
\begin{enumerate}
\item Conjecture \ref{conj:Dade reductive} holds for $B$;
\item Conjecture \ref{conj:AWC reductive} holds for $B$;
\item Alperin's Weight Conjecture holds for $B$ (see Conjecture \ref{conj:AWC});
\item Dade's Conjecture holds for $B$ (see Conjecture \ref{conj:Dade}).
\end{enumerate}
\end{prop}

\begin{proof}
Since $\ell$ is large, the block $B$ has abelian defect and therefore \cite{Kes-Mal13} implies that Conjecture \ref{conj:Dade reductive} is equivalent to Conjecture \ref{conj:AWC reductive} while Conjecture \ref{conj:AWC} is equivalent to Conjecture \ref{conj:Dade}. We show that Conjecture \ref{conj:Dade reductive} is equivalent to Conjecture \ref{conj:Dade}. As explained above, \cite[Proposition 3.3]{Kno-Rob89} implies that Conjecture \ref{conj:Dade} holds for $B$ if and only if
\begin{equation}
\label{eq:Equivalence Dade 1}
\sum\limits_{\d\in\mathfrak{E}(\G^F)/\G^F}(-1)^{|\d|}\k^d(B_\d)=0.
\end{equation}
On the other hand, using Lemma \ref{lem:Bad chains avoider} we deduce that
the contribution given by considering bad $\ell$-elementary abelian chains is zero and hence \eqref{eq:Equivalence Dade 1} is equivalent to
\begin{equation}
\label{eq:Equivalence Dade 2}
\sum\limits_{\d\in\mathfrak{E}_g(\G^F)/\G^F}(-1)^{|\d|}\k^d(B_\d)=0.
\end{equation}
Next, consider the bijection described in Lemma \ref{lem:Elementary abelian vs Levi} and suppose that $\sigma\in\CL(\G)$ corresponds to $\d\in\mathfrak{E}_g(\G^F)$. Since the bijection is $\aut_\mathbb{F}(\G^F)$-equivariant, it follows that $\G^F_\d=\G^F_\sigma$ and therefore $\k^d(B_\d)=\k^d(B_\sigma)$. Then \eqref{eq:Equivalence Dade 2} is equivalent to
\begin{align*}
0&=\sum\limits_{\sigma\in\CL(\G)/\G^F}(-1)^{|\sigma|}\k^d(B_\sigma)
\\
&=\sum\limits_{\sigma\in\CL(\G)_{>0}/\G^F}(-1)^{|\sigma|}\k^d(B_\sigma)+\k^d(B)
\end{align*}
which, since $\k^d_{\rm c}(B)=0$ (see Proposition \ref{prop:Minimal Brauer--Lusztig triple general}), holds if and only if Conjecture \ref{conj:Dade reductive} holds for $B$. 
\end{proof}

Next, we reformulate the Alperin--McKay Conjecture in terms of $e$-local structures. Recall that for any $e$-pair $(\L,\lambda)$ of $\G$ we denote by $\k(\n_\G(\L)^F,\lambda)$ the number of characters of $\n_\G(\L)^F$ lying above $\lambda$

\begin{prop}
\label{prop:Equivalence AM}
Assume Hypothesis \ref{hyp:Brauer--Lusztig blocks} with $\ell$ large and $(\G,F,e)$-adapted and consider an $\ell$-block $B$ of $\G^F$. Then the Alperin--McKay Conjecture (Conjecture \ref{conj:AM}) holds for $B$ if and only if
\begin{equation}
\label{eq:Equivalence AM 0}
\k(B)=\sum\limits_{(\L,\lambda)}\k\left(\n_\G(\L)^F,\lambda\right)
\end{equation}
where $(\L,\lambda)$ runs over a set of representatives for the action of $\G^F$ on $\CP_e(B)$. Moreover, this is equivalent to \eqref{eq:para consequence 3} whenever $B$ has non-trivial defect and $\z(\G^*)^{F^*}_\ell=1$.
\end{prop}

\begin{proof}
Let $D$ be a defect group of $B$ and consider its Brauer correspondent $b$ in $\n_{\G^F}(D)$. Notice that $D$ is abelian since $\ell$ is large and hence $\k(B)=\k_0(B)$ and $\k(b)=\k_0(b)$ (see \cite{Kes-Mal13}). To prove the first statement we need to show that
\begin{equation}
\label{eq:Equivalence AM 1}
\k(b)=\sum\limits_{(\L,\lambda)}\k\left(\n_\G(\L)^F,\lambda\right)
\end{equation}
where $(\L,\lambda)$ runs over a set of representatives for the action of $\G^F$ on $\CP_e(B)$.

Let $(\L,\lambda_0)$ be an $(e,\ell')$-cuspidal pair of $\G$ such that $B=b_{\G^F}(\L,\lambda_0)$ and notice that $\n_{\G^F}(D)=\n_{\G}(\L)^F$ by \cite[Lemma 4.16]{Cab-Eng99}. We claim that for every $(\M,\mu)\in\CP_e(B)$ the $e$-split Levi subgroup $\M$ is $\G^F$-conjugate to $\L$. In fact, consider a semisimple element $s\in\M^{*F^*}$ such that $\mu\in\E(\M^F,[s])$ and let $(\M(s_\ell),\mu(s_\ell))$ be the $e$-cuspidal pair of $\G(s_\ell)$ given by Lemma \ref{lem:Jordan decomposition for l-elements} (see also Lemma \ref{lem:Jordan decomposition for l-elements, order relation}). Under our assumptions $\G(s_\ell)$ is an $e$-split Levi subgroup of $\G$ (see Proposition \ref{prop:e-split Levi large primes} (i)) and therefore $(\M(s_\ell),\mu(s_\ell))$ is an $e$-cuspidal pair of $\G$. Moreover, Lemma \ref{lem:Jordan decomposition for l-elements, blocks} implies that $(\M(s_\ell),\mu(s_\ell))\in\CP_e(B)$ since $(\M,\mu)\in\CP_e(B)$. It follows that $B=b_{\G^F}(\M(s_\ell),\mu(s_\ell))$ and hence $\M(s_\ell)$ is $\G^F$-conjugate to $\L$ by \cite[Theorem 4.1]{Cab-Eng99}. On the other hand $\M(s_\ell)$ is an $e$-split Levi subgroup of $\M$ such that $\c_{\M^*}^\circ(s)\leq \M(s_\ell)^*$. Since $\mu$ is $e$-cuspidal, \cite[Proposition 1.10]{Cab-Eng99} implies that $\M(s_\ell)=\M$ and we conclude that $\M$ is $\G^F$-conjugate to $\L$.

Fix an $\n_\G(\L)^F$-transversal $\mathcal{T}'$ in the set of $e$-cuspidal characters $\lambda\in\irr(\L^F)$ such that $\bl(\lambda)^{\G^F}=B$. By the above paragraph we deduce that $\mathcal{T}:=\{(\L,\lambda)\mid \lambda\in\mathcal{T}'\}$ is an $\n_\G(\L)^F$-transversal in $\CP_e(\L,F)\cap \CP_e(B)$ and so, applying Lemma \ref{lem:Characters and blocks of chains normalizers} together with Brauer's first main theorem, we get
\[\irr(b)=\coprod\limits_{(\L,\lambda)\in\mathcal{T}}\irr\left(\n_\G(\L)^F\hspace{1pt}\middle|\hspace{1pt}\lambda\right).\]
Now, \eqref{eq:Equivalence AM 1} follows by noticing that $\mathcal{T}$ is also a $\G^F$-transversal in $\CP_e(B)$.

To conclude, assume that $B$ has non-trivial defect and that $\z(\G^*)^{F^*}_\ell=1$. In this case, Proposition \ref{prop:Minimal Brauer--Lusztig triple general} implies that no $e$-cuspidal character belong to the block $B$. In particular $\CP_e(B)=\CP_e(B)_<$ and $\k^d_{\rm c}(B)=0$ for every $d\geq 0$. Furthermore, the main result of \cite{Kes-Mal13} shows that $\k(B)=\k^d(B)$ and that $\k^d(\n_\G(\L)^F,\lambda)=\k(\n_\G(\L)^F,\lambda)$ for every $(\L,\lambda)\in\CP_e(B)$. Now, \eqref{eq:Equivalence AM 0} is equivalent to \eqref{eq:para consequence 3}, that is,
\[\k^d(B)=\k^d_{\rm c}(B)+\sum\limits_{(\L,\lambda)}\k^d\left(\n_\G(\L)^F,\lambda\right)\]
where $(\L,\lambda)$ runs over a set of representatives for the action of $\G^F$ on $\CP_e(B)_<$.
\end{proof}

By Remark \ref{rmk:Consequences of parametrisation} (see \eqref{eq:para consequence 3}), the following corollary provides evidence for Parametrisation \ref{para:iEBC}.

\begin{cor}
Assume Hypothesis \ref{hyp:Brauer--Lusztig blocks} with $\ell$ large and $(\G,F,e)$-adapted and consider an $\ell$-block $B$ of $\G^F$. Then
\[\k(B)=\sum\limits_{(\L,\lambda)}\k(\n_\G(\L),\lambda)\]
where $(\L,\lambda)$ runs over a set of representatives for the action of $\G^F$ on $\CP_e(B)$. Moreover, \eqref{eq:para consequence 3} holds whenever $B$ has non-trivial defect and $\z(\G^*)^{F^*}_\ell=1$.
\end{cor}

\begin{proof}
By \cite[Theorem 5.24]{Bro-Mal-Mic93} and the subsequence discussion,  under our assumptions the Alperin--McKay Conjecture holds for any block $B$ of $\G^F$. Then the result follows from Proposition \ref{prop:Equivalence AM}.
\end{proof}

Next, we consider the connection between Conjecture \ref{conj:CTC reductive}, the Character Triple Conjecture and the inductive condition for Dade's Conjecture.

\begin{prop}
\label{prop:Equivalence CTC}
Assume Hypothesis \ref{hyp:Brauer--Lusztig blocks} with $\ell$ large, $(\G,F,e)$-adapted and such that $\O_\ell(\G^F)=1=\z(\G^*)^{F^*}_\ell$. Consider an $\ell$-block $B$ of $\G^F$ with non-trivial defect and fix $d\geq 0$. If Conjecture \ref{conj:CTC reductive} holds for $B$ and $d$, then the Character Triple Conjecture (Conjecture \ref{conj:CTC}) holds for $B$ and $d$.
\end{prop}

\begin{proof}
Consider $(\e,\vartheta)\in\C^d(B)_+$. By \cite[Proposition 6.10]{Spa17} we may assume that $\e$ is an $\ell$-elementary abelian chain. If $\e$ is a bad $\ell$-elementary abelian chain (see Definition \ref{def:Good chains}), then we define
\[\Omega\left(\overline{\left(\e,\vartheta\right)}\right):=\overline{\left(\e',\vartheta\right)},\]
where $\e'$ is the chain corresponding to $\e$ via the bijection given by Lemma \ref{lem:Bad chains avoider}. Notice in this case that $\G^F_\e=\G^F_{\e'}$ and therefore that $(\e',\vartheta)\in\C^d(B)_-$. Assume that $\e$ is a good $\ell$-elementary abelian chain and consider the corresponding chain of $e$-split Levi subgroups $\sigma$ given by Lemma \ref{lem:Elementary abelian vs Levi}. Since the map $\d\mapsto\sigma$ is $\aut_{\mathbb{F}}(\G^F)$-equivariant, we know that $\G^F_\e=\G^F_\sigma$. Recall that $\L(\sigma)$ denotes the final term of $\sigma$. By Lemma \ref{lem:Characters and blocks of chains normalizers} (iii), there exists an $e$-cuspidal pair $(\M,\mu)$ of $\L$, unique up to $\G^F_\sigma$-conjugation, such that $\vartheta\in\irr^d(\G^F_\sigma\mid \E(\L^F,(\M,\mu))$. Then, as $\bl(\vartheta)^{\G^F}=B$, we have $\vartheta\in\irr^d(B_\sigma\mid \E(\L(\sigma)^F,(\M,\ab(\mu))))$.

Next, we claim that $(\M,\mu)\in\CP_e(B)_<$, so that $(\sigma,\M,\ab(\mu),\vartheta)\in\CL^d(B)_+$. First, observe that every character of $\E(\L^F,(\M,\mu))$ is contained in the block $\bl(\mu)^{\L^F}$ by Proposition \ref{prop:e-Harish-Chandra series and blocks}. Then, applying Lemma \ref{lem:Characters and blocks of chains normalizers} (i) and using the transitivity of block induction, it follows that $\bl(\vartheta)=(\bl(\mu)^{\L^F})^{\G^F_\sigma}=\bl(\mu)^{\G^F_\sigma}$. Since $(\d,\vartheta)\in\C^d(B)_+$, we deduce that $\bl(\vartheta)^{\G^F}=B$ and hence $(\M,\mu)\in\CP_e(B)$. Furthermore, as $B$ has non-trivial defect, Proposition \ref{prop:Minimal Brauer--Lusztig triple general} shows that $\M<\G$ and so $(\M,\mu)\in\CP_e(B)_<$. This proves our claim.

Now $(\sigma,\M,\ab(\mu),\vartheta)\in\CL^d(B)_+$ and we choose $(\rho,\n,\ab(\nu),\chi)\in\Lambda(\overline{(\sigma,\M,\ab(\mu),\vartheta)})$ where $\Lambda$ is the bijection given by Conjecture \ref{conj:CTC reductive}. If $\d$ is the good $\ell$-elementary abelian chain corresponding to $\rho$ via the bijection given by Lemma \ref{lem:Elementary abelian vs Levi}, then $(\d,\chi)\in\C^d(B)_-$ and we define
\[\Omega\left(\overline{\left(\e,\vartheta\right)}\right):=\overline{\left(\d,\chi\right)}.\]
Since $(\M,\mu)$ is unique up to $\G^F_\sigma$-conjugation while $\Lambda$ and the bijections given by Lemma \ref{lem:Elementary abelian vs Levi} and Lemma \ref{lem:Bad chains avoider} are equivariant, we conclude that $\Omega$ is a well defined $\aut_\mathbb{F}(\G^F)$-equivariant bijection. Moreover, using the property on character triples of $\Lambda$ it is immediate to show that $\Omega$ satisfies the analogous properties required by Conjecture \ref{conj:CTC} with respect to $X:=\G^F\rtimes \aut_\mathbb{F}(\G^F)$. This completes the proof.
\end{proof}

To obtain the inductive condition for Dade's Conjecture we apply Lemma \ref{lem:iDade universal vs covering}. Recall that, apart from a few exceptions, the universal covering group of a (non-abelian) simple group of Lie type is a finite reductive group $\G^F$ with $\G$ simple and simply connected. Moreover, in this case, $\aut_\mathbb{F}(\G^F)=\aut(\G^F)$ (see \cite[1.15]{GLS} and \cite[2.4]{Cab-Spa13}).

\begin{cor}
\label{cor:Equivalence iDade}
Consider the set up of Proposition \ref{prop:Equivalence CTC} and suppose that $\G^F/\z(\G^F)$ is a non-abelian simple group with universal covering group $\G^F$ where $\G$ is simple and simply connected. If Conjecture \ref{conj:CTC reductive} holds for $B$ and $d$, then the inductive condition for Dade's Conjecture holds for $B$ and $d$. 
\end{cor}

\begin{proof}
Consider the bijection $\Omega:\C^d(B)_+/\G^F\to\C^d(B)_-/\G^F$ constructed in the proof of Proposition \ref{prop:Equivalence CTC} and fix $(\d,\vartheta)\in\C^d(B)_+$ and $(\e,\chi)\in\Omega(\overline{(\d,\vartheta)})$. Since $\Omega$ satisfies the requirements of Conjecture \ref{conj:CTC}, we know that
\[\left(X_{\d,\vartheta},\G^F_\d,\vartheta\right)\iso{\G^F}\left(X_{\e,\chi},\G^F_\e,\chi\right)\]
with $X:=\G^F\rtimes \aut_\mathbb{F}(\G^F)$ and applying \cite[Lemma 3.4]{Spa17} we obtain $Z:=\ker(\vartheta_{\z(\G^F)})=\ker(\chi_{\z(\G^F)})$. Because, under our assumption, $\z(\G^F)$ has order coprime to $\ell$, it follows from \cite[Corollary 4.5]{Spa17} (see also Lemma \ref{lem:Centralizer, automorphisms and central subgroups}) that
\[\left(X_{\d,\vartheta}/Z,\G^F_\d/Z,\overline{\vartheta}\right)\iso{\G^F/Z}\left(X_{\e,\chi}/Z,\G^F_\e/Z,\overline{\chi}\right)\]
where $\overline{\vartheta}$ and $\overline{\chi}$ correspond to $\vartheta$ and $\chi$ respectively via inflation of characters. Now, the result follows from Lemma \ref{lem:iDade universal vs covering} after noticing that $\aut(\G^F)=\aut_\mathbb{F}(\G^F)$ under our hypothesis.
\end{proof}

To conclude we show that Parametrisation \ref{para:iEBC} implies the inductive Alperin--McKay condition (Conjecture \ref{conj:iAM}).

\begin{prop}
\label{prop:Equivalence iAM}
Assume Hypothesis \ref{hyp:Brauer--Lusztig blocks} with $\ell$ large and $(\G,F,e)$-adapted and consider an $\ell$-block $B$ of $\G^F$. If Parametrisation \ref{para:iEBC} holds with respect to every $(\L,\lambda)\in\CP_e(B)$, then the inductive Alperin--McKay condition (Conjecture \ref{conj:iAM}) holds for $B$ with respect to $\G^F\unlhd \G^F\rtimes \aut_\mathbb{F}(\G^F)$.
\end{prop}

\begin{proof}
Let $D$ be a defect group of $B$ and consider its Brauer correspondent $b$ in $\n_{\G^F}(D)$. Notice that $\irr_0(B)=\irr(B)$ and that $\irr_0(b)=\irr(b)$ by \cite{Kes-Mal13} since $D$ is abelian under our assumption. Let $(\L,\lambda_0)$ be an $(e,\ell')$-cuspidal pair of $\G$ such that $B=b_{\G^F}(\L,\lambda)$ and observe that $\n_{\G^F}(D)=\n_\G(\L)^F$ and that $\aut_\mathbb{F}(\G^F)_\L=\aut_\mathbb{F}(\G^F)_D$ by \cite[Lemma 4.16]{Cab-Eng99}. Fix an $\aut_\mathbb{F}(\G^F)_{\L,B}$-transversal $\mathcal{T}'$ in the set of $e$-cuspidal characters $\lambda\in\irr(\L^F)$ such that $\bl(\lambda)^{\G^F}=B$. Then, proceeding as in the proof of Proposition \ref{prop:Equivalence AM}, we deduce that $\mathcal{T}:=\{(\L,\lambda)\mid \lambda\in\mathcal{T}'\}$ is an $\aut_\mathbb{F}(\G^F)_{\L,B}$-transversal in $\CP_e(\L,F)\cap \CP_e(B)$ and a $\G^F\aut_\mathbb{F}(\G^F)_{\L,B}$-transversal in $\CP_e(B)$.

Next, for every $(\L,\lambda)\in\mathcal{T}$ we choose an $\aut_\mathbb{F}(\G^F)_{(\L,\lambda)}$-transversal $\mathcal{T}_B^{(\L,\lambda)}$ in $\E(\G^F,(\L,\lambda))$. We claim that
\[\mathcal{T}_B:=\coprod_{(\L,\lambda)\in\mathcal{T}}\mathcal{T}_B^{(\L,\lambda)}\]
is an $\aut_\mathbb{F}(\G^F)_{\L,B}$-transversal in $\irr(B)$. First suppose that $\chi\in\irr(B)$ and let $(\M,\mu)\in\CP_e(B)$ such that $\chi\in\E(\G^F,(\M,\mu))$ (see Theorem \ref{thm:Blocks are unions of e-HC series}). By the choices made in the previous paragraph, there exists $(\L,\lambda)\in\mathcal{T}$ such that $(\M,\mu)^{gx}=(\L,\lambda)$ for some $g\in\G^F$ and $x\in\aut_\mathbb{F}(\G^F)_{\L,B}$. It follows that $\chi^x=\chi^{gx}\in\E(\G^F,(\L,\lambda))$ and hence we find $\chi_0\in\mathcal{T}_B^{(\L,\lambda)}$ such that $\chi^{xy}=\chi_0$ for some $y\in\aut_\mathbb{F}(\G^F)_{(\L,\lambda)}$. Noticing that $\aut_\mathbb{F}(\G^F)_{(\L,\lambda)}\leq \aut_\mathbb{F}(\G^F)_{\L,B}$, this shows that $\chi$ is $\aut_\mathbb{F}(\G^F)_{\L,B}$-conjugate to an element in $\mathcal{T}_B$. On the other hand suppose that $(\L,\lambda)\in\mathcal{T}$, $\chi,\psi\in\mathcal{T}_B^{(\L,\lambda)}$ with $\chi=\psi^x$ for some $x\in\aut_\mathbb{F}(\G^F)_{\L,B}$. Then $(\L,\lambda)^{xg}=(\L,\lambda)$ for some $g\in \G^F$ by Proposition \ref{prop:e-Harish-Chandra series, disjointness} and we deduce that $\chi=\psi^x=\psi^{xg}$ with $xy\in\aut_\mathbb{F}(\G^F)_{(\L,\lambda)}$. By the choice of $\mathcal{T}^{(\L,\lambda)}_B$, this shows that $\chi=\psi$ and our claim is proved.

By assumption, for every $(\L,\lambda)\in\mathcal{T}$, there is an $\aut_\mathbb{F}(\G^F)_{(\L,\lambda)}$-equivariant bijection
\[\Omega_{(\L,\lambda)}^\G:\E\left(\G^F,(\L,\lambda)\right)\to\irr\left(\n_\G(\L)^F\hspace{1pt}\middle|\hspace{1pt}\lambda\right)\]
and hence $\mathcal{T}_b^{(\L,\lambda)}:=\Omega_{(\L,\lambda)}^\G(\mathcal{T}_B^{(\L,\lambda)})$ is an $\aut_\mathbb{F}(\G^F)_{(\L,\lambda)}$-transversal in $\irr(\n_\G(\L)^F\mid \lambda)$. We claim that
\[\mathcal{T}_b:=\coprod_{(\L,\lambda)\in\mathcal{T}}\mathcal{T}_b^{(\L,\lambda)}\]
is an $\aut_\mathbb{F}(\G^F)_{\L,B}$-transversal in $\irr(b)$. If $\vartheta\in\irr(b)$, arguing as in the proof of Proposition \ref{prop:Equivalence AM} we deduce that $\vartheta\in\irr(\n_\G(\L)^F\mid \mu)$ for some $(\L,\mu)\in\CP_e(\L,F)\cap \CP_e(B)$  via an application of Lemma \ref{lem:Characters and blocks of chains normalizers} and Brauer's first main theorem. Let $(\L,\lambda)\in\mathcal{T}$ and $x\in\aut_\mathbb{F}(\G^F)_{\L,B}$ such that $(\L,\mu)^x=(\L,\lambda)$. Then $\vartheta^x\in\irr(\n_\G(\L)^F\mid \lambda)$ and we can find $\vartheta_0\in\mathcal{T}_b^{(\L,\lambda)}$ such that $\vartheta^{xy}=\vartheta_0$ for some $y\in\aut_\mathbb{F}(\G^F)_{(\L,\lambda)}$. Therefore $\vartheta$ is conjugate to $\vartheta_0$ via an element of $\aut_\mathbb{F}(\G^F)_{\L,B}$. On the other, let $(\L,\lambda)\in\mathcal{T}$, $\vartheta,\eta\in\mathcal{T}_b^{(\L,\lambda)}$ and $x\in\aut_\mathbb{F}(\G^F)_{\L,B}$ such that $\vartheta^x=\eta$. By Clifford's theorem there exists $g\in\n_\G(\L)^F$ such that $\lambda^{xg}=\lambda$ and hence $\vartheta^{xg}=\vartheta^x=\eta$ for some $xg\in\aut_\mathbb{F}(\G^F)_{(\L,\lambda)}$. This shows that $\vartheta=\eta$ by the choice of $\mathcal{T}_b^{(\L,\lambda)}$ and so we obtain our claim.

Since $\mathcal{T}_B$ and $\mathcal{T}_b$ are in bijection, we obtain an $\aut_\mathbb{F}(\G^F)_{\L,B}$-equivariant bijection between the sets $\irr(B)$ and $\irr(b)$ by setting
\[\Omega(\chi^x):=\Omega_{(\L,\lambda)}^\G(\chi)^x\]
for every $(\L,\lambda)\in\mathcal{T}$, $\chi\in\mathcal{T}_B^{(\L,\lambda)}$ and $x\in\aut_\mathbb{F}(\G^F)_{\L,B}$. Recalling that $\n_\G(\L)^F=\n_{\G^F}(D)$ and $\aut_\mathbb{F}(\G^F)_{\L,B}=\aut_\mathbb{F}(\G^F)_{D,B}$, we deduce that the equivalence of character triples required by Conjecture \ref{conj:iAM} follow from the analogue properties given by Parametrisation \ref{para:iEBC}.
\end{proof}

%\begin{prop}
%\label{prop:Dade vs Dade reductive}

%\end{prop}

%\begin{prop}
%\label{prop:AWC vs AWC reductive}

%\end{prop}

%\begin{cor}
%\label{cor:Dade and AWC reductive for large primes}
%for this you would need Dade's conjecture for large primes.
%\end{cor}

%\section{Further evidences}

%\subsection{Large primes}
%\label{sec:Dade reductive for large primes}

%For $\ell$ large, Dade-Reductive is equivalent to Dade and the latter is claimed to holds for $\ell$ large in \cite{Bro-Mal-Mic93}. This claim however is incorrect... Therefore it is still open to shows that Dade-Reductive holds for $\ell$ large and this would also imply Dade for $\ell$ large.

%\subsection{Small rank}

%Can we prove Dade-Reductive for $\GL_2(q)$ or $\SL_2(q)$ as in \cite[Proposition 9.3]{Spa17}?

\bibliographystyle{alpha}
\bibliography{References}

\begin{thebibliography}{BMM93}

\bibitem[Alp76]{Alp76}
J.~L. Alperin.
\newblock The main problem of block theory.
\newblock In {\em Proceedings of the {C}onference on {F}inite {G}roups ({U}niv.
  {U}tah, {P}ark {C}ity, {U}tah, 1975)}, pages 341--356. Academic Press, New
  York, 1976.

\bibitem[Alp87]{Alp87}
J.~L. Alperin.
\newblock Weights for finite groups.
\newblock In {\em The {A}rcata {C}onference on {R}epresentations of {F}inite
  {G}roups ({A}rcata, {C}alif., 1986)}, volume~47 of {\em Proc. Sympos. Pure
  Math.}, pages 369--379. Amer. Math. Soc., Providence, RI, 1987.

\bibitem[Bon05]{Bon05}
C.~Bonnaf{\'e}.
\newblock Quasi-isolated elements in reductive groups.
\newblock {\em Comm. Algebra}, 33(7):2315--2337, 2005.

\bibitem[Bon06]{Bon06}
C.~Bonnaf{\'e}.
\newblock Sur les caract\`eres des groupes r\'{e}ductifs finis \`a centre non
  connexe: applications aux groupes sp\'{e}ciaux lin\'{e}aires et unitaires.
\newblock {\em Ast\'{e}risque}, (306):vi+165, 2006.

\bibitem[BM11]{Bon-Mic10}
C.~Bonnaf{\'e} and J.~Michel.
\newblock Computational proof of the {M}ackey formula for {$q>2$}.
\newblock {\em J. Algebra}, 327:506--526, 2011.

\bibitem[Bro90]{Bro90}
M.~Brou{\'e}.
\newblock Isom{\'e}tries parfaites, types de blocs, cat{\'e}gories
  d{\'e}riv{\'e}es.
\newblock {\em Ast{\'e}risque}, (181-182):61--92, 1990.

\bibitem[BM92]{Bro-Mal92}
M.~Brou{\'e} and G.~Malle.
\newblock Th{\'e}or{\`e}mes de {S}ylow g{\'e}n{\'e}riques pour les groupes
  r\'{e}ductifs sur les corps finis.
\newblock {\em Math. Ann.}, 292(2):241--262, 1992.

\bibitem[BMM93]{Bro-Mal-Mic93}
M.~Brou{\'e}, G.~Malle, and J.~Michel.
\newblock Generic blocks of finite reductive groups.
\newblock {\em Ast\'{e}risque}, (212):7--92, 1993.

\bibitem[BM89]{Bro-Mic89}
M.~Brou{\'e} and J.~Michel.
\newblock Blocs et s{\'e}ries de {L}usztig dans un groupe r{\'e}ductif fini.
\newblock {\em J. Reine Angew. Math.}, 395:56--67, 1989.

\bibitem[Bro22]{Bro22}
J.~Brough.
\newblock Characters of normalisers of $d$-split {L}evi subgroups in {${\rm
  Sp}_{2n}(q)$}.
\newblock {\em arXiv:2203.06072}, 2022.

\bibitem[BS20]{Bro-Spa20}
J.~Brough and B.~Sp{\"a}th.
\newblock On the {A}lperin-{M}c{K}ay conjecture for simple groups of type {A}.
\newblock {\em J. Algebra}, 558:221--259, 2020.

\bibitem[CE94]{Cab-Eng94}
M.~Cabanes and M.~Enguehard.
\newblock On unipotent blocks and their ordinary characters.
\newblock {\em Invent. Math.}, 117(1):149--164, 1994.

\bibitem[CE99]{Cab-Eng99}
M.~Cabanes and M.~Enguehard.
\newblock On blocks of finite reductive groups and twisted induction.
\newblock {\em Adv. Math.}, 145(2):189--229, 1999.

\bibitem[CE04]{Cab-Eng04}
M.~Cabanes and M.~Enguehard.
\newblock {\em Representation theory of finite reductive groups}, volume~1 of
  {\em New Mathematical Monographs}.
\newblock Cambridge University Press, Cambridge, 2004.

\bibitem[CS13]{Cab-Spa13}
M.~Cabanes and B.~Sp{\"a}th.
\newblock Equivariance and extendibility in finite reductive groups with
  connected center.
\newblock {\em Math. Z.}, 275(3-4):689--713, 2013.

\bibitem[CS17a]{Cab-Spa17I}
M.~Cabanes and B.~Sp{\"a}th.
\newblock Equivariant character correspondences and inductive {M}c{K}ay
  condition for type {$\bf A$}.
\newblock {\em J. Reine Angew. Math.}, 728:153--194, 2017.

\bibitem[CS17b]{Cab-Spa17II}
M.~Cabanes and B.~Sp{\"a}th.
\newblock Inductive {M}c{K}ay condition for finite simple groups of type
  {$\bf{C}$}.
\newblock {\em Represent. Theory}, 21:61--81, 2017.

\bibitem[CS19]{Cab-Spa19}
M.~Cabanes and B.~Sp{\"a}th.
\newblock Descent equalities and the inductive {M}c{K}ay condition for types
  {$\bf B$} and {$\bf E$}.
\newblock {\em Adv. Math.}, 356:106820, 48, 2019.

\bibitem[Dad92]{Dad92}
E.~C. Dade.
\newblock Counting characters in blocks. {I}.
\newblock {\em Invent. Math.}, 109(1):187--210, 1992.

\bibitem[Dad94]{Dad94}
E.~C. Dade.
\newblock Counting characters in blocks. {II}.
\newblock {\em J. Reine Angew. Math.}, 448:97--190, 1994.

\bibitem[DL76]{Del-Lus76}
P.~Deligne and G.~Lusztig.
\newblock Representations of reductive groups over finite fields.
\newblock {\em Ann. of Math. (2)}, 103(1):103--161, 1976.

\bibitem[DM91]{Dig-Mic91}
F.~Digne and J.~Michel.
\newblock {\em Representations of finite groups of {L}ie type}, volume~21 of
  {\em London Mathematical Society Student Texts}.
\newblock Cambridge University Press, Cambridge, 1991.

\bibitem[DM20]{Dig-Mic20}
F.~Digne and J.~Michel.
\newblock {\em Representations of finite groups of Lie type}, volume~95 of {\em
  London Mathematical Society Student Texts}.
\newblock Cambridge University Press, second edition, 2020.

\bibitem[Eng00]{Eng00}
M.~Enguehard.
\newblock Sur les {$l$}-blocs unipotents des groupes r\'{e}ductifs finis quand
  {$l$} est mauvais.
\newblock {\em J. Algebra}, 230(2):334--377, 2000.

\bibitem[Eng13]{Eng13}
M.~Enguehard.
\newblock Towards a {J}ordan decomposition of blocks of finite reductive
  groups.
\newblock {\em arXiv:1312.0106}, 2013.

\bibitem[FS82]{Fon-Sri82}
P.~Fong and B.~Srinivasan.
\newblock The blocks of finite general linear and unitary groups.
\newblock {\em Invent. Math.}, 69(1):109--153, 1982.

\bibitem[FS86]{Fon-Sri86}
P.~Fong and B.~Srinivasan.
\newblock Generalized {H}arish-{C}handra theory for unipotent characters of
  finite classical groups.
\newblock {\em J. Algebra}, 104(2):301--309, 1986.

\bibitem[GM20]{Gec-Mal20}
M.~Geck and G.~Malle.
\newblock {\em The character theory of finite groups of Lie type: A guided
  tour}, volume 187 of {\em Cambridge Studies in Advanced Mathematics}.
\newblock Cambridge University Press, 2020.

\bibitem[GLS98]{GLS}
D.~Gorenstein, R.~Lyons, and R.~Solomon.
\newblock {\em The classification of the finite simple groups. {N}umber 3.
  {P}art {I}. {C}hapter {A}}, volume~40 of {\em Mathematical Surveys and
  Monographs}.
\newblock American Mathematical Society, Providence, RI, 1998.

\bibitem[Hi{\ss}90]{His90}
G.~Hi{\ss}.
\newblock Zerlegungszahlen endlicher {G}ruppen vom {L}ie-{T}yp in
  nicht-definierender {C}harakteristik.
\newblock Habilitationsschrift, RWTH Aachen, 1990.

\bibitem[Hol22]{Hol22}
R.~Hollenbach.
\newblock On {$e$}-cuspidal pairs of finite groups of exceptional {L}ie type.
\newblock {\em J. Pure Appl. Algebra}, 226(1):106781, 2022.

\bibitem[Isa76]{Isa76}
I.~M. Isaacs.
\newblock {\em Character theory of finite groups}.
\newblock Academic Press [Harcourt Brace Jovanovich, Publishers], New
  York-London, 1976.

\bibitem[IMN07]{Isa-Mal-Nav07}
I.~M. Isaacs, G.~Malle, and G.~Navarro.
\newblock A reduction theorem for the {M}c{K}ay conjecture.
\newblock {\em Invent. Math.}, 170(1):33--101, 2007.

\bibitem[KM13]{Kes-Mal13}
R.~Kessar and G.~Malle.
\newblock Quasi-isolated blocks and {B}rauer's height zero conjecture.
\newblock {\em Ann. of Math. (2)}, 178(1):321--384, 2013.

\bibitem[KM15]{Kes-Mal15}
R.~Kessar and G.~Malle.
\newblock Lusztig induction and {$\ell$}-blocks of finite reductive groups.
\newblock {\em Pacific J. Math.}, 279(1-2):269--298, 2015.

\bibitem[KR89]{Kno-Rob89}
R.~Kn{\"o}rr and G.~R. Robinson.
\newblock Some remarks on a conjecture of {A}lperin.
\newblock {\em J. London Math. Soc. (2)}, 39(1):48--60, 1989.

\bibitem[Lin05]{Lin05}
M.~Linckelmann.
\newblock Alperin's weight conjecture in terms of equivariant {B}redon
  cohomology.
\newblock {\em Math. Z.}, 250(3):495--513, 2005.

\bibitem[Lin18]{Lin18II}
M.~Linckelmann.
\newblock {\em The block theory of finite group algebras. {V}ol. {II}},
  volume~92 of {\em London Mathematical Society Student Texts}.
\newblock Cambridge University Press, Cambridge, 2018.

\bibitem[Lus76]{Lus76}
G.~Lusztig.
\newblock On the finiteness of the number of unipotent classes.
\newblock {\em Invent. Math.}, 34(3):201--213, 1976.

\bibitem[Mal14]{Mal14}
G.~Malle.
\newblock On the inductive {A}lperin-{M}c{K}ay and {A}lperin weight conjecture
  for groups with abelian {S}ylow subgroups.
\newblock {\em J. Algebra}, 397:190--208, 2014.

\bibitem[MS16]{Mal-Spa16}
G.~Malle and B.~Sp{\"a}th.
\newblock Characters of odd degree.
\newblock {\em Ann. of Math. (2)}, 184(3):869--908, 2016.

\bibitem[MT11]{Mal-Tes}
G.~Malle and D.~Testerman.
\newblock {\em Linear algebraic groups and finite groups of {L}ie type}, volume
  133 of {\em Cambridge Studies in Advanced Mathematics}.
\newblock Cambridge University Press, Cambridge, 2011.

\bibitem[Mar91]{Mar91}
G.~A. Margulis.
\newblock {\em Discrete subgroups of semisimple {L}ie groups}, volume~17 of
  {\em Ergebnisse der Mathematik und ihrer Grenzgebiete (3)}.
\newblock Springer-Verlag, Berlin, 1991.

\bibitem[Nav98]{Nav98}
G.~Navarro.
\newblock {\em Characters and blocks of finite groups}, volume 250 of {\em
  London Mathematical Society Lecture Note Series}.
\newblock Cambridge University Press, Cambridge, 1998.

\bibitem[NS14]{Nav-Spa14I}
G.~Navarro and B.~Sp{\"a}th.
\newblock On {B}rauer's height zero conjecture.
\newblock {\em J. Eur. Math. Soc. (JEMS)}, 16(4):695--747, 2014.

\bibitem[Ros]{Ros-Reduction_CTC}
D.~Rossi.
\newblock A reduction for {S}p{\"a}th's {C}haracter {T}riple {C}onjecture.
\newblock {\em in preparation}.

\bibitem[Ros22a]{Ros22}
D.~Rossi.
\newblock Character {T}riple {C}onjecture for {$p$}-solvable groups.
\newblock {\em J. Algebra}, 595:165--193, 2022.

\bibitem[Ros22b]{Ros-Clifford_automorphisms_HC}
D.~Rossi.
\newblock Inductive local-global conditions and generalized {H}arish-{C}handra
  theory.
\newblock {\em arXiv:2204.10301}, 2022.

\bibitem[Ros22c]{Ros-iMcK}
D.~Rossi.
\newblock The {M}c{K}ay {C}onjecture and central isomorphic character triples.
\newblock {\em arXiv:2204.10300}, 2022.

\bibitem[Ruh22]{Ruh22AM}
L.~Ruhstorfer.
\newblock The {A}lperin--{M}c{K}ay conjecture for the prime $2$.
\newblock {\em arXiv:2204.06373}, 2022.

\bibitem[Sp{\"a}12]{Spa12}
B.~Sp{\"a}th.
\newblock Inductive {M}c{K}ay condition in defining characteristic.
\newblock {\em Bull. Lond. Math. Soc.}, 44(3):426--438, 2012.

\bibitem[Sp{\"a}13]{Spa13II}
B.~Sp{\"a}th.
\newblock A reduction theorem for the {A}lperin-{M}c{K}ay conjecture.
\newblock {\em J. Reine Angew. Math.}, 680:153--189, 2013.

\bibitem[Sp{\"a}17]{Spa17}
B.~Sp{\"a}th.
\newblock A reduction theorem for {D}ade's projective conjecture.
\newblock {\em J. Eur. Math. Soc. (JEMS)}, 19(4):1071--1126, 2017.

\bibitem[Sp{\"a}21]{Spa21}
B.~Sp{\"a}th.
\newblock Extensions of characters in type {${\rm D}$} and the inductive
  {M}c{K}ay condtion, {I}.
\newblock {\em arXiv:2109.08230}, 2021.

\bibitem[Sym05]{Sym05}
P.~Symonds.
\newblock The {B}redon cohomology of subgroup complexes.
\newblock {\em J. Pure Appl. Algebra}, 199(1-3):261--298, 2005.

\end{thebibliography}

\vspace{1cm}

\mbox{DEPARTMENT OF MATHEMATICS, CITY, UNIVERSITY OF LONDON, EC$1$V $0$HB, UNITED KINGDOM.}

\textit{Email address:} \href{mailto:damiano.rossi@city.ac.uk}{damiano.rossi@city.ac.uk}

\end{document}